\documentclass[11pt]{amsart}

\usepackage{amsmath,amssymb,amscd,amsthm,amsxtra}

\usepackage{graphicx, mathrsfs}
\usepackage{amssymb,amsmath,amsthm}
\usepackage{mathrsfs,dsfont}

\usepackage[]{hyperref}

\allowdisplaybreaks[2]

\newtheorem{theorem}{Theorem} [section]

\newtheorem{lemma}[theorem]{Lemma}
\newtheorem{proposition}[theorem]{Proposition}
\newtheorem{remark}[theorem]{Remark}

\newtheorem{definition}[theorem]{Definition}

\newcommand{\Z}{\mathbb{Z}}
\newcommand{\R}{\mathbb{R}}
\newcommand{\C}{\mathbb{C}}
\newcommand{\T}{\mathbb{T}}
\newcommand{\pr}{\mathbb{P}}

\newcommand{\TT}{\mathcal T}

\newcommand{\FF}{\mathcal F} 
\newcommand{\GG}{\mathcal G}

\newcommand{\cR}{\overline R}

\newcommand{\cu}{\overline u} 
\newcommand{\ca}{ \overline a}
\newcommand{\ch}{ \overline h}

\newcommand{\cg}{ \overline g}
\newcommand{\FS}{\mathcal S}
\newcommand{\D}{\mathcal D}
\newcommand{\RR}{\mathcal R}

\numberwithin{equation}{section}

\begin{document}

\title[A.S. well-posedness for the periodic 3D quintic NLS below $H^1$ ]{Almost sure well-posedness for the periodic 3D quintic nonlinear Schr\"odinger equation below the energy space}
\author[Nahmod]{Andrea R. Nahmod$^\ast$}
\address{$^\ast$ Department of Mathematics \\
University of Massachusetts at Amherst\\ 710 N. Pleasant Street, Amherst MA 01003}
\email{nahmod@math.umass.edu}
\thanks{$^\ast$ The first author is funded in part by NSF DMS 1201443.}
\author[Staffilani]{Gigliola Staffilani$^\dagger$}
\address{$^{\dagger}$ Department of Mathematics\\
Massachusetts Institute of Technology\\ 77 Massachusetts Avenue,  Cambridge, MA 02139}
\email{gigliola@math.mit.edu}
\thanks{$^\dagger$ The second author is funded in part by NSF
DMS 1068815.}
\date{}
\begin{abstract}
In this paper we prove an almost sure local well-posedness result for the periodic 3D quintic nonlinear Schr\"odinger equation in the supercritical regime,  that is below the critical space $H^1(\mathbb T^3)$. 
\end{abstract}
\maketitle

\section{Introduction}
In this paper we continue the study of almost sure well-posedness for certain dispersive equations in a supercritical regime. In the last two decades there has been a burst of activity -and significant progress- in the field of nonlinear dispersive equations and systems. These range from the development of  analytic tools in nonlinear Fourier and harmonic analysis combined with geometric ideas to address nonlinear estimates, to related deep functional analytic methods and profile decompositions to study local and global well-posedness and singularity formation for these equations and systems. The thrust of this body of work has focused mostly on deterministic aspects of wave phenomena where sophisticated tools from nonlinear Fourier analysis, geometry and analytic number theory have played a crucial role in the methods employed.  Building upon work by Bourgain \cite{B1, B2, B4} several works  have appeared in which the well-posedness theory has been studied from a nondeterministic point of view relying on and adapting probabilistic ideas and tools as well (c.f. \cite{BT1, BT2, T0, O, O1, T, NORS, NRSS, Th, CO, NPS, BThT, YD1, YD2} and references therein).

\smallskip
It is by now well understood that randomness plays a fundamental role in a variety of fields.  Situations when such a point of view is desirable include:  when there still remains a gap between local and global well-posedness  
when certain type of ill-posedness is present, and in the very important super-critical regime  when a deterministic well-posedness theory remains, in general, a major open problem in the field.  A set of important and tractable problems is concerned with those (scaling) equations 
for which global well posedness for {\it large data} is known at the {\it critical scaling level}. Of special interest is the case when the scale-invariant regularity $s_c = 1$ (energy or Hamiltonian).  A natural question then is that of understanding the {\it supercritical } (relative to scaling) long time dynamics for the nonlinear Schr\"odinger equation in the defocusing case.  Whether blow up occurs from classical data in the defocussing case remains a difficult open problem in the subject. However, what seems within reach at this time is to investigate and seek an answer to these problems {\it from a nondeterministic viewpoint};  namely for {\it random data}.  
\smallskip

In this paper we treat the energy-critical periodic quintic nonlinear Schr\"odinger equation (NLS), an especially important prototype in view  of   the results by Herr, Tzvetkov and Tataru \cite{HTT} establishing small data global well posedness in $H^1(\mathbb T^3)$ and of  Ionescu and Pausader  \cite{IP} proving {\it large data} global well posedness in $H^1(\mathbb T^3)$ in the defocusing case, the first critical result for NLS on a compact manifold.  Large data global well-posedness in $\mathbb R^3$ for the energy-critical quintic NLS had been previously established by Colliander, Keel, Staffilani, Takaoka and Tao in \cite{CKSTT}.

Our interest in this paper is to establish a local  almost sure well posedness for random data {\it below} $H^1(\mathbb T^3);$ that is,  in the supercritical regime relative to scaling{\footnote{i.e. for Cauchy data in $H^{s}(\T^3),\, s < s_c=1$ for the quintic NLS in $3D$}} and without any kind of symmetry restriction on the data.   In a seminal paper,  Bourgain \cite{B4} obtained almost sure global well posedness for the $2D$ periodic defocusing  (Wick ordered) cubic NLS with data {\it below } $L^2(\mathbb T^2)$, ie. in a supercritical regime ($s_c=0$){\footnote{See Brydges and Slade \cite{BrS} for a study of invariant measures associated to the $2D$  {\it focussing} cubic NLS.}}. Burq and Tzvetkov  obtained similar results for the supercritical ($s_c=\frac{1}{2}$) {\it radial} cubic NLW on compact Riemannian manifolds in $3D$. Both global results rely on the existence and invariance of associated Gibbs measures.
As it turns out, in many situations either the natural Gibbs or weighted Wiener construction does not produce an invariant measure or (and this is particularly so in higher dimensions)  it is thought to be impossible to make any reasonable construction at all. In the case of the ball or the sphere, if one restricts to the radial  case then constructions of invariant measures are possible as in  \cite{T, BT2, dS1, dS2, BB1, BB2, BB3}.  Recently, a probabilistic method based on energy estimates has been used to obtain supercritical almost sure global results, thus circumventing the use of invariant measures and  the restriction of radial symmetry.  In this context  Burq and Tzvetkov \cite{BT3} and  Burq, Thomann and Tzvetkov \cite{BTT}  considered the periodic cubic NLW, while Nahmod, Pavlovic and Staffilani \cite{NPS} treated the periodic Navier-Stokes equations. 
Colliander and Oh  \cite{CO} also proved almost sure global well-posedness  for the {\it subcritical} 1D periodic cubic NLS below $L^2$ in the absence of invariant measures by suitably adapting Bourgain's high-low method.  

\smallskip

Extending the local solutions we obtain here to global ones is the next natural step; it is worth noting however that unlike the work of Bourgain \cite{B4} one would need to proceed in the absence of invariant measures;  and unlike the work of Colliander and Oh \cite{CO}  the smoother norm in our case, namely $H^1(\T^3)$, on which one would need to rest to extend the local theory to a global one is in fact {\it critical}. This forces  the bounds on the Strichartz type norms to be of exponential type with respect to the energy, too large to be able to close the argument.

The problem we are considering here is the analogue of the supercritical local well-posedness result {\footnote{ a.s for data in $H^{-\beta}(\mathbb T^2)$,  $\beta>0$}} proved by Bourgain in \cite{B4} for the periodic mass critical defocusing cubic NLS in 2D. Of course, Bourgain then constructed a 2D Gibbs measure and proved that for data  in its statistical ensemble the local solutions obtained were in fact global, hence establishing almost sure global well posedness in $H^{-\epsilon}(\mathbb T^2)$,  $\epsilon>0$.

There are several  major complications in the work that we present below compared to the work of Bourgain: a quintic nonlinearity increases quite substantially the different cases that need to be treated when one analyzes the frequency interactions in the nonlinearity; the counting lemmata in a 3D lattice are much less favorable and the Wick ordering is not sufficient  to remove certain resonant frequencies that need to be eliminated. The latter is not surprising, and in fact  known within the context of quantum field renormalization (c.f. Salmhofer's book \cite{S:book}).  In particular, to overcome this last obstacle, we introduce an appropriate gauge transformation, we work on the gauged problem and then transfer the obtained result back to the original problem; which as a consequence is studied through a contraction method applied in a certain metric space of functions. A similar approach was used by the second author in \cite{St:gKdV}. Finally our estimates take place in function spaces where we must be careful about working  with the absolute  value of the Fourier transform. In fact the norms of these spaces  are not defined through the absolute value of the Fourier transform, a property  of the $X^{s,b}$ spaces in \cite{B4} which is quite useful,  see for example Section \ref{theproof}. 

\medskip
In this work we consider the Cauchy initial value problem, 
\begin{equation}\label{IVPmain}
\begin{cases}
&iu_t+\Delta u=\lambda u|u|^{4} \quad x\in \T^3\\
&u(0,x)=\phi(x) 
\end{cases}
\end{equation} where $\lambda=\pm 1$. 

We randomize the data in the following sense,

\begin{definition}\label{defrand} Let 
$(g_n(\omega))_{n\in \Z^3}$ be a sequence of complex i.i.d centered Gaussian random variables on a probability space $(\Omega,  A, \mathbb P)$.    For $\phi \in H^s(\T^3)$, let $(b_n)$ be its Fourier coefficients, that is 
\begin{equation}\label{sobolev}
\phi(x) \, =\, \sum_{n\in \Z^3} \, b_n \, e^{i n\cdot x},  \qquad \qquad  \sum_{n\in \Z^3} \, (1 + |n|)^{2s} \, |b_n|^2  < \infty\, \end{equation}
The map from $(\Omega,A)$ to $H^s(\T^3)$ equipped with the Borel sigma algebra, defined by 
\begin{equation}\label{map}
\omega \longrightarrow \phi^\omega, \qquad 
\phi^\omega(x) \, =\, \sum_{n\in \Z^3} \, g_n(\omega) b_n \, e^{i n\cdot x} \end{equation}
is called a map randomization.  
\end{definition}
\begin{remark} The map \eqref{map} is measurable and $\phi^\omega\in L^2(\Omega; H^s(\T^d))$, is an $H^s(\T^d)$-valued random variable. 
We recall that such a randomization does not introduce any $H^s$ regularization (see Lemma B.1 in \cite{BT1} for a proof of this fact), indeed $\|\phi^\omega\|_{H^s}\sim
\|\phi \|_{H^s}$. However  randomization  gives improved $L^p$ 
estimates almost surely. 

\end{remark}

Our setting to show almost sure local well posedness  is similar to that of Bourgain in \cite{B4}. 
More precisely, we consider data $\phi \in H^{1-\alpha -\varepsilon}(\T^3)$ for any $\varepsilon>0$ of the form 
\begin{equation}\label{initialdata}
\phi(x)=\sum_{n\in \Z^3}\frac{1}{\langle n\rangle^{\frac{5}{2}-\alpha}}e^{in\cdot x}
\end{equation}
whose randomization is
\begin{equation}\label{initial}
\phi^\omega(x)=\sum_{n\in \Z^3}\frac{g_n(\omega)}{\langle n\rangle^{\frac{5}{2}-\alpha}}e^{in\cdot x}
\end{equation}
 
Our main result can then be stated as follows,

\begin{theorem}[Main Theorem]\label{lwpIVP} Let $0< \alpha < \frac{1}{12}$,  $s \in (1+ 4\alpha, \, \frac{3}{2} - 2\alpha)$ and $ \phi$ as in \eqref{initialdata}.
Then there exists $0<\delta_0\ll 1$ and $r=r(s,\alpha)>0$ such that for any $\delta < \delta_0,$ there exists $\Omega_\delta\in A$ with 
$$\pr(\Omega_\delta^c)<
e^{-\frac{1}{\delta^{r}}},$$
and for each $\omega\in \Omega_\delta$ there exists a unique solution $u$ of \eqref{IVPmain}
in the space
$$S(t)\phi^\omega+X^s([0,\delta))_d,$$ where  $S(t)\phi^\omega$ is the linear evolution of the initial data $\phi^{\omega}$ given by  \eqref{initial}.
\end{theorem} Here we denoted by $X^s([0,\delta))_d$ the metric space $( X^s([0,\delta)),\, d)$ where $d$ is the metric defined by \eqref{metric} in Section \ref{resonant-gauged} and $X^s([0,\delta))$ is the space introduced  in Definition \ref{x-atomic} below.

\smallskip

\subsection*{Acknowledgement} The authors would like to thank the Radcliffe Institute for Advanced Study at Harvard University for its hospitality during the final stages of this work. They also  thank Luc Rey-Bellet  for several helpful conversations.

\section{Removing resonant frequencies: the gauged equation}\label{resonant-gauged}
The main idea in proving Theorem \ref{lwpIVP} goes back to Bourgain \cite{B4} and it consists on proving that if $u$ solves 
\eqref{IVPmain}, then $w=u-S(t)\phi^\omega$ is smoother; see also \cite{BT1, CO, NPS}. In fact  one reduces the problem to showing well-posedness for the  initial value problem involving  $w$, which is in fact treated as a deterministic function. The initial value problem that $w$ solves does not become a subcritical one, but it is of a  hybrid type involving also rougher but random terms, whose   decay and moments play a fundamental role.
For the NLS equation  this argument can be carried out  only after having removed 
certain resonant frequencies  in the nonlinear part of the equation. In this section in fact we write the Fourier coefficients of the quintic expression $|u|^4u$ and we identify the resonant part that needs to be removed  in order to be able to take advantage of the moments  coming from the randomized terms. We will go back to this concept more in details in Remark \ref{removing} below.

\medskip

Let's start by assuming that $\widehat u(n)(t)=a_n(t)$. We introduce the notation 
\begin{equation}\label{plane}
\Gamma(n)_{[i_1,i_2,\cdots,i_r]}:=\{(n_{i_1},\cdots,n_{i_r})\in \Z^{3r}\, \, /\, \, n=n_{i_1}-n_{i_2}+\cdots +(-1)^{r+1}n_{i_r}\}
\end{equation} to indicate various hyperplanes  
and $\Gamma(n)_{[i_1,i_2,\cdots,i_r]}^{c}$ is  its complement. 

\medskip

Next, for fixed time $t$, we take $\FF$, the Fourier transform in space,  and 
write,
\begin{eqnarray*}
\FF(|u(t)|^4u(t))(n)&=&\sum_{\Gamma(n)_{[1,\cdots,5]}} a_{n_1}(t)\overline{a_{n_2}}(t) a_{n_3}(t)\overline{a_{n_4}}(t)a_{n_5}(t)\\
&=&\sum_{\Gamma(n)_{[1,\cdots,5]}\cap\Gamma(0)_{[1,2,3,4]}^{C}\cap\Gamma(0)_{[1,2,5,4]}^{C}\cap\Gamma(0)_{[3,2,5,4]}^{C}}
a_{n_1}(t)\overline{a_{n_2}}(t) a_{n_3}(t)\overline{a_{n_4}}(t)a_{n_5}(t)\\
&+&\sum_{\Gamma(n)_{[1,\cdots,5]}\cap\Gamma(0)_{[1,2,3,4]}}a_{n_1}(t)\overline{a_{n_2}}(t) a_{n_3}(t)\overline{a_{n_4}}(t)a_{n_5}(t)\\ 
&+&\sum_{\Gamma(n)_{[1,\cdots,5]}\cap\Gamma(0)_{[1,2,5,4]}}a_{n_1}(t)\overline{a_{n_2}}(t) a_{n_3}(t)\overline{a_{n_4}}(t)a_{n_5}(t)\\
&+&\sum_{\Gamma(n)_{[1,\cdots,5]}\cap\Gamma(0)_{[3,2,5,4]}}a_{n_1}(t)\overline{a_{n_2}}(t) a_{n_3}(t)\overline{a_{n_4}}(t)a_{n_5}(t)\\
&-&\sum_{\Gamma(n)_{[1,\cdots,5]}\cap\Gamma(0)_{[1,2,3,4]}\cap\Gamma(0)_{[1,2,5,4]}\cap\Gamma(0)_{[3,2,5,4]}^{C}}a_{n_1}(t)\overline{a_{n_2}}(t) a_{n_3}(t)\overline{a_{n_4}}(t)a_{n_5}(t)\\
&-&\sum_{\Gamma(n)_{[1,\cdots,5]}\cap\Gamma(0)_{[1,2,3,4]}\cap\Gamma(0)_{[3,2,5,4]}\cap\Gamma(0)_{[1,2,5,4]}^{C}}a_{n_1}(t)\overline{a_{n_2}}(t) a_{n_3}(t)\overline{a_{n_4}}(t)a_{n_5}(t)\\
&-&\sum_{\Gamma(n)_{[1,\cdots,5]}\cap\Gamma(0)_{[3,2,5,4]}\cap\Gamma(0)_{[1,2,5,4]}\cap\Gamma(0)_{[1,2,3,4]}^{C}}a_{n_1}(t)\overline{a_{n_2}}(t) a_{n_3}(t)\overline{a_{n_4}}(t)a_{n_5}(t)\\
&-&2\sum_{\Gamma(n)_{[1,\cdots,5]}\cap\Gamma(0)_{[1,2,3,4]}\cap\Gamma(0)_{[3,2,5,4]}\cap\Gamma(0)_{[1,2,5,4]}}a_{n_1}(t)\overline{a_{n_2}}(t) a_{n_3}(t)\overline{a_{n_4}}(t)a_{n_5}(t)\\&=&\sum_{k=1,\cdots,8}I_k,  
\end{eqnarray*}
We now rewrite each $I_k$ using more explicitly the constraints in the hyperplanes.  $I_1$ is the most complicated and 
and we start by rewriting it. To that effect we introduce the following notation:
\begin{eqnarray}
\label{lambdan}\Lambda(n)&:=&\Gamma(n)_{[1,\cdots,5]}\cap\Gamma(0)_{[1,2,3,4]}^{C}\cap\Gamma(0)_{[1,2,5,4]}^{C}\cap\Gamma(0)_{[3,2,5,4]}^{C}\\
\label{sigman}\Sigma(n)&:=&\{(n_1,n_2,n_3,n_4,n_5)\in \Lambda(n) \, \, / \, \, n_1,n_3,n_5\ne n_2,n_4\}.
\end{eqnarray}
We have 
\begin{eqnarray}\label{I1}
 I_1&=& \sum_{\Lambda(n)}a_{n_1}(t)\overline{a_{n_2}}(t) a_{n_3}(t)\overline{a_{n_4}}(t)a_{n_5}(t)\\\notag
&=&\sum_{\Sigma(n)}a_{n_1}(t)\overline{a_{n_2}}(t) a_{n_3}(t)\overline{a_{n_4}}(t)a_{n_5}(t)\\\notag &&
+
6\left(\sum_{n_2}|a_{n_2}|^2\right)\sum_{\Gamma(n)_{[3,4,5]}, \, n_3,n_5\ne n_4}a_{n_3}(t)\overline{a_{n_4}}(t)a_{n_5}(t)\\\notag
&&-6|a_n|^2\sum_{\Gamma(n)_{[3,4,5]}, \, n_3,n_5\ne n_4}a_{n_3}(t)\overline{a_{n_4}}(t)a_{n_5}(t)\\\notag&&
-3\sum_{\Gamma(n)_{[3,1,5]}, n_3, n_5\ne n_1}|a_{n_1}(t)|^2\overline{a_{n_1}}(t)a_{n_3}(t)a_{n_5}(t)\\\notag
&&-3|a_n|^4a_n(t)+3|a_n|^2\overline{a_{n}(t)}\sum_{n_3+n_5=2n}a_{n_3}(t)a_{n_5}(t)\\\notag&&
-6\sum_{\Gamma(n)_{[2,4,5]}, n_2,n_5\ne n_4}|a_{n_2}(t)|^2a_{n_2}(t)\overline{a_{n_4}}(t)a_{n_5}(t)\\\notag
&&+2\sum_{n=2n_2-n_4, n_2\ne n_4}|a_{n_2}(t)|^2a_{n_2}^2(t)\overline{a_{n_4}}(t)
\end{eqnarray}
Note here that we can write
\begin{eqnarray}\label{removecont}
&&|a_n|^2\sum_{\Gamma(n)_{[3,4,5]}, \, n_3,n_5\ne n_4}a_{n_3}(t)\overline{a_{n_4}}(t)a_{n_5}(t)=
-2|a_n|^2a_n\left(\sum_{n_2}|a_{n_2}|^2\right)+|a_n|^4a_n\\\notag
&&+|a_n|^2\sum_{\Gamma(n)_{[3,4,5]}}a_{n_3}(t)\overline{a_{n_4}}(t)a_{n_5}(t).
\end{eqnarray}

It is easier to see that for $i=2,3,4$
\begin{equation}\label{I234}
I_{i}=a_n(t)\sum_{\Gamma(0)_{[1,2,3,4]}}a_{n_1}(t)\overline{a_{n_2}}(t) a_{n_3}(t)\overline{a_{n_4}}(t)=\hat u(n)(t)\int_{\T^3}|u|^4(x,t)dx,
\end{equation}
while  for $j=5,6,7$
\begin{equation}\label{I567}
I_{j}=-a_n^3(t)\sum_{n_2+n_4= 2n}\overline{a_{n_2}}(t)\overline{a_{n_4}}(t)+
a_n^2\sum_{n=n_2+n_4-n_1}\overline{a_{n_2}}(t)\overline{a_{n_4}}(t)a_{n_1}(t)
\end{equation}
and   for
\begin{equation}\label{I8}
I_{8}=-2a_n^3(t)\sum_{n_2+n_4= 2n}\overline{a_{n_2}}(t)\overline{a_{n_4}}(t).
\end{equation}
We summarize our findings from \eqref{I1}- \eqref{I8}. In this part of the argument the time variable is not important, hence we will omit it for now.  We write 
\begin{equation}\label{rewrite}
\FF\left(|u|^4u-3u\left(\int_{\T^3}|u|^4\,dx\right)\right)(n)=\sum_{k=1}^7J_k(a_n)
\end{equation}
with
\begin{eqnarray}\label{j1}
J_1&=&\sum_{\Sigma(n)}a_{n_1}\overline{a_{n_2}} a_{n_3}\overline{a_{n_4}}a_{n_5}
\\\label{j2}
J_2&=&6m\sum_{\Gamma(n)_{[1,2,3]}, \, n_3,\, n_1\ne n_2}a_{n_1}\overline{a_{n_2}}a_{n_3}\\\label{j3}
J_3&=&-6\sum_{\Gamma(n)_{[1,2,3]}, \, n_1, n_3\, \ne  n_2}|a_{n_1}|^2a_{n_1}\overline{a_{n_2}}a_{n_3}\\\notag&&-3\sum_{\Gamma(n)_{[1,2,3]}, \, n_1,\, n_3,\ne n_2}a_{n_1}|a_{n_2}|^2\overline{a_{n_2}}a_{n_3}\\\label{j4}
J_4&=&2\sum_{n=2n_1-n_2, }|a_{n_1}|^2a_{n_1}^2\overline{a_{n_2}}\\\label{j5}
J_5&=&-6|a_n|^2\sum_{\Gamma(n)_{[123]}
}a_{n_1}\overline{a_{n_2}}a_{n_3}
+3a_n^2\sum_{\Gamma(n)_{[214]}}\overline{a_{n_2}}a_{n_1}\overline{a_{n_4}}
\\\label{j6}
J_6&=&=-5a_n^3\sum_{n=n_2+n_4}\overline{a_{n_2}}\overline{a_{n_4}}+
3|a_n|^2\overline{a_n}\sum_{n=n_1+n_3}a_{n_1}a_{n_3}
\\\label{j7}
J_7&=&-11a_n|a_n|^4 +12m|a_n|^2a_n, 
\end{eqnarray}
where  $m=\int_{\T^3}|u(t,x)|^2dx,$ the conserved mass.
\begin{remark}\label{removing}
In the calculations above we wrote the nonlinear terms in \eqref{IVPmain}  in Fourier space, we isolated the term $u\int_{\T^3}|u|^4\, dx$ and we subtracted it from $|u|^4u$, see \eqref{rewrite}.  We show below that indeed in doing so we  separated those terms that we claim are not suitable  for smoother  estimates from the ones that are.  To understand this point  let us replace 
$a_n=\frac{g_n(\omega)}{\langle n\rangle^{\frac{5}{2}-\alpha}}$, for $ \alpha $ small, whose anti-Fourier transform barely misses to be in $H^1({\T^3})$. We want to claim that the randomness coming from $\{g_n(\omega)\}$ will  increase the 
 regularity of the nonlinearity in a certain sense, so that it can hold a bit more than one derivative.
 We realize immediately though that this claim cannot be true for the whole nonlinear term. For example the terms $I_i, \, i=2,3,4$ have no chance to improve their regularity because they are simply linear with respect to $a_n$, hence they need to be removed.  This same problem presented itself  in the work of Bourgain \cite{B4} and Colliander-Oh \cite{CO} who considered the cubic NLS below $L^2$. In particular in their case the problematic  term was of the type $a_n\int_{\T^d}|u|^2dx$ and the authors removed it by Wick ordering  the Hamiltonian. An important ingredient in making this successful was that $\int_{\T^d}|u|^2dx$, that is the mass, is independent of time. In our case Wick ordering the  Hamiltonian  is not helpful  since it does not  remove the terms $I_i, \, i=2,3,4$. As we mentioned before, the latter is not surprising, and in fact known within the context of quantum field renormalization (c.f. Salmhofer's book \cite{S:book}).

 If we knew that $\int_{\T^3}|u|^4dx$ were constant in time, then we could simply relegate those terms to the linear part of the equation. Since  this is obviously not the case   relegating these expressions with the main linear part of the equation  would prevent us from using the simple form of the solution for a Schr\"odinger equation with constant coefficients. A similar situation to the one  just described presented itself in \cite{St:gKdV} where a gauge transformation was used to remove the time dependent linear terms. We are able to use the same idea in this context and this is the content of what follows in this section.
\end{remark}

To prove  Main Theorem \ref{IVPmain} we  proceed in two steps. First we consider the initial value problem 
\begin{equation}\label{IVPremove}
\begin{cases}
&iv_t+\Delta v=\mathcal N(v) \quad x\in \T^3\\
&v(0,x)=\phi(x),
\end{cases}
\end{equation}
where
\begin{equation}\label{Nv}\mathcal N(v):=\lambda \left(v|v|^{4}-3v\left(\int_{\T^3}|v|^4\,dx\right)\right)
\end{equation}
 with $\lambda=\pm 1$ and $\phi(x)$ the initial datum  as in \eqref{IVPmain}.
To make the notation simpler  set
\begin{equation}\label{beta}
\beta_v(t) = 3\int_{\T^3}|v|^4\,dx
\end{equation}
and  define
\begin{equation}\label{fromvtou}
u(t,x):=e^{i\lambda\int_0^t\beta_v(s)\,ds}v(t,x).
\end{equation}
We observe that  
 $u$ solves the initial value problem \eqref{IVPmain}.
Now suppose that one obtains  well-posedness for the initial value problem \eqref{IVPremove} in a certain Banach space 
$(X,\|\cdot\|)$ then one can transfer those results to the initial value problem \eqref{IVPmain} by using a metric space $X_d:= (X,d)$ where
\begin{equation} \label{metric} d(u,v):=\|e^{-i\lambda\int_0^t\beta_u(s)\,ds}u(t,x)-e^{-i\lambda\int_0^t\beta_v(s)\,ds}v(t,x)\|.\end{equation}

The fact  that this is indeed a metric follows from  using the properties of the norm $\|\cdot\|$ and the fact that if 
$$e^{-i\lambda\int_0^t\beta_u(s)\,ds}u(t,x)=e^{-i\lambda\int_0^t\beta_v(s)\,ds}v(t,x)$$
then $\beta_v(t)=\beta_u(t)$ and hence $u=v$.

\medskip
From this moment on we work exclusively with the initial value problem \eqref{IVPremove}. In particular below we  prove the following result:

\begin{theorem}\label{lwpIVPremove} Let $0< \alpha < \frac{1}{12}$,  $s \in (1+ 4\alpha, \, \frac{3}{2} - 2\alpha)$ and $ \phi$ as in \eqref{initialdata}.
There exists $0<\delta_0\ll 1$ and $r=r(s,\alpha)>0$ such that for any $\delta < \delta_0,$ there exists $\Omega_\delta\in A$ with 
$$\pr(\Omega_\delta^c)<
e^{-\frac{1}{\delta^{r}}},$$
and for each $\omega\in \Omega_\delta$ there exists a unique solution $u$ of \eqref{IVPremove}
in the space
$$S(t) \phi^\omega+X^s([0,\delta)),$$  with initial condition $\phi^{\omega}$ given by  \eqref{initial}. 
\end{theorem} Here in the space $X^s([0,\delta))$ is defined in Section \ref{Spaces}.

\medskip
Thanks to  the transformation \eqref{fromvtou}, Theorem \ref{lwpIVPremove}  translates to  Main Theorem \ref{lwpIVP}.

\section{Probabilistic Set Up}

We first recall a classical result that goes back to Kolmogorov, Paley and Zygmund. 

\begin{lemma}[Lemma 3.1 \cite{BT1}]\label{gLp} Let $\{g_n(\omega)\}$ be a sequence of complex i.i.d. zero mean Gaussian random variables on a probability space $(\Omega,  A, \mathbb P)$ and $(c_n) \in \ell^2$. 
  Define \begin{equation} \label{allrandom} F(\omega) := \sum_{ n } \, c_{n}   g_{n}(\omega) \end{equation} Then, there exists $C >0$ such that for every $\lambda >0$ we have 
\begin{equation} \label{allrandombound} \mathbb P ( \{ \omega \, :\, | F(\omega) | > \lambda \, \} ) \, \leq \exp{\left( \frac{ - C\, \lambda^{2}}{ {\| F(\omega) \|^2}_{L^2(\Omega)}} \right)}. \end{equation} As a consequence there exists $C>0$ such that for every $q \geq 2$ and every $(c_n)_{n}  \in \ell^2$, 
\begin{equation*}
\left\| \sum_{n}  c_n g_n(\omega) \right\|_{L^q(\Omega)} \leq C \sqrt{q} \left( \sum_{n} \, c_n^2 \right)^{\frac{1}{2}}.
\end{equation*}  
\end{lemma}
We also recall the following basic probability results:

\begin{lemma}\label{preindependence}  Let $1\leq m_1 < m_2 \dots < m_k= m $;  $f_1$ be a Borel measurable function of $m_1$ variables, $f_2$  one of $m_2-m_1$ variables, $\dots, f_k$ one of $m_k - m_{k-1}$ variables. If  $\{X_1, X_2, \dots X_m\}$ are real-valued independent  random variables, then  the $k$ random variables $f_1(X_1, \dots X_{m_k}), f_2(X_{m_1+1}, \dots, X_{m_2}), \dots, f_k(X_{m_{k-1}+1}, \dots X_{m_k})$  are independent random variables. 
\end{lemma}

\begin{lemma}\label{independence}  Let  $k\ge 1$ and  consider $\{g_{n_j}\}_{1\leq j\leq k}$ and $\{g_{n'_j}\}_{1\leq j\leq k} \in \mathcal N_{\C}(0,1)$ complex $L^2(\Omega)$-normalized independent Gaussian random variables  such that \, $n_i \neq n_j$ and  $n'_i \neq n'_j$ for $i \neq j$.
$$\left| \int_\Omega \prod_{j=1}^k g_{n_j}(\omega)\prod_{i=1}^k \cg_{n'_i}(\omega)\, dp(\omega) \right| \, \leq \, \int_\omega \prod_{\ell=1}^k |g_{n_{\ell}}(\omega)|^2dp(\omega).$$
\end{lemma}
\begin{proof}
For every pair $(n_{\ell}, n'_i)$  such that $n_{\ell}=n'_i$ we write $K_{n_j}(\omega):=|g_{n_j}(\omega)|^2$  and note that thanks to the independence and normalization of  $\{g_{n_j}\}$, for $n_j\ne n_i$,   we have that $ \mathbb E(K_{n_j}g_{n_i})=0$. The latter together with Lemma \ref{preindependence} give the desired conclusion.
\end{proof}

More generally, in the next sections we will repeatedly use a classical Fernique or large deviation-type result related to the product of $\{G_n\}_{1\leq n\leq d} \in \mathcal N_{\C}(0,1)$,  complex $L^2$ normalized independent Gaussians. This result follows from the hyper-contractivity property of the Ornstein-Uhlenbeck semigroup (c.f. \cite{T, TT} for a nice exposition) by writing $G_n = H_n + i L_n$ where $\{H_1, \dots, H_d, L_1, \dots L_d \} \in \mathcal N_{\mathbb R}(0,1)$ are real centered independent Gaussian random variables with the same variance. Note that $\mathbb E( G_n^2)=\mathbb E( G_n)=0$.  

\begin{remark}\label{plusone} Note that for $\{G_{n}(\omega)\}_n \in \mathcal N_{\C}(0,1)$,  complex $L^2$ normalized independent Gaussians, if we write $|G_{n}(\omega)|^2=(|G_{n}(\omega)|^2-1)+1$, then thanks to the independence and normalization of $G_{n}$, \,  $Y_{n}(\omega):= |G_{n}(\omega)|^2-1$ is a centered real Gaussian random variable such that  for $n \ne n^{\prime}$, 
$\, \mathbb E(Y_{n} G_{n^{\prime}})=0 =   \mathbb E(Y_{n} Y_{n^{\prime}}) $. 
\end{remark}

\begin{proposition}[Propositions 2.4 in \cite{TT} and Lemma 4.5 in \cite{T}]\label{chaos} Let $d \ge 1$ and $c(n_1, \dots, n_k) \in \C$. Let $\{G_n\}_{1\leq n\leq d} \in \mathcal N_{\C}(0,1)$ be  complex centered $L^2$ normalized independent Gaussians.  For $k\ge 1$ denote by $A(k, d):= \{ (n_1, \dots, n_k) \in \{1, \dots, d\}^k, \, n_1 \leq \dots \leq n_k \}$ and 
\begin{equation}\label{form} 
F_k(\omega) = \sum_{A(k, d)}  c(n_1, \dots, n_k) G_{n_1}(\omega) \dots G_{n_k}(\omega). 
\end{equation}
Then for all $d \ge 1$ and $p \ge 2$ 
$$ \Vert F_k \Vert_{L^p(\Omega)} \lesssim \sqrt{k+1} (p -1)^{\frac{k}{2}}  \Vert F_k \Vert_{L^2(\Omega)}. $$ As a consequence from Chebyshev's inequality we have that for every $\lambda>0$ 
\begin{equation}
\label{largedeviation}  \mathbb P ( \{ \omega \, :\, | F_k(\omega) | > \lambda \, \} ) \, \lesssim \exp{\left( \frac{ - C\, \lambda^{\frac{2}{k}}}{ \| F(\omega) \|^{\frac{2}{k}}_{L^2(\Omega)}} \right)}.
\end{equation}
\end{proposition}

\begin{remark} \label{examples} In Sections \ref{blocks} and \ref{theproof}  we will rely repeatedly on Proposition \ref{chaos}, particularly \eqref{largedeviation}, as well as Lemma \ref{gLp}, and \eqref{allrandombound}.  Indeed, in proving our estimates we will encounter expressions of the following type: 

\smallskip

\noindent Let $\Sigma := \{ (n_1, \dots n_{r},  \ell_1, \dots, \ell_s)  \, : \, |n_j| \sim N_j, \, |\ell_i| \sim L_i,  \,  n_j \neq \ell_i, \, 1 \leq j\leq r, \, \,  1\leq i \leq s, \}$ and  \begin{equation*}  F(\omega) := \sum_{ (n_1, \dots, n_r, \ell_1, \dots, \ell_s) \in \Sigma} \, c_{n_1} \dots c_{n_r} b_{\ell_1} \dots b_{\ell_s}  \,  g_{n_1}(\omega) \dots g_{n_r}(\omega) \cg_{\ell_1}(\omega) \dots \cg_{\ell_s}(\omega),\end{equation*} where $\{ g_{n_1}(\omega) \dots g_{n_r}, g_{\ell_1}(\omega) \dots g_{\ell_s}(\omega) \} \in \mathcal N_{\C}(0,1)$ are complex centered $L^2$ normalized independent Gaussians.
Then by  Proposition \ref{chaos},  there exist $C>0,  \gamma = \gamma(r, s)  >0$ such that for every $\lambda >0$ we have 
\begin{equation*}  \mathbb P ( \{ \omega \, :\, | F(\omega) | > \lambda \, \} ) \, \leq \exp{\left( \frac{ - C\, \lambda^{\frac{2}{\gamma}}}{ \| F(\omega) \|^{\frac{2}{\gamma}}_{L^2(\Omega)}} \right)}. \end{equation*} We will also apply Proposition \ref{chaos} in the context of Remark \ref{plusone}.

\end{remark}

\begin{lemma}\label{supg} Let $\{g_n(\omega)\}$ be a sequence of complex i.i.d zero mean Gaussian random variables on a probability space $(\Omega,  A, \mathbb P)$. Then, 
\begin{enumerate}
\item \, For $1 \leq p < \infty$ there exists $c_p >0$ (independent of $n$) such that $  \| g_n\|_{L^p(\Omega)} \leq c_p$.
\item \, Given $\varepsilon, \delta >0$, for   $N$ large and $\omega$ outside of a set of measure $\delta$, 
\begin{equation} \label{sup-bound} \sup_{|n| \geq  N}\,  |g_n (\omega)| \leq  N^\varepsilon. \end{equation}
\item \,  Given $\varepsilon, \delta >0$ and $\omega$ outside of a set of measure $\delta$, 
\begin{equation} \label{sup-bound-all}  |g_n (\omega)| \lesssim  \langle n\rangle^{\varepsilon} \end{equation}
\end{enumerate}
\end{lemma}

\begin{proof}  Part (1) follows from the fact that higher moments of $\{g_n(\omega)\}$ are uniformly bounded. 

For part (2)  first recall that if $\{X_j(\omega)\}_{j \geq 1}$ is a sequence of i.i.d random variables such that $\mathbb E (|X_j|)= E < \infty$ then 
\begin{equation}\label{identical} \mathbb P(|X_j| \geq j ) = \mathbb P(|X_1| \geq j )
\end{equation} and  $$\sum_j \mathbb P(|X_j| \geq j ) = \sum_j  \mathbb P(|X_1| \geq j ) \leq \mathbb E (|X_1|) < \infty. $$ By Borel-Cantelli
$\mathbb P(|X_j| \geq j  \, \mbox{ for infinitely many } j) =0$ whence one can show that $\lim_{j \to \infty} \dfrac{|X_j(\omega)|}{j}  = 0$ almost surely in $\omega$. Egoroff's Theorem then ensures that given $\delta >0$ 
$$\lim_{j \to \infty} \dfrac{|X_j(\omega)|}{j}  = 0$$ uniformly outside a set of measure $\delta$.  Thus  we have that for $j_0$ sufficiently large 
$$  \frac{|X_j(\omega)|}{j}  \leq 1  \qquad j \geq j_0, $$ for $\omega$ outside an exceptional set of $\delta$ measure. 

If $\{g_n(\omega)\}$ are a sequence of i.i.d. complex Gaussian random variables then given $\varepsilon >0$, if we choose $r = \dfrac{1}{\varepsilon}$ then $\mathbb E( |g_n|^r ) < \infty$. 
By applying the argument above with $X_n(\omega) = |g_n (\omega)|^r$ we have the desired conclusion (cf. \cite{O, CO})

For part (3)  fix $M\gg 1$ such that (2) holds for any $|n|\geq M$. By \eqref{identical} 
$$\mathbb P(|g_n(\omega)| \geq M^\varepsilon) = \mathbb P(|g_M(\omega)| \geq M^\varepsilon)$$
for all $|n|\leq M$. Let $\mathcal A:=\cup_{|n|\leq M-1} \{\omega\, \, /\, \, |g_M(\omega)| \geq M^\varepsilon\}$, then by part (2) 
$\mathbb P(\mathcal A)\leq C_M\delta$. Hence by choosing a smaller $\delta$ in part (2) we have the desired result.
\end{proof}




\section{Function Spaces}\label{Spaces}

For the purpose of establishing our almost sure local well-posedness result, it suffices to work with $X^s$ and $Y^s$, the atomic function 
spaces used by Herr, Tataru and Tzvetkov \cite{HTT}. It is worth emphasizing that while working with these spaces, one should not rely on the notion of the norms depending on the absolute value of the Fourier transform, a feature that is quite useful when working within the context of $X^{s,b}$ spaces.   

In this section we recall their definition and summarize the main properties by following the presentation in \cite{HTT} Section 2.  In what follows, $\mathcal H$ is a separable Hilbert space on $\C$ and $\mathcal Z$ denotes the set of finite partitions $-\infty< t_0 < t_1 < \dots t_K \leq \infty$ of the real line; with the convention that if $t_k=\infty$ then $v(t_K):=0$ for any function $v: \mathbb R \to \mathcal H$. 

\begin{definition}[Definition 2.1 in \cite{HTT}] Let $1 \leq p < \infty$. For $\{ t_k\}_{k=0}^K \in \mathcal Z$ and $\{\phi_k\}_{k=0}^{K-1} \subset \mathcal H$ with $\sum_{k=0}^{K-1} \| \phi_k \|_{\mathcal H}^p =1$. A $U^p$-atom is a piecewise defined function $a:\mathbb R \to \mathcal H$ of the form $$a = \sum_{k=1}^K \chi_{[t_{k-1},\, t_k)} \, \phi_{k-1}. $$ The atomic Banach space $U^p(\mathbb R, \mathcal H)$ is then defined to be the set of all functions $u: \mathbb R \to \mathcal H$ such that $$u = \sum_{j=1}^{\infty} \lambda_j a_j,  \quad \mbox{ for } U^p\text{atoms } \, a_j, \quad \{\lambda_j\}_j \in \ell^1, $$  with the norm 
$$ \|u\|_{U^p}: =\, \inf \left\{\sum_{j=1}^{\infty} | \lambda_j| \, : \, u = \sum_{j=1}^{\infty} \lambda_j a_j,  \quad \lambda_j \in \mathbb C, \,  \text{ and }  \, a_j \, \mbox{ an } U^p \text{atom} \, \right\}$$
\end{definition}
\noindent Here $\chi_I$ denotes the indicator function over the set $I$.  Note that for $1 \leq p \leq q < \infty$, \, 
\begin{equation}\label{embed2} 
U^p(\mathbb R, \mathcal H) \hookrightarrow U^q(\mathbb R, \mathcal H)  \hookrightarrow L^{\infty}(\mathbb R, \mathcal H), 
\end{equation}  and functions in $U^p(\mathbb R, \mathcal H) $ are right continuous, $\lim_{t \to -\infty} u(t) = 0$.

\begin{definition}[Definition 2.2 in \cite{HTT}] Let $1 \leq p < \infty$. The Banach space $V^p(\mathbb R, \mathcal H)$ is  defined to be the set of all functions $v: \mathbb R \to \mathcal H$ such that $$ \|v\|_{V^p}: =\, \sup_{\, \{t_k\}_{k=0}^K \in \mathcal Z\, } \left( \, \sum_{k=1}^{K}  \| v(t_k) - v(t_{k-1}) \|_{\mathcal H}^p \right)^{\frac{1}{p}} \quad \text{is finite. }$$ 
The Banach subspace of all right continuous functions $v:\mathbb R \to \mathcal H$ such that $\lim_{t \to -\infty} v(t) =0$, endowed with the same norm as above is denoted by $V_{rc}^p(\mathbb R, \mathcal H)$. Note that 
\begin{equation}\label{embed1}
U^p(\mathbb R, \mathcal H) \hookrightarrow V_{rc}^p(\mathbb R, \mathcal H) \hookrightarrow L^{\infty}(\mathbb R, \mathcal H), \end{equation}

\end{definition}

\medskip

\begin{definition}[Definition 2.5 in \cite{HTT}] For $s \in \mathbb R$ we let $U^p_{\Delta} H^s$ - respectively $V^p_{\Delta} H^s$- be the space of all functions $u: \mathbb R \to H^s(\mathbb T^3)$ such that $t \to e^{-it \Delta} u(t)$ is in $U^p(\mathbb R, H^s)$ -respectively in $V^p_{\Delta} H^s$- with norm $$ \| u \|_{U^p_{\Delta} H^s} := \| e^{-it \Delta} u(t)\|_{U^p(\mathbb R, H^s)} \qquad \| u \|_{V^p_{\Delta} H^s}:=  \| e^{-it \Delta} u(t)\|_{V^p(\mathbb R, H^s)}. $$

\end{definition}
\medskip 

\noindent We will take $\mathcal H$ to be $H^s$.  We refer the reader to \cite{HHK},  \cite{HTT},  and references therein for detailed definitions and properties of the $U^p$ and $V^p$ spaces.

\medskip 

\begin{definition}[Definition 2.6 in \cite{HTT}] \label{x-atomic} For $s \in \mathbb R$ we define the space $X^s$ as the space of all functions $u: \mathbb R \longrightarrow H^s(\mathbb T^3)$ such that for every $n \in \mathbb Z^3$ the map $t \longrightarrow e^{it |n|^2} \widehat{u(t)}(n)$ is in $U^2(\mathbb R, \mathbb C)$, and for which the norm
\begin{equation}\label{xnorm} 
\| u\|_{X^s}\, := \, \left( \, \sum_{n \in \mathbb Z^3} \, \langle n \rangle^{2s} \, \| e^{it |n|^2 } \, \widehat{u(t)}(n)\|^2_{U^2_t} \, \right)^{\frac{1}{2}} \quad \text{is finite. }
\end{equation}
\end{definition}

The $X^s$ spaces are variations of the spaces
$U^p_{\Delta} H^s$ and  $V^p_{\Delta} H^s$ corresponding to the Schr\"odinger flow and defined as follows:

\begin{definition}[Definition 2.7 in \cite{HTT}] \label{y-atomic} For $s \in \mathbb R$ we define the space $Y^s$ as the space of all functions $u: \mathbb R \longrightarrow H^s(\mathbb T^3)$ such that for every $n \in \mathbb Z^3$ the map $t \longrightarrow e^{it |n|^2} \widehat{u(t)}(n)$ is in $V_{rc}^2(\mathbb R, \mathbb C)$, and for which the norm
\begin{equation}\label{ynorm} 
\| u\|_{Y^s}\, := \, \left( \, \sum_{n \in \mathbb Z^3} \, \langle n \rangle^{2s} \, \| e^{it |n|^2 } \, \widehat{u(t)}(n)\|^2_{V^2_t} \, \right)^{\frac{1}{2}} \quad \text{is finite. }
\end{equation}
\end{definition}
\noindent Note that 
\begin{equation}\label{embed3}
 U^2_{\Delta} H^s \hookrightarrow X^s \hookrightarrow Y^s   \hookrightarrow V^2_{\Delta} H^s  
 \end{equation} whence one has that for any partition of $\mathbb Z^3: =\cup_k C_k$,  
$$ \left( \sum_k \| P_{C_k} u \|_{V^2_{\Delta} H^s}^2 \right)^{\frac{1}{2}} \, \lesssim \| u \|_{Y^s} $$
(cf. Section 2 in \cite{HTT}).  

Additionally, when $s=0$ by orthogonality we have 
\begin{equation}\label{sum-cubes}
\left( \sum_k \| P_{C_k} u \|_{Y^0}^2 \right)^{\frac{1}{2}} \, = \| u \|_{Y^0}.
\end{equation}
We also have the embedding 
\begin{equation}\label{embed4}
X^s \hookrightarrow Y^s \hookrightarrow L^\infty_tH^s_x
\end{equation}
for $s\geq 0$ \, (c.f. \cite{IP}).
\medskip 

\begin{remark} [Proposition 2.10 in \cite{HTT}]\label{linearsol} From the atomic structure of the $U^2$ spaces  one can immediately see that for $s \geq 0, \, T>0$ and $\phi \in H^s(\mathbb T^3)$, the solution to the linear Schr\"odinger equation $u:= e^{it \Delta} \phi$ belongs to $X^s([0, T))$ and $\| u\|_{X^s([0, T))} \leq \|\phi\|_{H^s}$. 
\end{remark}

\smallskip 

\begin{remark} Another important feature of the atomic structure of the $U^2$ spaces is the fact that just like  the $X^{s,b}$ spaces  they enjoy a `transfer principle'. We recall in our context the precise statement below for completeness.\end{remark}

\begin{proposition}[Proposition 2.19 in \cite{HHK}]\label{transfer}
Let $T_0:L^2\times\dots \times L^2\rightarrow L^1_{loc}$ be a m-linear operator. Assume that for some $1\leq p, \, q\leq \infty $
\begin{equation}
\label{transf1}
\|T_0(e^{it\Delta}\phi_1,\dots,e^{it\Delta}\phi_m)\|_{L^p(\R, L_x^q(\T^3))}\lesssim \prod_{i=1}^m\|\phi_i\|_{L^2(\T^3)}.
\end{equation}
Then, there exists an extension $T:U^p_\Delta\times\dots \times U^p_\Delta\rightarrow L^p(\R, \, L^q(\T^3))$ satisfying 
\begin{equation}\label{transf2} 
\|T (u_1,\dots ,u_m)\|_{L^p(\R, \, L_x^q(\T^3))}\lesssim \prod_{i=1}^m\|u_i\|_{U_{\Delta}^p}; 
\end{equation}
and such that $T(u_1, \dots, u_m)(t, \cdot) = T_0(u_1(t), \dots, u_m(t))(\cdot), \, a.e.\,$
In other words, one can reduce estimates for multilinear operators on functions in $U_{\Delta}^p$ to similar estimates on solutions to the linear Schr\"odinger equation. 
\end{proposition}
\medskip 


We will use the following interpolation result  at the end of Section 8 to obtain  bounds 
in terms of the $X^s$ spaces from those in $U^2_\Delta H^s$ and $U^p_\Delta H^s$  just as in in \cite{HTT}
The proof relies solely on linear interpolation \cite{HHK, HTT}.

\begin{proposition}[Proposition 2.20 in \cite{HHK} and Lemma 2.4 \cite{HTT}]\label{U-interpolation}
Let $q_1, \dots q_m > 2$ where $m=1, 2, \mbox{or } 3$, \, $E$ be a Banach space and $T: U^{q_1}\times \dots \times U^{q_m}\rightarrow E$ be a bounded m-linear operator with \begin{equation} \label{inter-bd1} 
 \|T (u_1,\dots ,u_m)\|_{E}\leq C \prod_{i=1}^m\|u_i\|_{U_{\Delta}^{q_i}}
 \end{equation} In addition assume there exists $0 <C_2\leq C$ such that the estimate 
 \begin{equation} \label{inter-bd2} 
 \|T (u_1,\dots ,u_m)\|_{E}\leq C_2 \prod_{i=1}^m\|u_i\|_{U_{\Delta}^2}
 \end{equation} holds true.  Then, $T$ satisfies the estimate 
 \begin{equation} \label{inter-bd3} 
 \|T (u_1,\dots ,u_m)\|_{E}\lesssim C_2  (\ln{\frac{C}{C_2} +1})^m \, \prod_{i=1}^m\|u_i\|_{V^2}, \qquad u_i \in V^2_{rc}, \, \, i=1, \dots m, 
 \end{equation}  where $V^2_{rc}$ denotes the closed subspace of $V^2$ of all right continuous functions of $t$ such that $\lim_{t \to -\infty} v(t) =0$. 
\end{proposition}
Finally we state  two results from \cite{HTT} we  rely on in the next sections. In what follows, $\mathcal I$ denotes the Duhamel operator,
$$ \mathcal I(f)(t):= \int_0^t \, e^{i(t-t') \Delta} \, f(t') \, dt', \qquad  t \geq 0,$$ defined for  $f \in L^1_{\text{loc}}([0, \infty), L^2(\mathbb T^3)).$

\begin{proposition}[Proposition 2.11 in \cite{HTT}] \label{duhamel} Let $s \geq 0$ and $T>0$. For $f\in L^1([0, T),\, H^s(\mathbb T^3))$ we have $\mathcal I(f)\in X^s([0,T))$ and 
$$\|\mathcal I(f)\|_{ X^s([0,T))}\leq \sup_{v\in Y^{-s}([0,T)): \|v\|_{Y^{-s}}=1}
\left|\int_0^T\int_{\T^3} f(t,x)\overline{v(t,x)}dx dt\right|.$$ 
\end{proposition}
As a consequence, note we have 
\begin{equation}\label{from-duhamel}  \| \mathcal I (f) \|_{X^s([0, T))} \lesssim \|f\|_{ L^1([0, T),\, H^s(\mathbb T^3))}
\end{equation} 

\smallskip 

\begin{proposition}[Proposition 4.1 in \cite{HTT}]\label{all-deterministic} Let $s \geq 1$ be fixed. Then for all $T \in (0, 2\pi]$ and $u_k \in X^s([0, T)),\, k=1, \dots 5$, the estimate 
\begin{equation}\label{5deterministic} \| \mathcal I\bigl( \prod_{k=1}^5  \widetilde{u_k}   \bigr) \|_{X^s([0, T))} \lesssim \sum_{j=1}^5 \, \, \| u_j\|_{X^s([0, T))}\prod_{k=1, k \neq j}^5 \| u_k\|_{X^1([0, T))}
 \end{equation} holds true, where $\widetilde{u_k} $ denotes either $\overline{u_k} $ or $ u_k $. In particular, \eqref{5deterministic} follows from the estimate for the multilinear form:
$$\left| \int_{[0, T)\times \T^3} \, \mathcal \prod_{k=0}^5  \widetilde{u_k}  dx dt \right| \, \lesssim \, \| u_0\|_{Y^{-s}([0, T))} \, \sum_{j=1}^5 \, \left( \| u_j\|_{X^s([0, T))}\prod_{k=1, k \neq j}^5 \| u_k\|_{X^1([0, T))}\right) 
$$ where $u_0:= P_{\leq N} v$. 
\end{proposition}

\medskip
Next, we recall the $L^p(\T \times \T^3)$ Strichartz-type estimates of Bourgain's \cite{B5} in this context. First recall the usual Littlewood-Paley decomposition of  periodic functions. For $N \geq 1$ a dyadic number, we denote by $P_{\leq N}$ the rectangular Fourier  projection operator $$P_{\leq N} f =  \sum_{n=(n_1, n_2, n_3)\in \mathbb Z^3 \, : \, |n_i| \leq N} \widehat{f}(n)e^{ i n \cdot x}.$$ Then $P_N= P_{\leq N} - P_{\leq N-1}$ so that $P_{\leq N} = \sum_{M=1}^N P_M$ and $P_N^{\perp}:= I - P_N$. We then have 
$$ \| f \|_{H^s(\mathbb T^3)}: = \bigl( \sum_{n \in \mathbb Z^3} \langle n \rangle^{2s} \, | \widehat{f}(n)|^2 \bigr)^{\frac{1}{2}} \, \equiv \,  \bigl( \sum_{N \geq 1}  N^{2s} \, \|P_N (f) \|_{L^2(\mathbb T^3)}^2 \bigr)^{\frac{1}{2}}. $$ 

\smallskip

\begin{definition} \label{cubescollection}\, For $N \geq 1$,  we denote by  $\mathcal C_N$  the collection of cubes $C$ in $\mathbb Z^3$ with sides parallel to the axis of sidelength $N$. 
\end{definition}

\begin{proposition}\label{prop_stric}[Proposition 3.1, Corollary 3.2 in \cite{HTT} (cf. \cite{B5})] Let $p >4$. For all $N\geq1$ we have 
\begin{eqnarray} \label{Strichartz-1}
&&\| P_N \, e^{it \Delta} \phi \|_{L^p(\T \times \T^3)}\, \lesssim \, N^{\frac{3}{2} - \frac{5}{p}} \, \| P_N \phi \|_{L^2(\T^3)}, \\
&& \| P_C \, e^{it \Delta} \phi \|_{L^p(\T \times \T^3)}\, \lesssim \, N^{\frac{3}{2} - \frac{5}{p}} \, \| P_C \phi \|_{L^2(\T^3)}, \label{Strichartz-2}\\
&&   \| P_C u  \|_{L^p(\T \times \T^3)}\, \lesssim \, N^{\frac{3}{2} - \frac{5}{p}} \, \| P_C u \|_{U^p_{\Delta}L^2}, \label{Strichartz-3}
\end{eqnarray}
where $P_C$ is the Fourier projection operator onto $C \in \mathcal C_N$ defined by the multiplier $\chi_C$, the characteristic function over $C$.
\end{proposition}

\medskip 

Finally we prove two propositions which will play an important role in Sections \ref{blocks} and \ref{theproof}.

\smallskip

\begin{proposition}\label{lowerorderterms} Let  $u$, $v$ and $w$ be functions of $x$ and $t$ such that, 
\begin{eqnarray*}
\widehat u(n,t)&=&a^1_n(t)a^2_n(t)a^3_n(t)\\
\widehat v(n,t)&=&a^1_n(t)a^2_n(t)a^3_n(t)a^4_n(t)a^5_n(t)\\
\widehat w(n,t)&=&a^1_n(t)a^2_n(t)a^3_n(t)\sum_{m}a^4_{m}a^5_{n-m}.
\end{eqnarray*}
and $|n|\sim N$. Assume  that $J\subseteq\{1,2,3,4,5\}$ and  if $i\in J$ then  
$$a_n^i(t)=\frac{g_n(\omega)}{|n|^{\frac{3}{2}+\varepsilon}}e^{it|n|^2}$$
while if $i\notin J$ then there is a detrministic function $f_i$ such that $\widehat{f_i}(n,t)=a_n^i(t)$. Then
\begin{eqnarray}\label{ulp}
\|P_Nu\|_{L^{p}(\T\times\T^3)}&\lesssim&\prod_{i\notin J\cap\{1,2,3\}}\|P_Nf_i\|_{Y^0}, \quad p>4\\\label{ul2}
\|P_Nu\|_{L^{2}(\T\times\T^3)}&\lesssim&\prod_{i\notin J\cap\{1,2,3\}}\|P_Nf_i\|_{Y^0}\\
\label{vl2}
\|P_Nv\|_{L^{2}(\T\times\T^3)}&\lesssim&\prod_{i\notin J}\|P_Nf_i\|_{Y^0}\\\label{wl2}
\|P_Nw\|_{L^{2}(\T\times\T^3)}&\lesssim&\prod_{i\notin J, i\ne 4,5}\|P_Nf_i\|_{Y^0}\prod_{j\notin J, j=4,5}\|f_j\|_{Y^0}.
\end{eqnarray}
\end{proposition}
\begin{proof}
To prove \eqref{ulp} we write $u=k_1 * k_2 * k_3$, where the convolution is only with respect to the space variable. Then by Young's inequality in the space variable followed by H\"older's inequality and the embedding \eqref{embed4}  we have the desired inequality.

\medskip

To prove \eqref{ul2} we use Plancherel
\begin{eqnarray*}\|P_Nu\|_{L^{2}(\T\times\T^3)}&\lesssim& \left\|\|\chi_{|n|\sim N}a_n^1a_n^2a_n^3\|_{\ell^2}\right\|_{L^{\infty}(\T)}
\lesssim \left\|\prod_{i=1}^3\|\chi_{|n|\sim N}a_n^i\|_{\ell^2}\right\|_{L^{\infty}(\T)}\\
&\lesssim&\left\|\prod_{i=1}^3\|P_Nf_i\|_{L^{2}_x}\right\|_{L^{\infty}(\T)}\lesssim \prod_{i\notin J\cap\{1,2,3\}}
\|P_Nf_i\|_{L^\infty(\T,L^2(\T^3))}\end{eqnarray*}
and the conclusion follows from the embedding \eqref{embed4}.

\medskip
To prove \eqref{vl2} we proceed in a similar manner.

\medskip
To prove \eqref{wl2} we first write 
$$\|P_Nw\|_{L^{2}(\T\times\T^3)}\sim \|P_N(k_1* k_2* k_3*(k_4k_5))\|_{L^{2}(\T\times\T^3)},$$
and  by Young's,  H\"older's   and Cauchy-Schwarz  inequality we continue with
\begin{eqnarray*}
&\lesssim&\left\|\prod_{i=1}^3\|P_Nk_i\|_{L^{2}}\|P_N(k_4k_5)\|_{L^{1}}\right\|_{L^{2}(\T)}\\
&\lesssim&\left\|\prod_{i=1}^3\|P_Nk_i\|_{L^{2}}\|k_4\|_{L^{2}}\|k_5\|_{L^{2}}\right\|_{L^{2}(\T)}\\&\lesssim &
\prod_{i\notin J, i\ne 4,5}\|P_Nf_i\|_{L^\infty(\T,L^2(\T^3))}\prod_{j\notin J, j=4,5}\|f_j\|_{L^\infty(\T,L^2(\T^3))}.
\end{eqnarray*}

\end{proof}

We now state  a trilinear $L^2$ estimate that is similar to  Proposition 3.5 in \cite{HTT} but in which some of the functions may be linear evolution of random data.
\begin{proposition}
Assume $N_1\geq N_2\geq N_3$ and that $C \in \mathcal{C}_{N_2}$,  a cube of sidelength $N_2$. Assume also that $J\subseteq \{1,2,3\}$ and  such that if $j\in J$ then $\widehat{u_j}(n)=e^{i|n|^2t}\frac{g_n(\omega)}{|n|^{\frac{3}{2}+\varepsilon}}$ for $\varepsilon>0$ small. Then
\begin{equation}\label{strichartz-mix}
\|P_CP_{N_1}\widetilde{u_1}P_{N_2}\widetilde{u_2}P_{N_3}\widetilde{u_3}\|_{L^2(\T\times \T^3)}\lesssim N_2N_3\prod_{j\notin J}\|P_{N_j}u_j\|_{U^4_\Delta L^2}.
\end{equation}
and 
\begin{equation}\label{2strichartz-mix}
\|P_CP_{N_1}\widetilde{u_1}P_{N_2}\widetilde{u_2}\|_{L^2(\T\times\T^3)}\lesssim N_2^{\frac{1}{2}+\varepsilon}\prod_{j\notin J}\|P_{N_j}u_j\|_{U^4_\Delta L^2},
\end{equation}
where $\widetilde{u_k} $ denotes either $\overline{u_k} $ or $ u_k $.

Moreover  \eqref{strichartz-mix} and \eqref{2strichartz-mix} also hold with 
 the $Y^0$ norms in the right hand side.

\end{proposition}
\begin{proof}
To prove \eqref{strichartz-mix} we follow  the proof of (24) in \cite{HTT}. We  write
$$\|P_CP_{N_1}\widetilde{u_1}P_{N_2}\widetilde{u_2}P_{N_3}\widetilde{u_3}\|_{L^2(\T \times \T^3)} \,\lesssim\, \|P_CP_{N_1}u_1\|_{L^p}\|P_{N_2}u_2\|_{L^p}
\|P_{N_3}u_3\|_{L^q}$$
where $\frac{2}{p}+\frac{1}{q}=\frac{1}{2}$ and $4<p<5$. Then we use \eqref{Strichartz-2}
for the random linear functions and \eqref{Strichartz-3} for the deterministic functions to obtain 
$$\|P_CP_{N_1}\widetilde{u_1}P_{N_2}\widetilde{u_2}P_{N_3}\widetilde{u_3}\|_{L^2(\T \times \T^3)} \, \lesssim \, N_2N_3\left(\frac{N_3}{N_2}\right)^{-2+\frac{10}{p}}\prod_{j\notin J}\|P_{N_j}u_j\|_{U^4_\Delta L^2},$$
where we used the embedding \eqref{embed2}.

\medskip
To prove \eqref{2strichartz-mix} we use H\"older's inequality to write 
\begin{equation}\label{thel4+}\|P_CP_{N_1}\widetilde{u_1}P_{N_2}\widetilde{u_2}\|_{L^2(\T \times \T^3)} \, \lesssim \,  \|P_CP_{N_1}u_1\|_{L^{4+\varepsilon}}\|P_{N_2}u_2\|_{L^{4+\varepsilon}}\end{equation}
we then  we use  \eqref{Strichartz-2}, \eqref{Strichartz-3}  and the embedding \eqref{embed2} to continue with
$$\lesssim N_2^{\frac{1}{2}+\varepsilon}\prod_{j\notin J}\|P_{N_j}u_j\|_{U^4_\Delta L^2}.$$
To obtain the $Y^0$ in the right hand side we use  the interpolation Proposition \ref{U-interpolation} and the embedding \eqref{embed2}.
\end{proof}

\medskip


\section{Almost sure local well-posedness for the initial value problem \eqref{IVPremove}}

We define  
\begin{equation}\label{linear}
v_0^\omega(t,x)=S(t)\phi^\omega(x) 
\end{equation}
where $\phi^\omega(x)$ is as in \eqref{initial} and instead of solving the initial value problem \eqref{IVPremove} we solve the one  for $w=v-v_0^\omega$:
\begin{equation}\label{IVPdifference}
\begin{cases}
&iw_t+\Delta w=\mathcal N(w+v_0^\omega)\quad x\in \T^3\\
&w(0,x)=0,
\end{cases}
\end{equation}
where $\mathcal N(\cdot)$ was defined in \eqref{Nv}. To understand the nonlinear term of \eqref{IVPdifference} we express it in terms of its spatial Fourier transform. Let $a_n:=\hat v(n)$, $\theta^\omega_n:=\FF(S(t)\phi^\omega)(n)$, then $b_n:=\hat w(n)=a_n-\theta^\omega_n$. Now we recall \eqref{rewrite} and 
in it we replace $a_n$ with $b_n+\theta^\omega_n$. Then
\begin{equation}\label{rewritedifference}
\FF(\mathcal N(w+v_0^\omega))(n)=\sum_{k=1}^7J_k(b_n+\theta^\omega_n),
\end{equation}
where here $J_k(b_n+\theta^\omega_n)$ means that the terms $J_k$ defined in \eqref{j1}--\eqref{j7} are evaluated for the sequence $(b_n+\theta^\omega_n)$ instead of $a_n$. 


\medskip
We are now ready to state the almost sure well-posedness result for the initial value problem \eqref{IVPdifference}.

\begin{theorem}\label{lwpdifferenceIVP}  Let $0< \alpha < \frac{1}{12}$,  $s \in (1+ 4\alpha, \, \frac{3}{2} - 2\alpha)$.
There exists $0<\delta_0\ll 1$ and $r=r(s,\alpha)>0$ such that for any $\delta < \delta_0,$ there exists $\Omega_\delta\in A$ with 
$$\pr(\Omega_\delta^c)<
e^{-\frac{1}{\delta^{r}}},$$
and for each $\omega\in \Omega_\delta$ there exists a unique solution $w$ of \eqref{IVPdifference}
in the space $X^s([0,\delta))\cap C([0,\delta),H^s(\T^3)).$
\end{theorem}

The proof of this theorem follows from  the following two  propositions via contraction mapping argument.

\begin{proposition}\label{main-lemma} Let $0< \alpha < \frac{1}{12}$,  $s \in (1+ 4\alpha, \, \frac{3}{2} - 2\alpha),$  \, $\delta\ll1$ and $R>0$ be fixed. Assume  $N_i, \, i=0, \dots 5$ are dyadic numbers and $N_1 \geq N_2 \geq N_3 \geq N_4 \geq N_5$. Then there exists $\rho=\rho(s, \alpha)>0$, $\mu>0$,  and  $\Omega_\delta \in A$ such that 
$$\pr(\Omega_\delta^c)<  e^{-\frac{1}{\delta^{r}}},$$ and such that for $\omega\in \Omega_\delta$ we have: 

\noindent
If  $N_1\gg N_0$  or $P_{N_1}w=P_{N_1}v_0^\omega$
\begin{eqnarray}\label{outcome1}
 &&\left| \int_0^{2\pi} \int_{\T^3}  D^s \bigl( \mathcal N(P_{N_i}(w+v_0^\omega))   \bigr) \overline{ P_{N_0}h }  \, dx \, dt \right| \\\notag
&& \qquad \qquad \lesssim \delta^{-\mu r} N_1^{-\rho} \| P_{N_0}h\|_{Y^{-s}}\left( 1+ \prod_{i\notin J}\|P_{N_i}w\|_{X^s}\right). 
 \end{eqnarray}
 If $N_1\sim N_0$ and $P_{N_1}w\ne P_{N_1}v_0^\omega$
\begin{eqnarray}\label{outcome2} &&\left| \int_0^{2\pi} \int_{\T^3}  D^s \bigl( \mathcal N(P_{N_i}(w+v_0^\omega))   \bigr) \overline{ P_{N_0}h }  \, dx \, dt \right|\\\notag  
& \lesssim&
 \delta^{-\mu r} N_2^{-\rho} \| P_{N_0}h\|_{Y^{-s}}\|P_{N_1}w\|_{X^s}
\left( 1+ \prod_{i\notin J, i\ne 1}\|\psi_{\delta}P_{N_i}w\|_{X^s}\right),
 \end{eqnarray}
where  $v_0^\omega$ is as   in \eqref{linear}, $w\in  X^s([0, 2\pi])$, $J\subseteq\{1,2,3,4,5\}$ denote those indices corresponding to random functions.
\end{proposition}

\begin{proposition}\label{main-estimate}  Let $0< \alpha < \frac{1}{12}$,  $s \in (1+ 4\alpha, \, \frac{3}{2} - 2\alpha)$ and $\delta\ll1$ be fixed. Let $v_0^\omega$ be defined as in \eqref{linear} and assume $w \in X^s([0,2\pi])$. Then there exist $\theta=\theta(s,\alpha)>0, \, r=r(s,\alpha)$ and  $\Omega_\delta\in A$ such that 
$$\pr(\Omega_\delta^c)< e^{-\frac{1}{\delta^{r}}},$$ and such that for $\omega\in \Omega_\delta$

\begin{equation}\label{5mixed} \| \mathcal I\bigl (\psi_\delta \mathcal N(w+v_0^\omega) \bigr) \|_{X^s([0, 2\pi])} \lesssim \delta^\theta \left(1+ \| \psi_\delta w\|_{X^s([0, 2\pi])}^5\right)
 \end{equation}
 where $\mathcal N(\cdot) $ was defined in \eqref{Nv} and $\psi_\delta$ is a smooth time cut-off of the interval $[0,\delta]$.
\end{proposition}

The proof of Proposition \ref{main-lemma} is the content of Sections \ref{blocks} and \ref{theproof} while Proposition \ref{main-estimate} is proved in Section \ref{proof-main-estimate}.

  
\medskip

\section{Auxiliary lemmata and further notation}

We begin by recalling some counting estimates for integer lattice sets (c.f.  Bourgain \cite{B5}).
\begin{lemma}\label{lattice-counting} Let $S_R$ be a sphere of radius $R$, $B_r$ be a ball of radius $r$ and $\mathcal P$ be a plane in $\R^3$. Then we have
\begin{eqnarray}\label{sphere}
&&\#Z^3\cap S_R\lesssim R\\\label{ball}
&&\#Z^3\cap B_r\cap S_R\lesssim \min(R,r^2)\\\label{planepi}
&&\#Z^3\cap B_r\cap \mathcal P\lesssim r^2.
\end{eqnarray}
\end{lemma}

Next, we state a result we will invoke when the the higher frequencies correspond to deterministic terms and one can afford to ignore the moments given by the lower frequency random terms as well as rely on Strichartz estimates.
\smallskip

\begin{lemma} \label{cutall}  Assume  $N_i, \, i=0, \dots 5$ are dyadic numbers and $N_1 \sim N_0$ and $N_1 \geq N_2 \geq N_3 \geq N_4 \geq N_5$. Let $\{C\}$ be a partition of  $\Z^3$ by cubes  $C \in \mathcal{C}_{N_2}$,  and let $\{Q\}$ be a partition of  $\Z^3$  by cubes  $Q \in \mathcal{C}_{N_3}$. Then  
\begin{eqnarray}\label{cut1} &&\sum_{N_i,\,  i=0, \dots 5} \,\left| \int_0^1 \int_{\T^3}  P_{N_1} f_1 P_{N_2} f_2P_{N_3} f_3P_{N_4} f_4P_{N_5} f_5\overline{P_{N_0}  h}   \, dx \, dt\right| \lesssim\\\notag
&&\sum_{N_i,\,  i=0, \dots 5}  \left( \sup_{\tilde C} \| P_{\tilde C} P_{N_1} f_1P_{N_2} f_2 P_{N_\ell} f_\ell\|_{L^2_{xt}}^2 \sum_{C, Q} \|P_{Q}P_{C} \overline{P_{N_0}  h} P_{N_3} f_3P_{N_r} f_r\|_{L^2_{xt}}^2\right)^\frac{1}{2}
\end{eqnarray}
where $\ell\ne r\in \{4,5\}$ and $\tilde C$ are cubes whose sidelength  is $10N_2$.
\end{lemma}
\begin{proof} The proof of \eqref{cut1} follows from orthogonality arguments. 
\end{proof}

Just as Bourgain in \cite{B4}, in the course of the proof we will use the following classical result about matrices, which we state as a lemma for convenience. 
\begin{lemma} \label{matrixnorm} 
Let $\mathcal A = (A_{ik})_{\substack{1\leq i\leq N\\ 1\leq k \leq M}}$  be an $N\times M$ matrix with adjoint $\mathcal A^{\ast}=( \overline{A}_{kj})_{\substack{1\leq k\leq M\\ 1\leq j \leq N}}$. Then,
\begin{equation} 
\| \mathcal A\mathcal A^{\ast} \| \leq   \max_{1 \leq j \leq N} \left( \sum_{k=1}^M | A_{jk}|^2 \right)  \, + \,   \left( \sum_{ i \neq j} \left| \sum_{k=1}^M A_{ik} \overline{A}_{jk} \right|^2 \right)^{\frac{1}{2}}  \label{frobenius}
\end{equation} where $\| \cdot \|$ means the 2-norm.
\end{lemma}

\begin{proof}  Decompose  $\mathcal A\mathcal A^{\ast}$ into the sum of a diagonal matrix $D$ plus an off-diagonal one $F$.  Then note the 2-norm of $D$ is bounded by the square root of the largest eigenvalue of $D D^\ast$ which, since $D$ is diagonal, is the maximum of the absolute value of the diagonal entries of D. This gives the first term in \eqref{frobenius}. Bounding the 2-norm of $F$ by the Fr\"obenius norm of $F$ gives the second term in \eqref{frobenius}. \end{proof}

\medskip
\noindent
{\bf Notation}: Given  $k$-tuples $(n_1,\dots, n_k)\in \Z^{3k}$,  a set of constraints $\mathscr{C}$ on them,  and a subset of indices $\{i_1, \dots, i_h\}\subseteq \{1,\dots,k\}$,   we denote by
$S_{(n_{i_1},\dots,n_{i_h})}$ the set of $(k-h)$-tuples $(n_{j_1},\dots, n_{j_{k-h}}), \, \, \{j_1, \dots j_{k-h}\}=  \{1,\dots,k\} \setminus \{i_1, \dots, i_h\}$, which satisfy the constraints $\mathscr{C}$ for fixed $(n_{i_1},\dots,n_{i_h})$.  We also denote by $|S_{(n_{i_1},\dots,n_{i_h})}|$ its cardinality.

\medskip 

\section{The Trilinear and Bilinear Building Blocks} \label{blocks}

\smallskip

In this section, we denote by  $D_j := e^{it \Delta} P_{N_j} \phi $  solutions to the linear equation for data $\phi$ in $L^2$ localized at frequency $N_j$ and by 
$R_k$ the function defined as, 
\begin{equation}\label{Rhat}\widehat{R_k}(n)=\chi_{\{|n|\sim N_k\}}(n)\frac{g_n(\omega)}{\langle n\rangle^{\frac{3}{2}}}e^{it|n|^2},\end{equation}
and representing the linear evolution of a random function of type  \eqref{initial}, localized at frequency $N_k$  and {\it almost} $L^2$ normalized.

\subsection{Trilinear Estimates}\label{trilinear-est} We prove certain trilinear estimates which serve as building blocks for the proof in Section \ref{theproof}.  Their proofs are of the same flavor as those presented by  Bourgain in \cite{B4}.  For $N_j, \, j=1,2,3$ dyadic numbers, let $\alpha_j = 0$ or $1$ for $j=1, 2, 3$ and define
 \begin{equation}\label{upsilon}\Upsilon(n,m) := \left\{ (n_1, m_1; n_2, m_2 ; n_3, m_3) :  \begin{aligned} \quad &n= (-1)^{\alpha_1} n_1 + (-1)^{\alpha_2} n_2+ (-1)^{\alpha_3} n_3 \\ &n_k \neq n_{\ell} \, \,\mbox{ whenever } \, \, \alpha_k \neq \alpha_{\ell}, \\ &|n_j| \sim N_j, \quad j=1, \dots 3  \\
&m= (-1)^{\alpha_1} m_1 + (-1)^{\alpha_2} m_2+ (-1)^{\alpha_3} m_3 \, \end{aligned} \right\}\end{equation} and define $T_{\Upsilon}$ to be the multilinear operator with multiplier $\chi_\Upsilon$.  

\smallskip

\begin{proposition} \label{alltrilinear} Fix  $N_1\geq N_2\geq N_3$, $r, \, \delta>0$ and $C \in \mathcal{C}_{N_2}$. Then there exists $\mu, \varepsilon>0$, a set $\Omega_\delta\in A$ such that $\pr(\Omega_\delta^c)\leq 
e^{-\frac{1}{\delta^r}}$ and  such that for any $\omega\in \Omega_\delta$ we have the following estimates: 

\begin{eqnarray}
\label{barRDR}  && \, \,  \| T_{\Upsilon}(P_{C} \bar R_1, \tilde D_2, R_3)\|_{L^2(\T \times \T^3)}\, \lesssim \, \delta^{-\mu r}\,  N_2^\frac{5}{4}  N_1^{-\frac{1}{2} }  \| P_{N_2} \phi \|_{L_x^2 }. \\
 \label{barRDbarR} &&  \, \, \| T_{\Upsilon}(P_{C}\bar R_1, \tilde D_2, \bar R_3)\|_{L^2(\T\times \T^3)} \, \lesssim    \, \delta^{-\mu r} N_2^\frac{5}{4}  N_1^{-\frac{1}{2} }  \| P_{N_2} \phi \|_{L_x^2 }. \\
  \label{DbarRR}  &&  \, \,  \|T_{\Upsilon}( P_C\tilde D_1,\bar R_2, R_3)\|_{L^2(\T \times \T^3)} \,\lesssim  \,  \delta^{-\mu r}  N_2^{\frac{3}{4}}  \|  P_{C}P_{N_1} \phi \|_{L_x^2 }. \\
  \label{DRR} && \, \, \| T_{\Upsilon}(P_C\tilde D_1, R_2, R_3)\|_{L^2(\T\times \T^3)}\, \lesssim  \,  \delta^{-\mu r}  N_2^{\frac{3}{4}}  \|  P_{C}P_{N_1} \phi \|_{L_x^2 }.\\
\label{barRRD} && \, \,  \| T_{\Upsilon}(P_{C}\bar R_1,  R_2, \tilde D_3)\|_{L^2(\T \times \T^3)} \lesssim \, \delta^{-\mu r} \left[ N_1^{-\frac{3}{4} }  N_2^{\frac{1}{2}}  N_3^{\frac{5}{4}} 
+  N_1^{-\frac{1}{2} }  N_2^{\frac{1}{2}}  N_3^{\frac{3}{4}} \right]  \,  \|  P_{N_3} \phi \|_{L_x^2 }. \\
\label{barRbarRD}  && \, \, \| T_{\Upsilon}(P_{C}\bar R_1,  \bar R_2, \tilde D_3)\|_{L^2(\T \times \T^3)}\,  \lesssim \, \delta^{-\mu r} \left[ N_1^{-\frac{3}{4} }  N_2^{\frac{1}{2}}  N_3^{\frac{5}{4}} +  N_1^{-\frac{1}{2} }  N_2^{\frac{1}{2}}  N_3^{\frac{3}{4}} \right] \,   \| P_{N_3} \phi \|_{L_x^2 }. \\
\label{RDD}  &&\,\, \| T_{\Upsilon}(P_{C}R_1, \tilde D_2, \tilde D_3)\|_{L^2(\T\times \T^3)} \, \lesssim \, \delta^{-\mu r}  N_2^{\frac{1}{2} + \frac{3 \theta}{4}} N_1^{-\frac{1}{2} +\varepsilon}  \min(N_1, N_2^2)^{\frac{1-\theta}{2} } N_3^{\frac{3}{2}}  \quad \times \\ 
&& \hspace{2.5in} \times \qquad   \| P_{N_2} \phi \|_{L_x^2 }  \| P_{N_3} \phi \|_{L_x^2 }, \quad 0\leq \theta\leq 1. \notag \\
\label{DRD} &&\, \, \| T_{\Upsilon}( P_C\tilde D_1, R_2, \tilde D_3)\|_{L^2(\T\times \T^3)} \, \lesssim \,  \delta^{-\mu r}  N_2^{\frac{1}{2} + \varepsilon} N_3^{\frac{3}{2}}   \|P_{N_1} \phi \|_{L_x^2 }  \| P_{N_3} \phi \|_{L_x^2 }.\\
 \label{barRbarRR} &&\, \,  \|T_{\Upsilon}( P_{C}\bar R_1,  \bar R_2, R_3)\|_{L^2(\T \times \T^3)} \, \lesssim \, \delta^{-\mu r}   N_1^{-\frac{1}{2}} N_2^{\frac{1}{2}}.\\
\label{barRRbarR} &&\, \, \| T_{\Upsilon}( P_{C}\bar R_1, R_2, \bar R_3)\|_{L^2(\T \times \T^3)}\, \lesssim \, \delta^{-\mu r} N_1^{-\frac{1}{2}} N_2^{\frac{1}{2}}. \\
\label{barRRR} &&\, \, \| T_{\Upsilon}(  P_{C}\bar R_1, R_2, R_3)\|_{L^2(\T \times \T^3)}\, \lesssim \, \delta^{-\mu r}  N_1^{-\frac{1}{2}} N_2^{\frac{1}{2}}. 
%
%
\end{eqnarray}

\noindent Note that here the bar $-$ indicates complex conjugate while the tilde $\sim$ indicates both complex conjugate or not. Also, without writing it explicitly, we always assume that if $\widehat{\overline R}(n_1)$ and $\widehat{ R}(n_2)$ appear in the trilinear expressions in the left hand side, then $n_1\ne n_2$.
\end{proposition}
\begin{remark} In using the trilinear estimates above, sometimes it is convenient to interpret a random term as deterministic and choose the minimum estimate possible. For example, in considering $\|P_{C}\bar R_1  \bar R_2 R_3\|_{L^2}$ we have a choice between \eqref{barRbarRR} and \eqref{barRbarRD} by 
thinking of $R_3$ as an `almost' $L^2$  normalized $\tilde D_3$ function. 
\end{remark}

 \begin{proposition}{\label{One}} Let  $D_j$ and $R_k$ be as above and fix $N_1\geq N_2\geq N_3$, $r, \, \delta>0$ and $C \in \mathcal{C}_{N_2}$. Then there exists $\mu >0$ and a set $\Omega_\delta\in A$ such that $\pr(\Omega_\delta^c)\leq 
e^{-\frac{1}{\delta^r}}$ such that  for any $\omega\in \Omega_\delta$ we have  \eqref{barRDR} and  \eqref{barRDbarR}.


\end{proposition}
 \begin{proof} As in \cite{HTT} we will first assume that the deterministic functions $D_i$ are localized linear solutions, that is
 $D_i=P_{N_i}S(t)\psi$ and $\hat\psi(n)=a_n$. Once an estimate is proved with $\|\chi{_{N_i}}(n)\, a_n\|_{\ell^2}$ in the right hand side we then invoke the  transfer principle of Proposition \ref{transfer} to complete the proof.
 
 \medskip
 We start by estimating \eqref{barRDR}. Without any loss of generality we assume that $\tilde D_2 = D_2$. By using Fourier transform to write the left hand side we note that it is enough to estimate 
\begin{equation}\label{T1barRDR}
\mathcal{T}:=  \sum_{m \in \Z, \, n \in \Z^3} \,\left|   \sum_{\substack{ n= -n_1+ n_2 + n_3 \\n_1\ne n_2, n_3
\\ m = -|n_1|^2 +|n_2|^2+|n_3|^2}} \chi_{C}(n_1)\frac{\cg_{n_1}(\omega)}{|n_1|^\frac{3}{2}}
a_{n_2} \frac{g_{n_3}(\omega)}{|n_3|^{\frac{3}{2}}} \right|^2,
\end{equation}
where we recall that $C$ is a cube of sidelength $N_2$.  We are going to use duality and a change of variable since, as it will be apparent below, the counting with respect to the time frequency will be more favorable. 

Using duality we have that
$$\mathcal{T}=\left[\sup_{\|\gamma\otimes k\|_{\ell^2}\leq 1}\left|\sum_{m,n}k(n)\gamma(m)\sum_{\substack{ n= -n_1+ n_2 + n_3 \\n_1\ne n_2, n_3
\\ m = -|n_1|^2 +|n_2|^2+|n_3|^2}} \chi_{C}(n_1)\frac{\cg_{n_1}(\omega)}{|n_1|^\frac{3}{2}}
a_{n_2} \frac{g_{n_3}(\omega)}{|n_3|^{\frac{3}{2}}}\right|\right]^2
$$
Let $\zeta := m - |n_2|^2= -|n_1|^2 + |n_3|^2$, then we continue with
\begin{eqnarray*}
\mathcal{T}&=&\left[\sup_{\|\gamma\otimes k\|_{\ell^2}\leq 1}\left|\sum_{n_2}a_{n_2}\sum_{\zeta}\gamma(\zeta+|n_2|^2)
\sum_{\substack{ n= -n_1+ n_2 + n_3 \\n_1\ne n_2, n_3
\\ \zeta = -|n_1|^2 +|n_3|^2}} \chi_{C}(n_1)\frac{\cg_{n_1}(\omega)}{|n_1|^\frac{3}{2}}
 \frac{g_{n_3}(\omega)}{|n_3|^{\frac{3}{2}}}k_n\right|\right]^2\\
&\lesssim&\sup_{\|\gamma\otimes k\|_{\ell^2}\leq 1}\|a_{n_2}\|_{\ell^2_{n_2}}^2\, \|\gamma\|_{\ell^2_{\zeta}}^2\, 
\sum_{n_2,\zeta}\left|\sum_{\substack{ n= -n_1+ n_2 + n_3 \\n_1\ne n_2, n_3
\\ \zeta = -|n_1|^2 +|n_3|^2}} \chi_{C}(n_1)\frac{\cg_{n_1}(\omega)}{|n_1|^\frac{3}{2}}
\frac{g_{n_3}(\omega)}{|n_3|^{\frac{3}{2}}}k_n 
\right|^2.
\end{eqnarray*}
All in all,  we then have to estimate uniformly for $\|\gamma\otimes k\|_{\ell^2}\leq 1$,
\begin{equation}\label{t1dual}
 \|a_{n_2}\|_{\ell^2}^2\|\gamma\|_{\ell^2}^2\sum_{n_2} \sum_{|\zeta| \leq N_1N_2}  \left| \sum_{n} \sigma_{n_2, n}  k_n \right|^2, 
\end{equation}  
where  
$$ \sigma_{n_2, n} =  \sum_{\substack{ n_2=n_1+n - n_3, \, n_1\ne n_2, n_3\\    \zeta= -|n_1|^2 +|n_3|^2}}  \chi_{C}(n_1) \frac{\cg_{n_1}(\omega)}{|n_1|^{\frac{3}{2}}} \frac{g_{n_3}(\omega)}{|n_3|^{\frac{3}{2}}}.$$
Note that  $\sigma_{n, n_2}$ also depends on $\zeta$  but we estimate it independently of $\zeta$. 
If we denote by $\GG$  the matrix of entries $ \sigma_{n_2, n}$, and we recall  that the variation in $\zeta$ is at most $N_1N_2$,  we are then reduced to estimating
$$
 \|a_{n_2}\|_{\ell^2}^2 N_1N_2  \|\GG\GG^*\|. 
$$
We note that by Lemma \ref{matrixnorm}
$$\|\GG\GG^*\|\lesssim \max_{n_2}\sum_{n}|\sigma_{n_2, n}|^2+\left(\sum_{n_2\ne n_2'}\left|\sum_{n\in \tilde C}\sigma_{n_2, n}\overline{\sigma}_{n_2', n}\right|^2\right)^\frac{1}{2}\, =:\, M_1+M_2,$$
where $\tilde C$ is a cube of sidelength  approximately $ N_2$.
%
To estimate $M_1$ we first define the set 
$$
S_{(\zeta,n_2)}=\{(n_1,n,n_3) \, : \, n_2=n_1+n - n_3, \, n_1\ne n_2,  n_3,\,  \zeta= -|n_1|^2 +|n_3|^2 \} 
$$ with $|S_{(\zeta,n_2)}|\lesssim N_3^3 N_1$, where we use \eqref{sphere} for fixed $n_3$. Then we have 
\begin{eqnarray*}M_1&\lesssim&  \sup_{(n_2,\zeta)}\left| \sum_{\substack{ n_2=n_1+n - n_3, \, n_1\ne n_2, n_3\\    \zeta= -|n_1|^2 +|n_3|^2}}  \chi_{C}(n_1) \frac{\cg_{n_1}(\omega)}{|n_1|^{\frac{3}{2}}} \frac{g_{n_3}(\omega)}{|n_3|^{\frac{3}{2}}}\right|^2.
\end{eqnarray*}
Now we use \eqref{largedeviation} with $\lambda=\delta^{-r}\|F_2(\omega)\|_{L^2}$ and Lemma \ref{independence} to obtain for $\omega$ outside a set of measure 
$e^{-\frac{1}{\delta^r}}$ the bound
\begin{eqnarray}\notag M_1&\lesssim&  \sup_{(n_2,\zeta)}\delta^{-2r} \sum_{S_{(\zeta,n_2)}}\sum_{S_{(\zeta,n_2)}}
\frac{1}{|n_1|^{\frac{3}{2}}} \frac{1}{|n_3|^{\frac{3}{2}}}\frac{1}{|\xi_1|^{\frac{3}{2}}} \frac{1}{|\xi_3|^{\frac{3}{2}}} \left| \int_\Omega\cg_{n_1}(\omega)g_{n_3}(\omega)g_{\xi_1}(\omega)\cg_{\xi_3}(\omega)\, dp(\omega) \right| \\\label{M1}
&\lesssim& \sup_{(n_2,\zeta)}\delta^{-2r}\sum_{S_{(\zeta,n_2)}}\frac{1}{|n_1|^{3}} \frac{1}{|n_3|^{3}}
\lesssim  \delta^{-2r}N_1^{-3 }  N_3^{-3} N_3^3N_1\sim \delta^{-2r}N_1^{-2 }.
\end{eqnarray}

To estimate $M_2$ we first write
\begin{eqnarray*}
M_2^2&=&\sum_{n_2\ne n_2'}\left|\sum_{n\in \tilde C}\sigma_{n_2, n}\overline \sigma_{n_2', n}\right|^2\sim
\sum_{n_2\ne n_2'}\left|\sum_{S_{(n_2,n_2')}}\frac{\cg_{n_1}(\omega)}{|n_1|^{\frac{3}{2}}}\frac{g_{n_3}(\omega)}{|n_3|^{\frac{3}{2}}}\frac{g_{n_1'}(\omega)}{|n_1'|^{\frac{3}{2}}}\frac{\cg_{n_3'}(\omega)}{|n_3'|^{\frac{3}{2}}}\right|^2
\end{eqnarray*}
where
\begin{equation*}
S_{(n_2,n_2',\zeta)}=\left\{(n, n_1, n_3,n_1', n_3') \, : \,  \begin{aligned}
&n_2=n_1+n-n_3, \quad n_2'=n_1'+n-n_3', \\
& n_1\ne n_2, n_3,\quad
  n_1'\ne n_2', n_3', \quad n\in \tilde C\\
&\zeta= -|n_1|^2 +|n_3|^2, \quad \zeta= -|n_1'|^2 +|n_3'|^2\end{aligned} \right\}.
\end{equation*}
We need to organize the estimates according to whether some frequencies are the same or not, in all we have six cases:
\begin{itemize}
\item {\bf Case $\beta_1$:} $n_1, n_1', n_3, n_3'$ are all different. 
\item {\bf Case $\beta_2$:} $n_1=n_1'; \, \, n_3\ne n_3'.$
\item {\bf Case $\beta_3$:} $n_1\ne n_1'; \, \, n_3= n_3'.$
\item {\bf Case $\beta_4$:} $n_1\ne n_3'; \, \, n_3 =  n_1'.$
\item {\bf Case $\beta_5$:} $n_1= n_3'; \, \, n_3\ne  n_1'.$
\item {\bf Case $\beta_6$:} $n_1= n_3'; \, \, n_3 =  n_1'.$ 
\end{itemize}

\medskip
{\bf Case $\beta_1$:} We define the set  
$$
S_{(\zeta)}=\left\{(n_2,n_2',n, n_1, n_3,n_1', n_3') \, : \,  \begin{aligned}&n_2=n_1+n-n_3, \quad n_2'=n_1'+n-n_3',\, \\
 &n_1\ne n_2, n_3,\quad n_1'\ne n_2', n_3', \quad n_1, n_1' \in C\\ 
&\zeta = -|n_1|^2 +|n_3|^2 , \quad \zeta= -|n_1'|^2 +|n_3'|^2\end{aligned}\right\}.
$$
and we note that   $|S_{(\zeta)}|\lesssim N_1^2N_3^6N_2^3$  since  $n\in \tilde C$ and for fixed $n_3$ and $n_3'$ we use \eqref{sphere} to count $n_1$ and $n_1'$.  Using \eqref{largedeviation} with
$\lambda=\delta^{-2r}\|F_4(\omega)\|_{L^2}$ and again Lemma \ref{independence} we can  write for $\omega$ as above
$$M_2^2\lesssim \delta^{-4r} \sum_{n_2\ne n_2'} \sum_{S_{(n_2,n_2',\zeta)}} \frac{1}{|n_1|^{3}}\frac{1}{|n_3|^{3}}\frac{1}{|n_1'|^{3}}\frac{1}{|n_3'|^{3}}\lesssim \delta^{-4r} N_1^{-6}N_3^{-6}N_1^2N_3^6N_2^3\sim  \delta^{-4r} N_1^{-4}N_2^3.$$

\medskip

\medskip
{\bf Case $\beta_2$:} First define the set, 
$$
S_{(n_2,n_2', n_3, n_3',\zeta)}=\left\{(n, n_1) \, : \,  \begin{aligned}&n_2=n_1+n-n_3, \quad n_2'=n_1+n-n_3',\\
&n_1\ne n_2, n_2', n_3, n_3', \quad n\in \tilde C\\
&\zeta = -|n_1|^2 + |n_3|^2, \quad \zeta= -|n_1|^2 +|n_3'|^2\end{aligned}  \right\}.
$$
To compute  $|S_{(n_2,n_2', n_3, n_3',\zeta)}|$ we  count $n_1$, then $n$ is determined. Since $n_1$ sits on a sphere  then by \eqref{sphere} we have  $|S_{(n_2,n_2', n_3, n_3',\zeta)}|\lesssim N_1$. Then we set 
$$
S_{(\zeta)}=\left\{(n_2,n_2',n, n_1, n_3, n_3') \, : \,  \begin{aligned} &n_2=n_1+n-n_3, \quad n_2'=n_1+n-n_3',\\
&n_1\ne n_2, n_2', n_3, n_3', \quad n\in \tilde C\\
&\zeta= -|n_1|^2 + |n_3|^2, \quad \zeta= -|n_1|^2 +|n_3'|^2\end{aligned}  \right\}.
$$
with  $|S_{(\zeta)}|\lesssim N_1N_3^6N_2^3,$  where we used again that $n\in \tilde C$ and \eqref{sphere}. Now, we have that
\begin{equation} \label{quote}M_2^2\,\sim\, 
\sum_{n_2\ne n_2'}\left|\sum_{S_{(n_2,n_2',\zeta)}}\frac{|g_{n_1}(\omega)|^2}{|n_1|^{3}}\frac{g_{n_3}(\omega)}{|n_3|^{\frac{3}{2}}}\frac{\cg_{n_3'}(\omega)}{|n_3'|^{\frac{3}{2}}}\right|^2 \quad \lesssim \quad  \mathcal{Q}_1 \,\, + \, \,  \mathcal{Q}_2,\end{equation} where 
\begin{eqnarray} \label{twosuma}
&& \mathcal{Q}_1:=  \quad \sum_{n_2\ne n_2'} \left|\sum_{S_{(n_2,n_2',\zeta)}}\frac{(|g_{n_1}(\omega)|^2 -1) }{|n_1|^{3}}\frac{g_{n_3}(\omega)}{|n_3|^{\frac{3}{2}}}\frac{\cg_{n_3'}(\omega)}{|n_3'|^{\frac{3}{2}}} \right|^2  \\
&& \mathcal{Q}_2:= \quad  \sum_{n_2\ne n_2'}  \left|\sum_{S_{(n_2,n_2',\zeta)}}\frac{1}{|n_1|^{3}}\frac{g_{n_3}(\omega)}{|n_3|^{\frac{3}{2}}}\frac{\cg_{n_3'}(\omega)}{|n_3'|^{\frac{3}{2}}}\right|^2  \label{twosumb} 
\end{eqnarray}
We estimate first $\mathcal{Q}_2$.  We rewrite, 
\begin{equation}\label{last2}
\mathcal{Q}_2 \, \sim \, \sum_{n_2\ne n_2'}\left|\sum_{n_3,n_3'}\left[\sum_{S_{(n_2,n_2',n_3,n_3',\zeta)}}\frac{1}{|n_1|^{3}} 
\frac{1}{|n_3|^{\frac{3}{2}}}\frac{1}{|n_3'|^{\frac{3}{2}}}\right]g_{n_3}(\omega)\cg_{n_3'}(\omega)\right|^2. 
\end{equation}

We now proceed as in the argument presented in \eqref{M1} above. We use \eqref{largedeviation} with $\lambda=\delta^{-r}\|F_2(\omega)\|_{L^2}$, Lemma \ref{independence}  and then use \eqref{sup-bound-all}, to obtain that for $\omega$ outside a set of measure $e^{-\frac{1}{\delta^r}}$ one has,
\begin{eqnarray}\label{probabilistic2}
\eqref{last2} \,& \lesssim& \, \delta^{-2r} \sum_{n_2\ne n_2'}\sum_{n_3,n_3'}\left[\sum_{S_{(n_2,n_2',n_3,n_3',\zeta)}}\frac{1}{|n_1|^{3}}
\frac{1}{|n_3|^{\frac{3}{2}}}\frac{1}{|n_3'|^{\frac{3}{2}}}\right]^2 \notag \\
&\lesssim& \delta^{-2r}N_1^{-6}N_3^{-6}\sum_{n_2\ne n_2'}\sum_{n_3, n_3'}|S_{(n_2,n_2', n_3, n_3',\zeta)}|^2 \notag\\
&\lesssim&\delta^{-2r}N_1^{-6}N_3^{-6}N_1\sum_{n_2\ne n_2'}\sum_{n_3, n_3'}|S_{(n_2,n_2', n_3, n_3',\zeta)}| \notag\\
&\lesssim &\delta^{-2r}N_1^{-6}N_3^{-6}N_1|S_{(\zeta)}|\sim \delta^{-2r}N_1^{-4}N_2^3.
\end{eqnarray}
To estimate $\mathcal{Q}_1$ we let 
\begin{equation}\label{n-determined} S_{(n_2,n_2', n_1, n_3, n_3',\zeta)}:= \left\{n \, : \,  \begin{aligned}&n_2=n_1+n-n_3, \quad n_2'=n_1+n-n_3',\\
&n_1\ne n_2, n_2', n_3, n_3', \quad n\in \tilde C\\
&\zeta = -|n_1|^2 + |n_3|^2, \quad \zeta= -|n_1|^2 +|n_3'|^2\end{aligned}  \right\}, 
\end{equation} and note that its cardinality is $1$ since $n$ is determined for fixed $(n_2,n_2', n_1, n_3,n_3')$. We have, 
\begin{equation*}
\mathcal{Q}_1 \sim\, \sum_{n_2\ne n_2'}\left|\sum_{n_1\ne n_2,n_2', n_3, \, n_3 \neq n_3'} \left[\sum_{S^2_{(n_2,n_2', n_1, n_3, n_3')}} \frac{1}{|n_1|^{3}} \frac{1}{|n_3|^{\frac{3}{2}}}\frac{1}{|n_3'|^{\frac{3}{2}}}\right] (|g_{n_1}(\omega)|^2 -1)  g_{n_3}(\omega)\cg_{n_3'}(\omega)\right|^2. 
\end{equation*}  Proceeding like above, we obtain in this case that  for $\omega$ outside a set of measure $e^{-\frac{1}{\delta^r}}$, 
\begin{equation*}
\mathcal{Q}_1 \, \lesssim \, \delta^{-2r} N_1^{-6}N_3^{-6}  |S_{(\zeta)}|\sim \delta^{-2r}N_1^{-5}N_2^3, 
\end{equation*} which is a better estimate. Hence all in all we obtain in this case that 
\begin{equation}\label{final-quote}M_2^2 \lesssim \delta^{-2r}N_1^{-4}N_2^3. \end{equation}

\medskip
{\bf Case $\beta_3$:} In this case we define first 
$$
S_{(n_2,n_2', n_1, n_1',\zeta)}=\left\{(n, n_3) \, : \,  \begin{aligned}n_2&=-n_1+n-n_3, \quad n_2'=n_1'+n-n_3,\\
 &n_3, n_2, n_2'\ne n_1, n_1', \quad n\in \tilde C\\
&\zeta= -|n_1|^2 + |n_3|^2, \quad \zeta= -|n_1'|^2 +|n_3|^2\end{aligned}\right\},
$$
with $|S_{(n_2,n_2', n_1, n_1',\zeta)}|\lesssim N_3^2$ by \eqref{ball} since $n$ is determined by $n_3$ and these ones lies on a sphere of radius at most $N_1$ intersection a ball of radius $N_3$.  If now  we define
$$
S_{(\zeta)}=\left\{(n_2,n_2',n, n_1, n_1', n_3) \, : \,  \begin{aligned} &n_2=-n_1+n-n_3, \quad n_2'=n_1'+n-n_3 ,\\
&n_3, n_2, n_2'\ne n_1, n_1', \quad n\in \tilde C\\
&\zeta= -|n_1|^2 + |n_3|^2, \quad \zeta= -|n_1'|^2 +|n_3|^2\end{aligned} \right\},
$$
then   $|S_{(\zeta)}|\lesssim N_1^2N_3^3N_2^3,$  since again $n$ ranges in a cube of size $N_2$ and we use \eqref{sphere} to count $n_1$ and $n_1'$.  We follow the argument  used above in \eqref{quote}-\eqref{final-quote} to bound $M_2^2$ but now with the couple $(n_1, n_1')$ instead and corresponding sums $\mathcal{Q}_1$ and $\mathcal{Q}_2$. Just as in Case $\beta_2$ above, the bound for $\mathcal{Q}_2$ is larger. We then obtain for $\omega$ outside a set of measure $e^{-\frac{1}{\delta^r}}$, 
\begin{eqnarray*}
M_2^2&\lesssim& \delta^{-2r}N_1^{-6}N_3^{-6}\sum_{n_2\ne n_2'}\sum_{n_1, n_1'}|S_{(n_2,n_2', n_1, n_1',\zeta)}|^2\\
&\lesssim&\delta^{-2r}N_1^{-6}N_3^{-6}N_3^2\sum_{n_2\ne n_2'}\sum_{n_1, n_1'}|S_{(n_2,n_2', n_1, n_1',\zeta)}|\\
&\lesssim &\delta^{-2r}
N_1^{-6}N_3^{-6}N_3^2|S_{(\zeta)}|\sim \delta^{-2r}N_1^{-4}N_3^{-1}N_2^3.
\end{eqnarray*}

\medskip
{\bf Case $\beta_4$:} In this case note that $N_1\sim N_3\sim N_2$. We define the two sets. First  
$$
S_{(n_2,n_2', n_1, n_3',\zeta)}=\left\{(n, n_3) \, : \,  \begin{aligned}&n_2=n_1+n-n_3, \quad n_2'=n_3+n-n_3',\\
&n_2, n_2', n_3, n_3'\ne n_1, \quad n\in \tilde C\\
&\zeta = -|n_1|^2 +|n_3|^2, \quad \zeta= |n_3|^2 +|n_3'|^2\end{aligned} \right\}
$$
and since $n_3$ lives on a sphere of radius at most $N_1$, from \eqref{sphere} we have  $|S_{(n_2,n_2', n_1, n_3',\zeta)}|\lesssim N_1$ and then
$$
S_{(\zeta)}=\left\{(n_2,n_2',n, n_1, n_3', n_3) \, : \,  \begin{aligned}&n_2=n_1+n-n_3, \quad n_2'=n_3+n-n_3',\\
&n_2, n_2', n_3,n_3'\ne n_1, \quad n\in \tilde C\\
&\zeta=-|n_1|^2 +|n_3|^2, \quad \zeta= -|n_3|^2 +|n_3'|^2\end{aligned} \right\},
$$
with  $|S_{(\zeta)}|\lesssim N_1N_2^3N_3^6.$  Just as in case $\beta_3$ and following the argument in \eqref{quote}-\eqref{final-quote} but with the couple $(n_1, n_3')$ instead 
we obtain that for $\omega$ outside a set of measure $e^{-\frac{1}{\delta^r}}$, 
\begin{eqnarray*}
M_2^2&\lesssim& \delta^{-2r}N_1^{-6}N_3^{-6}\sum_{n_2\ne n'_2}\sum_{n_1, n_3'}|S_{(n_2,n_2', n_1, n_3',\zeta)}|^2\\
&\lesssim&\delta^{-2r}N_1^{-6}N_3^{-6}N_1\sum_{n_2\ne n_2'}\sum_{n_1, n_3'}|S_{(n_2,n_2', n_1, n_3',\zeta)}|\\
&\lesssim &\delta^{-2r}
N_1^{-6}N_3^{-6}N_1|S_{(\zeta)}|\sim \delta^{-2r}N_1^{-4}N_2^{-3}.
\end{eqnarray*}

\medskip
{\bf Case $\beta_5$:} By symmetry this case is exactly the same as Case $\beta_4$.

\medskip
We are now ready to put all the estimates above together and bound $\mathcal{T}$ in cases $\beta_1 -\beta_5$:
\begin{eqnarray*}
\mathcal{T}&\lesssim& \|a_{n_2}\|_{\ell^2}^2N_1N_2  \|\GG\GG^*\|\lesssim \|a_{n_2}\|_{\ell^2}^2N_1N_2 (M_1+M_2)\\
&\lesssim&\|a_{n_2}\|_{\ell^2}^2\delta^{-2r}N_1N_2N_1^{-2}N_2^\frac{3}{2}\lesssim\delta^{-2r} N_2^\frac{5}{2}N_1^{-1}\|a_{n_2}\|_{\ell^2}^2.
\end{eqnarray*}

\medskip
{\bf Case $\beta_6$:} In this case 
\begin{equation*}
S_{(n_2,n_2',\zeta)}=\left\{(n, n_1, n_3) \, : \,  \begin{aligned}&n_2=n_1+n-n_3, \quad n_2'=n_3+n-n_1,\\
&n_1\ne n_2, n_2', n_3, \quad |n_1|^2 =|n_3|^2, \quad n\in \tilde C\end{aligned} \right\}.
\end{equation*}
At this point notice that the summation on $\zeta$ is eliminated and that in this case $N_1\sim N_2\sim N_3$.
We  have $S_{(n_2,n_2',\zeta)}\sim N_3^4$.  Using \eqref{sup-bound-all} we have, for $\omega$ outside a set of measure $e^{-\frac{1}{\delta^r}}$, that
\begin{eqnarray}\label{beta6}
M_2^2&=&\sum_{n_2\ne n_2'}\left|\sum_{n\in \tilde C}\sigma_{n_2, n}\overline \sigma_{n_2', n}\right|^2\sim
\sum_{n_2\ne n_2'}\left|\sum_{S_{(n_2,n_2',\zeta)}}\frac{|g_{n_1}(\omega)|^2}{|n_1|^{3}}\frac{|g_{n_3}(\omega)|^2}{|n_3|^{3}}\right|^2\\\notag
&\lesssim&\sum_{n_2\ne n_2'}N_1^{-6+\varepsilon}N_3^{-6}|S_{(n_2,n_2',\zeta)}|^2 \lesssim N_1^{-6+\varepsilon}N_3^{-6}N_3^4|S_{(\zeta)}|,
\end{eqnarray}
where 
$$S_{(\zeta)}=\left\{(n_2, n_2', n, n_1, n_3) \, :  \begin{aligned}&n_2=n_1+n-n_3, \quad n_2'=n_3+n-n_1,\\
& n_1\ne n_3, n_2, n_2'\quad |n_1|^2 =|n_3|^2, \quad n\in \tilde C\end{aligned}\right\}$$
and $|S_{(\zeta)}|\lesssim N_2^3N_3^4$. 
Hence $M_2\lesssim N_1^{-3+\varepsilon}N_2^{\frac{5}{2}}$ and as a consequence 
$$\mathcal{T}\lesssim \|a_{n_2}\|_{\ell^2}^2N_1^{-3+\varepsilon}N_2^{\frac{5}{2}}.$$

\medskip

We now notice that to prove \eqref{barRDbarR} we first have to consider the case when $n_1=n_3$,  which here  it is not excluded,  and then  we can use  exactly the same argument as above since a plus or minus sign in front of $n_3$ does not  change any of the counting.  

Consider now \eqref{barRDbarR}  with $n_1=n_3$. Note that $N_1\sim N_2\sim N_3$. We now have 
\begin{equation}\label{T1barRDRn1=n3}
\mathcal{T}:=  \sum_{m \in \Z, n \in \Z^3} \,\left|   \sum_{\substack{ n= -2n_1+ n_2 \\
\\ m = -2|n_1|^2 +|n_2|^2}} \frac{(\cg_{n_1}(\omega))^2}{|n_1|^3}
a_{n_2}  \right|^2.
\end{equation}
Let $S_{(m,n)}=\{(n_1,n_2) \, / \, n= -2n_1+ n_2, \,  m = -2|n_1|^2 +|n_2|^2\}$, and note that 
$|S_{(m,n)}|\lesssim N_1$. Then 
$$\mathcal{T}\lesssim N_1\sum_{m, n}   \sum_{S_{(m,n)}} \frac{|\cg_{n_1}(\omega)|^4}{|n_1|^6}
|a_{n_2}|^2\sim N_1\sum_{ n, n_1 \in \Z^3}   \frac{|\cg_{n_1}(\omega)|^4}{|n_1|^6}
|a_{n+2n_1}|^2\lesssim N_1^{-2+\varepsilon}\|a_{n_2}\|_{l^2}^2,$$
where we use \eqref{sup-bound-all} for $\omega$ outside a set of measure $e^{-\frac{1}{\delta^r}}$.
\end{proof}

\bigskip

 \begin{proposition}{\label{Two}}   Let  $D_j$ and $R_k$ be as above and fix $N_1\geq N_2\geq N_3$, $r, \, \delta>0$ and $C \in \mathcal{C}_{N_2}$. Then there exists $\mu >0$ and a set $\Omega_\delta\in A$ such that $\pr(\Omega_\delta^c)\leq 
e^{-\frac{1}{\delta^r}}$ such that  for any $\omega\in \Omega_\delta$ we have $\eqref{DbarRR}$ and $  \eqref{DRR} $.

\end{proposition}

%
%

\begin{proof}
 We start by estimating \eqref{DbarRR} where without any loss of generality we assume that $\tilde D_1=D_1$. We now have, 
\begin{equation}\label{T1DbarRR}
\mathcal{T}:=  \sum_{m\in \Z, n \in \Z^3} \,\left|   \sum_{\substack{ n= n_1-n_2 + n_3 \\n_2\ne n_3
\\ m = |n_1|^2 -|n_2|^2+|n_3|^2}} \chi_{C}(n_1)a_{n_1}\frac{\cg_{n_2}(\omega)}{|n_2|^\frac{3}{2}}
\frac{g_{n_3}(\omega)}{|n_3|^{\frac{3}{2}}} \right|^2.
\end{equation}
We are going to use duality and change of variables with $\zeta := m - |n_1|^2= -|n_2|^2 + |n_3|^2$ again. Note though that if $n_1$ is in a cube $C$ of size $N_2$ then also $n$ will be in a cube $\tilde C$ of approximately  the same size. Then just as in \eqref{t1dual} we need to estimate
$$
 \|\chi_Ca_{n_1}\|_{\ell^2}^2\|\gamma\|_{\ell^2}^2\sum_{n_1} \sum_{|\zeta| \leq N_2^2}  \left| \sum_{n} \sigma_{n_1, n}  \chi_{\tilde C}(n)k_n \right|^2, 
$$ 
where  
$$ \sigma_{n_1, n} =  \sum_{\substack{ n_1=n_2+n - n_3, \, n_2\ne n_1, n_3\\    \zeta= -|n_2|^2 +|n_3|^2}}  \frac{\cg_{n_2}(\omega)}{|n_2|^{\frac{3}{2}}} \frac{g_{n_3}(\omega)}{|n_3|^{\frac{3}{2}}}.$$
If we denote by $\GG$  the matrix of entries $ \sigma_{n_1, n}$, and we recall  that the variation in $\zeta$ is at most $N_2^2$,  we are then reduced to estimating
$$
 \|\chi_Ca_{n_1}\|_{\ell^2}^2N_2^2  \|\GG\GG^*\|. 
$$
We note that by Lemma \ref{matrixnorm}, 
$$\|\GG\GG^*\|\lesssim \max_{n_1}\sum_{n}|\sigma_{n_1, n}|^2+\left(\sum_{n_1\ne n_1'}\left|\sum_{n\in\tilde C }\sigma_{n_1, n}\overline{\sigma}_{n_1', n}\right|^2\right)^\frac{1}{2}\,=:\,M_1+M_2,$$
where
$\tilde C$ is a cube of side length approximately $N_2$.
From this point on the proof is similar to the one already provided for \eqref{barRDR} where $n_2$ is replaced by $n_1$.  
We still go through the argument though, since the size of $n_1$ and $n_2$ are different.


To estimate $M_1$ we first define the set 
$$
S_{(\zeta,n_1)}=\{(n_2,n,n_3) \, : \, n_2 \neq n_1, n_3, \quad n_2=n_1-n +n_3, \quad  \zeta= -|n_2|^2 +|n_3|^2 \}.$$ Applying \eqref{sphere} for each fixed $n_3$,  we have that $|S_{(\zeta,n_1)}|\lesssim N_3^3 N_2 $
since $n_2$ sits on a sphere of radius approximately $ N_2$ . Then  we proceed as in \eqref{M1} to obtain for $\omega$ outside a set of measure $e^{-\frac{1}{\delta^r}}$, the bound 
$$M_1\lesssim  \delta^{-r}N_2^{-3 }  N_3^{-3} N_3^3N_2\sim \delta^{-2r}N_2^{-2 }.$$
To estimate $M_2$ we first write
\begin{eqnarray*}
M_2^2&=&\sum_{n_1\ne n_1'}\left|\sum_{n\in \tilde C}\sigma_{n_1, n}\overline \sigma_{n_1', n}\right|^2\sim
\sum_{n_1\ne n_1'}\left|\sum_{S_{(n_1,n_1',\zeta)}}\frac{\cg_{n_2}(\omega)}{|n_2|^{\frac{3}{2}}}\frac{g_{n_3}(\omega)}{|n_3|^{\frac{3}{2}}}\frac{g_{n_2'}(\omega)}{|n_2'|^{\frac{3}{2}}}\frac{\cg_{n_3'}(\omega)}{|n_3'|^{\frac{3}{2}}}\right|^2
\end{eqnarray*}
where
\begin{equation*}
S_{(n_1,n_1',\zeta)}=\left\{(n, n_2, n_3,n_2', n_3') \, : \,  \begin{aligned}
&n_2=n_1-n+n_3, \quad n_2'=n_1'-n+n_3',\\
& n_2 \neq n_1, n_3, \quad n_2' \neq n_1', n_3', \quad n\in \tilde C\\
&\zeta= -|n_2|^2 +|n_3|^2, \quad \zeta= -|n_2'|^2 +|n_3'|^2\end{aligned}\right\}.
\end{equation*}

We organize once again the estimates according to whether some frequencies are the same or not. As before, all in all we have six cases:
\begin{itemize}
\item {\bf Case $\beta_1$:} $n_2, n_2', n_3, n_3'$ are all different. 
\item {\bf Case $\beta_2$:} $n_2=n_2'; \, \, n_3\ne n_3'.$
\item {\bf Case $\beta_3$:} $n_2\ne n_2'; \, \, n_3= n_3'.$
\item {\bf Case $\beta_4$:} $n_2\ne n_3'; \, \, n_3 =  n_2'.$
\item {\bf Case $\beta_5$:} $n_2= n_3'; \, \, n_3\ne  n_2'.$
\item {\bf Case $\beta_6$:} $n_2= n_3'; \, \, n_3 =  n_2'.$ 
\end{itemize}

\medskip
{\bf Case $\beta_1$:} We define the set  
\begin{equation*}
S_{(\zeta)}=\left\{(n_1,n_1',n, n_2, n_3,n_2', n_3') \, : \,  \begin{aligned}&n_2=n_1-n+n_3, \quad n_2'=n_1'-n+n_3' \\
&n_2 \neq n_1, n_3, \quad n_2' \neq n_1', n_3', \quad \, n_1, n_1' \in C\\
&\zeta = -|n_2|^2 +|n_3|^2 , \quad \zeta= -|n_2'|^2 +|n_3'|^2\end{aligned}\right\}.
\end{equation*}
and note that  $|S_{(\zeta)}|\lesssim N_2^2N_3^6N_2^3$   by Lemma \ref{lattice-counting} since for $n_3$ fixed,  $n_2$ and $n_2'$ sit on sphere of radius $\sim N_2$ and $n\in \tilde C$ a cube of side length approximately $N_2$.  Hence, for $\omega$ outside a set of measure $e^{-\frac{1}{\delta^r}}$, we obtain in this case,  
$$M_2^2\lesssim \delta^{-4r}N_2^{-6}N_3^{-6}N_2^2N_3^6N_2^3\sim \delta^{-4r}N_2^{-1}.$$

\medskip
{\bf Case $\beta_2$:} In this case we define two sets. We start with  
\begin{equation*}
S_{(n_1,n_1', n_3, n_3', \zeta)}=\left\{(n, n_2) \, : \,  \begin{aligned}& n_2=n_1-n+n_3, \quad n_2=n_1'-n+n_3' ,\\
 &n_2 \neq n_1, n_1', n_3, n_3',  \quad n\in \tilde C \\
&\zeta = -|n_2|^2 + |n_3|^2, \quad \zeta= -|n_2|^2 +|n_3'|^2\end{aligned} \right\}.
\end{equation*}
To compute  $|S_{(n_1,n_1', n_3, n_3',\zeta)}|$, it is enough to count $n_2$, then $n$ is determined.  Since $n_2$ sits on a sphere of radius $\sim N_2$ we have by \eqref{sphere} that   $|S_{(n_1,n_1', n_3, n_3',\zeta)}|\lesssim N_2$. Then we set 
\begin{equation*}
S_{(\zeta)}=\left\{(n_1,n_1',n, n_2, n_3, n_3') \, : \,  \begin{aligned}&n_2=n_1-n+n_3, \quad n_2=n_1'-n+n_3' ,\\
& n_2 \neq n_1, n_1', n_3, \, n_3', \quad  n\in \tilde C\\
&\zeta= -|n_2|^2 + |n_3|^2, \quad \zeta= -|n_2|^2 +|n_3'|^2\end{aligned}\right\}
\end{equation*}
for which   $|S_{(\zeta)}|\lesssim N_2N_3^6N_2^3,$  where we used again that $n\in \tilde C$. Arguing as in \eqref{quote}- \eqref{final-quote}, we then have for $\omega$ outside a set of measure $e^{-\frac{1}{\delta^r}}$ that 
\begin{eqnarray*}
M_2^2&\lesssim& \delta^{-2r}N_2^{-6}N_3^{-6}\sum_{n\ne n'}\sum_{n_3, n_3'}|S_{(n_1,n_1', n_3, n_3',\zeta)}|^2\\
&\lesssim&\delta^{-2r}N_2^{-6}N_3^{-6}N_2\sum_{n_1\ne n_1'}\sum_{n_3, n_3'}|S_{(n_1,n_1', n_3, n_3',\zeta)}|\\
&\lesssim &\delta^{-2r}N_2^{-6}N_3^{-6}N_2|S_{(\zeta)}|
\sim\delta^{-2r} N_2^{-1}.
\end{eqnarray*}

\medskip
{\bf Case $\beta_3$:} In this case we define first 
\begin{equation*}
S_{(n_2,n_2', n_1, n_1',\zeta)}=\left\{(n, n_3) \, : \,  \begin{aligned}&n_2=n_1-n+n_3, \quad n_2'=n_1'-n+n_3 ,\\
&n_2, n_2'\neq n_3, n_1, n_1',  \quad n\in \tilde C\\
&\zeta= -|n_2|^2 + |n_3|^2, \quad \zeta= -|n_2'|^2 +|n_3|^2\end{aligned} \right\}
\end{equation*}
for which  $|S_{(n_2,n_2', n_1, n_1',\zeta)}|\lesssim N_3^2$ since $n$ is determined by $n_3$ and this one lies on a sphere of radius at most $N_1$ intersection a ball of radius $N_3$ (see Lemma \ref{lattice-counting}). Then we define 
\begin{equation*}
S_{(\zeta)}=\left\{(n_2,n_2',n, n_1, n_1', n_3) \, : \,  \begin{aligned} &n_2=n_1-n+n_3, \quad n_2'=n_1'-n+n_3 ,\\
& n_2, n_2'\neq n_3, n_1, n_1', \quad n\in \tilde C\\
&\zeta= -|n_2|^2 + |n_3|^2, \quad \zeta= -|n_2'|^2 +|n_3|^2\end{aligned} \right\}
\end{equation*}
for which   $|S_{(\zeta)}|\lesssim N_2^2N_3^3N_2^3,$  since again $n$ ranges in a cube of size $N_2$. We then have, as usual using \eqref{largedeviation} and \eqref{sup-bound-all} as above that for $\omega$ outside a set of measure $e^{-\frac{1}{\delta^r}}$,
\begin{eqnarray*}
M_2^2&\lesssim& \delta^{-2r}N_2^{-6}N_3^{-6}\sum_{n_1\ne n_1'}\sum_{n_2, n_2'}|S_{(n_2,n_2', n_1, n_1',\zeta)}|^2\\
&\lesssim&\delta^{-2r}N_2^{-6}N_3^{-6}N_3^2\sum_{n_1\ne n_1'}\sum_{n_2, n_2'}|S_{(n_2,n_2', n_1, n_1',\zeta)}|\\
&\lesssim &\delta^{-2r}
N_2^{-6+\varepsilon}N_3^{-6}N_3^2|S_{(\zeta)}|\sim \delta^{-2r}N_2^{-1+\varepsilon}N_3^{-1}.
\end{eqnarray*}

\medskip
{\bf Case $\beta_4$:} In this case note that $N_3\sim N_2$. We define the two sets 
\begin{equation*}
S_{(n_1,n_1', n_2, n_3',\zeta)}=\left\{(n, n_3) \, : \,  \begin{aligned}&n_2=n_1 - n+n_3, \quad n_3=n_1'-n+n_3' ,\\
&n_2 \neq n_1, n_3; n_3\ne n_3', n_1',\,   \quad n\quad n\in \tilde C\\
&\zeta = -|n_2|^2 +|n_3|^2, \quad \zeta= -|n_3|^2 +|n_3'|^2\end{aligned} \right\}
\end{equation*}
with $|S_{(n_1,n_1', n_2, n_3',\zeta)}|\lesssim N_2$ since $n_3$ lives on a sphere of radius at most $N_2$; and 
\begin{equation*}
S_{(\zeta)}=\left\{(n_1,n_1',n, n_2, n_3', n_3) \, : \,  \begin{aligned}&n_2=n_1- n + n_3, \quad n_3=n_1'-n+n_3' ,\\
&n_2 \neq n_1, n_3; n_3\ne n_3', n_1',\,  
 \quad n\quad n\in \tilde C\\
&\zeta=-|n_2|^2 +|n_3|^2, \quad \zeta= -|n_3|^2 +|n_3'|^2\end{aligned} \right\}
\end{equation*}
with  $|S_{(\zeta)}|\lesssim N_2N_2^3N_3^6$  since for fixed $n_3, n_3'$, the frequencies  $n_2$ sit on a sphere of radius at most $N_2$ and $n\in \tilde C$ (see Lemma \ref{lattice-counting}). We then have as above that for $\omega$ outside a set of measure $e^{-\frac{1}{\delta^r}}$, 
\begin{eqnarray*}
M_2^2&\lesssim& \delta^{-2r}N_2^{-6}N_3^{-6}\sum_{n\ne n'}\sum_{n_2, n_3'}|S_{(n_1,n_1', n_2, n_3',\zeta)}|^2\\
&\lesssim&\delta^{-2r}N_2^{-6}N_3^{-6}N_2\sum_{n\ne n'}\sum_{n_2, n_3'}|S_{(n_1,n_1', n_2, n_3',\zeta)}|\\
&\lesssim &\delta^{-2r}
N_2^{-6}N_3^{-6}N_2|S_{(\zeta)}|\sim \delta^{-2r}N_2^{-1}.
\end{eqnarray*}

\medskip
{\bf Case $\beta_5$:} By symmetry this case is exactly the same as Case $\beta_4$.

\medskip
We are now ready put all the estimates together and bound $\mathcal{T}$ in cases $\beta_1 -\beta_5$:
\begin{eqnarray*}
\mathcal{T}&\lesssim& \|\chi_C a_{n_1}\|_{\ell^2}^2N_2^2  \|\GG\GG^*\|\lesssim \|a_{n_1}\|_{\ell^2}^2N_2^2 (M_1+M_2)\\
&\lesssim&\|\chi_C a_{n_1}\|_{\ell^2}^2\delta^{-2r}N_2^2N_2^{-\frac{1}{2}}\sim \|\chi_C a_{n_1}\|_{\ell^2}^2 \delta^{-2r}N_2^{\frac{3}{2}} .
\end{eqnarray*}

\medskip
{\bf Case $\beta_6$:} In this case 
\begin{equation*}
S_{(n_1,n_1',\zeta)}=\left\{(n, n_2, n_3) \, : \,  \begin{aligned}n_2&=n_1-n+n_3, \quad n_3=n_1'-n+n_2 ,\\
&n_2 \neq n_3, n_1, \quad  |n_2|^2 =|n_3|^2,  \quad n\in \tilde C\end{aligned}\right\}.
\end{equation*}
At this point notice that  $\Delta \zeta=1$ and that in this case $N_2\sim N_3$.
We  have $S_{(n_1,n_1',\zeta)}\sim N_3^4$. We then have as in \eqref{beta6}
$$M_2^2\lesssim N_2^{-6+\epsilon}N_3^{-6}N_3^4|S_{(\zeta)}|$$
where 
$$S_{(\zeta)}=\left\{(n_1, n_1', n, n_2, n_3) \, : \, \begin{aligned}n_2&=n_1-n+n_3, \quad n_3=n_1'-n + n_2 ,\\
&n_2 \neq n_3, n_1, \quad  |n_2|^2 =|n_3|^2,  \quad n\in \tilde C\end{aligned}\right\}$$
and $|S_{(\zeta)}|\lesssim N_2^3N_3^4$. Hence, all in all we have that for $\omega$ outside a set of measure $e^{-\frac{1}{\delta^r}}$, \, $M_2\lesssim N_2^{-\frac{1}{2}+\epsilon}$ and as a consequence,  
$$\mathcal{T} \, \lesssim \, \|\chi_C a_{n_1}\|_{\ell^2}^2N_2^{-\frac{1}{2}+\epsilon}$$ in this case, which is a better bound.

\medskip

To prove \eqref{DRR} we write
\begin{equation}\label{T1DRR}
\mathcal{T}:=  \sum_{m \in \Z, n \in \Z^3} \,\left|   \sum_{\substack{ n= n_1+n_2 + n_3 
\\ m = |n_1|^2 +|n_2|^2+|n_3|^2}} \chi_{C}(n_1)a_{n_1}\frac{g_{n_2}(\omega)}{|n_2|^\frac{3}{2}}
\frac{g_{n_3}(\omega)}{|n_3|^{\frac{3}{2}}} \right|^2.
\end{equation} 
We can repeat the argument above after checking the case $n_2=n_3$. In this case \eqref{T1DRR} becomes
$$\mathcal{T}=  \sum_{m \in \Z, n \in \Z^3} \,\left|   \sum_{\substack{ n= n_1+2n_2 
\\ m = |n_1|^2 +2|n_2|^2}} \chi_{C}(n_1)a_{n_1}\frac{(g_{n_2}(\omega))^2}{|n_2|^3}
 \right|^2.
$$
Let $S_{(m,n)}=\{(n_1,n_2) \, / \, n= n_1+ 2n_2, \,  m = |n_1|^2 +2|n_2|^2\}$, and note that by Lemma \ref{lattice-counting}, 
$|S_{(m,n)}|\lesssim \min(N_1,N_2^2)$. Then 
\begin{eqnarray*} \mathcal{T} &\lesssim&  \min(N_1,N_2^2)\sum_{m , n }   \sum_{S_{(m,n)}} \frac{|g_{n_2}(\omega)|^4}{|n_2|^6}
|\chi_C a_{n_1}|^2\\
&\sim&  \min(N_1,N_2^2)\sum_{ n, n_1}   \frac{|g_{\frac{n-n_1}{2}}(\omega)|^4}{|\frac{n-n_1}{2}|^6}
|\chi_C a_{n_1}|^2 \, \lesssim  \, \min(N_1,N_2^2)N_2^{-3+\varepsilon}\|\chi_{C} a_{n_1}\|_{l^2}^2, 
\end{eqnarray*} where we used \eqref{sup-bound-all} for $\omega$ outside a set of measure $e^{-\frac{1}{\delta^r}}$.
\end{proof}

\bigskip

 \begin{proposition}{\label{Three}}  Let  $D_j$ and $R_k$ be as above and fix $N_1\geq N_2\geq N_3$, $r, \, \delta>0$ and $C \in \mathcal{C}_{N_2}$. Then there exists $\mu>0$ and a set $\Omega_\delta\in A$ such that $\pr(\Omega_\delta^c)\leq 
e^{-\frac{1}{\delta^r}}$ such that  for any $\omega\in \Omega_\delta$ we have  \eqref{barRRD} and \, \eqref{barRbarRD}.
 
  \end{proposition}

\begin{proof} Without loss of generality we assume that $\tilde D_3=D_3$. We write, 
\begin{equation}\label{T1barRRD}
\mathcal{T}:=  \sum_{m \in Z, n \in \Z^3} \,\left|   \sum_{\substack{ n= -n_1+ n_2 + n_3, \\n_1\ne n_2, n_3
\\ m = -|n_1|^2 +|n_2|^2+|n_3|^2}} \chi_{C}(n_1)\frac{\cg_{n_1}(\omega)}{|n_1|^\frac{3}{2}}
\frac{g_{n_2}(\omega)}{|n_2|^{\frac{3}{2}}} \, a_{n_3} \right|^2,
\end{equation}
where $C \in \mathcal{C}_{N_2}$. Let us now define 
$$ \sigma_{n, n_3} =  \sum_{\substack{ n= -n_1+n_2+n_3,\, \,  n_1\ne n_2, n_3\\  m = -|n_1|^2 +|n_2|^2+ |n_3|^2  }}  \chi_{C}(n_1)\frac{\cg_{n_1}(\omega)}{|n_1|^{\frac{3}{2}}} \frac{g_{n_2}(\omega)}{|n_2|^{\frac{3}{2}}}.$$

If we denote by $\GG$ the matrix of entries $ \sigma_{n, n_3}$, by using that the variation in $m$ is at most $N_1N_2$  we can then continue the estimate of $\mathcal{T}$ in \eqref{T1barRRD} by
$$
\mathcal{T}\lesssim \|a_{n_3}\|_{\ell^2}^2N_1N_2  \|\GG\GG^*\|. 
$$
Once again by Lemma \ref{matrixnorm}, 
$$\|\GG\GG^*\|\lesssim \max_{n}\sum_{n_3 }|\sigma_{n, n_3}|^2+\left(\sum_{n\ne n'}\left|\sum_{n_3}\sigma_{n, n_3}\overline{\sigma}_{n', n_3}\right|^2\right)^\frac{1}{2}=: M_1+M_2.$$
%
To estimate $M_1$ we first define the set 
$$
S_{(m,n)}=\{(n_1,n_2,n_3) \, :  n_1 \neq n_2, n_3, \quad \, n= -n_1+n_2+n_3, \quad  m = -|n_1|^2 +|n_2|^2+ |n_3|^2 \}.$$
By \eqref{planepi} we have that $ |S_{(m,n)}|\lesssim N_3^3N_2^2$ 
since once fixed $n_3$  we use $m = -|n_2+n_3-n|^2 +|n_2|^2+ |n_3|^2$ to count $n_2$ which lives 
on the intersection of a plane with a ball of radius $N_2$.  Then as in \eqref{M1} we have, for $\omega$ outside a set of measure $e^{-\frac{1}{\delta^r}}$,  that 
\begin{eqnarray*}M_1&\lesssim&  \sup_{n,m}\, \sum_{n_3}\left|\sum_{\substack{ n= -n_1+n_2+n_3, \, \, n_1\ne n_2, n_3\,  \\  m = -|n_1|^2 +|n_2|^2+ |n_3|^2  }}  \frac{\cg_{n_1}(\omega)}{|n_1|^{\frac{3}{2}}} \chi_{C}(n_1)\frac{g_{n_2}(\omega)}{|n_2|^{\frac{3}{2}}}\right|^2\\
&\lesssim& \sup_{n,m}\, \delta^{-r}N_1^{-3}N_2^{-3}|S_{(m,n)}|\lesssim 
\delta^{-2r}N_1^{-3}  N_2^{-3} N_3^3N_2^2\sim \delta^{-2r}N_1^{-3 }  N_2^{-1} N_3^3 .
\end{eqnarray*}

To estimate $M_2$ we first write
\begin{eqnarray*}
M_2^2&=&\sum_{n\ne n'}\left|\sum_{n_3}\sigma_{n, n_3}\overline \sigma_{n', n_3}\right|^2\sim
\sum_{n\ne n'}\left|\sum_{S_{(n,n',m)}}\frac{\cg_{n_1}(\omega)}{|n_1|^{\frac{3}{2}}}\frac{g_{n_2}(\omega)}{|n_2|^{\frac{3}{2}}}\frac{g_{n_1'}(\omega)}{|n_1'|^{\frac{3}{2}}}\frac{\cg_{n_2'}(\omega)}{|n_2'|^{\frac{3}{2}}}\right|^2,
\end{eqnarray*}
where
\begin{equation*}
S_{(n,n',m)}=\left\{(n_3, n_1, n_2,n_1', n_2') \, : \,  \begin{aligned}
&n=-n_1+n_2+n_3, \quad n'=-n_1'+n_2'+n_3,\\
&n_1 \neq n_2, n_3, \quad n_1'\neq n_2', n_3, \quad n_1, n_1' \in C\\
&m = -|n_1|^2 +|n_2|^2+|n_3|^2, \quad m= -|n_1'|^2 +|n_2'|^2+|n_3|^2\end{aligned}\right\}.
\end{equation*}
Just like for the proof of \eqref{barRDR} we need to organize the estimates according to whether some frequencies are the same or not, in all we have six cases.
\begin{itemize}
\item {\bf Case $\beta_1$:} $n_1, n_1', n_2, n_2'$ are all different. 
\item {\bf Case $\beta_2$:} $n_1=n_1'; \, \, n_2\ne n_2'.$
\item {\bf Case $\beta_3$:} $n_1\ne n_1'; \, \, n_2= n_2'.$
\item {\bf Case $\beta_4$:} $n_1\ne n_2'; \, \, n_2 =  n_1'.$
\item {\bf Case $\beta_5$:} $n_1= n_2'; \, \, n_2\ne  n_1'.$
\item {\bf Case $\beta_6$:} $n_1= n_2'; \, \, n_2 =  n_1'.$ 
\end{itemize}

\medskip
{\bf Case $\beta_1$:} In this case we let 
\begin{equation*}
S_{(m)}=\left\{(n,n',n_3, n_1, n_2,n_1', n_2') \, : \,  \begin{aligned} &n=-n_1+n_2+n_3, \quad n'=-n_1'+n_2'-n_3,\\
&n_1 \neq n_2, n_3, \quad n_1'\neq n_2', n_3, \quad n_1, n_1' \in C \\
&m = -|n_1|^2 +|n_2|^2+|n_3|^2 , \quad m= -|n_1'|^2 +|n_2'|^2+|n_3|^2\end{aligned}\right\}
\end{equation*}
with  $|S_{(m)}|
\lesssim N_1^2N_2^6N_3^3.$ As in the argument we use for \eqref{M1}, this gives that for $\omega$ outside a set of measure $e^{-\frac{1}{\delta^r}}$, 
$$M_2^2\lesssim \delta^{-4r}N_1^{-6}N_2^{-6}N_1^2N_2^6N_3^3\sim \delta^{-4r}N_1^{-4}N_3^3.$$

\medskip
{\bf Case $\beta_2$:} In this case we define two sets. We start with  
\begin{equation*}
S_{(n,n', n_2, n_2',m)}=\left\{(n_3, n_1) \, : \,  \begin{aligned}&n=-n_1+n_2+n_3, \quad n'=-n_1+n_3+n_2' ,\\
&n_1 \neq n_2, n_2', n_3 \\ 
&m = -|n_1|^2 +|n_2|^2+ |n_3|^2, \quad m= -|n_1|^2 +|n_3|^2+|n_2'|^2\end{aligned}\right\}.
\end{equation*}
To compute  $|S_{(n,n', n_2, n_2',m)}|$ we  count $n_3$  then $n_1$ is determined. Since $n_3$ sit on a plane we have  by \eqref{planepi} that $|S_{(n,n', n_2, n_2',m)}|
\lesssim N_3^2$. Then we set 
\begin{equation*}
S_{(m)}=\left\{(n,n',n_3, n_1, n_2, n_2') \, : \,  \begin{aligned}&n=-n_1+n_2+n_3, \quad n'=-n_1+n_3+n_2' ,\\
&n_1 \neq n_2, n_2', n_3 \\ 
&m = -|n_1|^2 +|n_2|^2+ |n_3|^2, \quad m= -|n_1|^2 +|n_3|^2+|n_2'|^2\end{aligned}\right\}
\end{equation*}
for which  $|S_{(m)}|\lesssim N_1N_2^6N_3^3.$  We then have following the argument in \eqref{quote}-\eqref{final-quote} that for $\omega$ outside a set of measure $e^{-\frac{1}{\delta^r}}$, 
\begin{eqnarray*}
M_2^2&\lesssim& \delta^{-2r}N_1^{-6}N_2^{-6}\sum_{n\ne n'}\sum_{n_2, n_2'}|S_{(n,n', n_2, n_2',m)}|^2\\
&\lesssim&\delta^{-2r}N_1^{-6}N_2^{-6}N_3^2\sum_{n\ne n'}\sum_{n_2, n_2'}|S_{(n,n', n_2, n_2',m)}|\\
&\lesssim &\delta^{-2r}N_1^{-6}N_2^{-6}N_3^2|S_{(m)}|
\sim \delta^{-2r}N_1^{-5}N_3^5.
\end{eqnarray*}

\medskip
{\bf Case $\beta_3$:} In this case we define first 
\begin{equation*}
S_{(n,n', n_1, n_1',m)}=\left\{(n_2, n_3) \, : \,  \begin{aligned}&n=-n_1+n_2+n_3, \quad n'=-n_1'+n_2+n_3 ,\\
&n_2, n_3 \neq n_1, n_1', \quad n_1, n_1' \in C \\ 
&m = -|n_1|^2 +|n_2|^2+ |n_3|^2, \quad m= -|n_1'|^2 +|n_2|^2+|n_3|^2\end{aligned}\right\}
\end{equation*}
with $|S_{(n,n', n_1, n_1',m)}|\lesssim N_3^2$ since $n_2$ is determined by $n_3$ and this one lines on a sphere of radius at most $N_1$. On the other hand 
\begin{equation*}
S_{(m)}=\left\{(n,n',n_2, n_1, n_1', n_3) \, : \,  \begin{aligned}&n=-n_1+n_2+n_3, \quad n'= -n_1'+n_2+n_3,\\
&n_2, n_3 \neq n_1, n_1', \quad n_1, n_1' \in C \\ 
&m =-|n_1|^2 +|n_2|^2+ |n_3|^2 , \quad m=  -|n_1'|^2 +|n_2|^2+|n_3|^2\end{aligned}\right\}
\end{equation*}
with  $|S_{(m)}|\lesssim N_1^2N_3^3N_2^3.$  Hence arguing as above we  have 
\begin{eqnarray*}
M_2^2&\lesssim& \delta^{-2r}N_1^{-6}N_2^{-6}\sum_{n\ne n'}\sum_{n_1, n_1'}|S_{(n,n', n_1, n_1',m)}|^2\\
&\lesssim&\delta^{-2r}N_1^{-6}N_2^{-6}N_3^2\sum_{n\ne n'}\sum_{n_1, n_1'}|S_{(n,n', n_1, n_1',m)}|\\
&\lesssim &\delta^{-2r}
N_1^{-6}N_2^{-6}N_3^2|S_{(m)}|
\sim\delta^{-2r} N_1^{-4}N_2^{-3}N_3^5, 
\end{eqnarray*} for $\omega$ outside a set of measure $e^{-\frac{1}{\delta^r}}$.

\medskip
{\bf Case $\beta_4$:} We define the two sets 
\begin{equation*}
S_{(n,n', n_1, n_2',m)}=\left\{(n_2, n_3) \, : \,  \begin{aligned}&n=-n_1+n_2+n_3, \quad n'=-n_2+n_2'+n_3,\\
&n_2,n_3 \neq n_1, n_2' \\ 
&m = -|n_1|^2 +|n_2|^2+|n_3|^2, \quad m= -|n_2|^2 +|n_2'|^2+|n_3|^2\end{aligned}\right\}
\end{equation*}
for which, since $n_3$ lives on a sphere of radius at most $N_1$, we have  $|S_{(n,n', n_1, n_2',m)}|\lesssim \min(N_1, N_3^2)$ and 
\begin{equation*}
S_{(m)}=\left\{(n,n',n_3, n_1, n_2', n_2) \, : \,  \begin{aligned}&n=-n_1+n_2+n_3, \quad n'= -n_2+n_2'+n_3,\\
&n_2,n_3 \neq n_1, n_2' \\ 
&m =-|n_1|^2 +|n_2|^2+|n_3|^2, \quad m= -|n_2|^2 +|n_2'|^2+|n_3|^2\end{aligned}\right\}
\end{equation*}
with  $|S_{(m)}|\lesssim N_1N_3^3N_2^6.$  We then have in this case that 
\begin{eqnarray*}
M_2^2&\lesssim& \delta^{-2r}N_1^{-6}N_2^{-6}\sum_{n\ne n'}\sum_{n_1, n_2'}|S_{(n,n', n_1, n_2',m)}
|^2\\
&\lesssim&\delta^{-2r}N_1^{-6}N_2^{-6}\min(N_1,N_3^2)\sum_{n\ne n'}\sum_{n_1, n_2'}|S_{(n,n', n_1, n_2',m)}|\\
&\lesssim &\delta^{-2r}
N_1^{-6}N_2^{-6}\min(N_1,N_3^2)|S_{(m)}|
\sim \delta^{-2r}N_1^{-4}N_3^{3}, 
\end{eqnarray*} for $\omega$ outside a set of measure $e^{-\frac{1}{\delta^r}}$.

\medskip

{\bf Case $\beta_5$:} This case is exactly the same as Case $\beta_4$.

\medskip

We are now ready to estimate $\mathcal{T}$ in cases $\beta_1-\beta_5$.
\begin{eqnarray*}
\,\, \mathcal{T}&\lesssim& \|a_{n_2}\|_{\ell^2}^2N_1N_2  \|\GG\GG^*\|\lesssim \|a_{n_2}\|_{\ell^2}^2N_1N_2 (M_1+M_2)\\
&\lesssim&\|a_{n_2}\|_{\ell^2}^2\delta^{-4r}[N_1N_2(N_1^{-\frac{5}{2}}N_3^\frac{5}{2}+N_1^{-2}N_3^\frac{3}{2})]\lesssim\delta^{-4r}[
 N_1^{-\frac{3}{2}}N_2N_3^\frac{5}{2}+N_1^{-1}N_2N_3^\frac{3}{2}]\|a_{n_2}\|_{\ell^2}^2.
\end{eqnarray*}


\medskip
{\bf Case $\beta_6$:} In this case 
\begin{equation*}
S_{(n,n')}=\left\{(n_3, n_1, n_2) \, : \,  \begin{aligned} &n=-n_1+n_2+n_3, \quad n'=-n_2+n_1+n_3 ,\\
&n_1 \neq n_2, n_3, \quad |n_1|^2=|n_2|^2, \quad m=|n_3|^2\end{aligned}\right\}
\end{equation*}
so $N_1\sim N_2$ and $\Delta m\lesssim N_3^2$. We  have $S_{(n,n',m)}\lesssim N_2^3N_3$ since $n_3$ sits on a sphere of radius at most $N_3$. We then have as in \eqref{beta6} that for $\omega$ outside a set of measure $e^{-\frac{1}{\delta^r}}$, 
$$M_2^2\lesssim N_1^{-6+\varepsilon}N_2^{-6}N_2^3N_3|S_{(m)}|,$$
where 
$$S_{(m)}=\left\{(n, n', n_3, n_1, n_2) \, : \, \begin{aligned} &n=-n_1+n_2+n_3, \quad n'=-n_2+n_1+n_3 ,\\
&n_1 \neq n_2, n_3, \quad |n_1|^2=|n_2|^2, \quad m=|n_3|^2\end{aligned}\right\}$$
and $|S_{(m)}|\lesssim N_3N_2^3N_2$ since again $n_3$ sits on a sphere of radius at most $N_3$ and for fixed $n_2$ we have that $n_1$ sits on a sphere of radius at most $N_2$. Hence $M_2\lesssim  N_1^{-\frac{5}{2}+\varepsilon} N_3$ and as a consequence 
$$\mathcal{T}\lesssim \|a_{n_3}\|_{\ell^2}^2N_3^3N_1^{-\frac{5}{2}+\varepsilon}.$$


\medskip
The proof of \eqref{barRbarRD} proceeds very much like the one we just presented. Actually when $n_1=n_2$ the estimates may be made better since we will not have planes, but spheres involved in the counting.  On the other hand  here $n_1=n_2$ could be a possibility. In this case we have
$$\mathcal{T}:=  \sum_{m\in Z, n \in \Z^3} \,\left|   \sum_{\substack{ n= -2n_1+n_3 
\\ m = -2|n_1|^2+ |n_3|^2}}\frac{(\cg_{n_1}(\omega))^2}{|n_1|^3} a_{n_3}
 \right|^2
$$
Let $S_{(m,n)}=\{(n_1,n_3) \, / \, n= -2n_1+ n_3, \,  n_3\neq n_1, \, m = -2|n_1|^2 +|n_3|^2\}$, and note that by Lemma \ref{lattice-counting}, 
$|S_{(m,n)}|\lesssim \min(N_1,N_3^2)$. Then using \eqref{sup-bound-all} for $\omega$ outside a set of measure $e^{-\frac{1}{\delta^r}}$, we have that 
\begin{eqnarray*}\mathcal{T}&\lesssim&  \min(N_1,N_3^2)\sum_{m, n} \,  \sum_{S_{(m,n)}} \frac{|g_{n_1}(\omega)|^4}{|n_1|^6}
|a_{n-2n_1}|^2\\
&\lesssim & \min(N_1,N_3^2)N_1^{-3+\varepsilon}\|a_{n_3}\|_{l^2}^2.
\end{eqnarray*}

\end{proof}

\bigskip

\begin{proposition}{\label{Four}}  Let  $D_j$ and $R_k$ be as above and fix $N_1\geq N_2\geq N_3$, $r, \, \delta>0$ and $C \in \mathcal{C}_{N_2}$. Then there exists $\mu, \varepsilon>0$ and a set $\Omega_\delta\in A$ such that $\pr(\Omega_\delta^c)\leq 
e^{-\frac{1}{\delta^r}}$ such that  for any $\omega\in \Omega_\delta$ we have  \eqref{RDD} for any $0\leq \theta \leq 1$, and \eqref{DRD} .

\end{proposition}

\begin{proof}

We now move to \eqref{RDD}. Without loss of generality we assume that $\tilde D_i=D_i, \, i=2,3.$ We have that 
$$
\mathcal{T}:=  \sum_{m \in \Z, n \in \Z^3} \,\left|   \sum_{\substack{ n= n_1+n_2 + n_3 \\ m = |n_1|^2 +|n_2|^2+ |n_3|^2}} 
 \chi_{C}(n_1)\frac{g_{n_1}(\omega)}{|n_1|^{\frac{3}{2}}} a_{n_2}a_{n_3} \right|^2.
$$
Then  
$$
\mathcal{T}\lesssim  \sum_{m \in \Z,  n\in \tilde C} \,\left|   \sum_{n_2, n_3}   \sigma_{n,n_2}a_{n_2}a_{n_3}
 \right|^2,
$$
where $\tilde C$ is again a cube of sidelength approximately $N_2$ and 
$$ \sigma_{n, n_2} =  \begin{cases}\dfrac{g_{n-n_2-n_3}(\omega)}{|n-n_2-n_3|^{\frac{3}{2}}} \, & \mbox{ if } 
m = |n-n_2-n_3|^2 +|n_2|^2+ |n_3|^2  \\ 0 &\mbox{ otherwise}\end{cases}.$$  Note that  $\sigma_{n, n_2}$ also depends on $m$ and $n_3$ but we estimate it independently of $m$ and $n_3$ and take the supremum on them. 
By Cauchy-Schwarz in $n_3$, the fact that $\Delta m\lesssim N_1N_2$ and Lemma \ref{matrixnorm} we  have that
$$
\mathcal{T}\lesssim \|a_{n_3}\|_{\ell^2}^2\|a_{n_2}\|_{\ell^2}^2N_1N_2 N_3^3 \|\GG\GG^*\|;
$$
and as usual by Lemma \ref{matrixnorm} we have that 
$$\|\GG\GG^*\|\lesssim \max_{n}\sum_{n_2 }|\sigma_{n, n_2}|^2+\left(\sum_{n\ne n' \in  \tilde C}\left|\sum_{n_2}\sigma_{n, n_2}\overline{\sigma}_{n', n_2}\right|^2\right)^\frac{1}{2}=:M_1+M_2.$$
To estimate $M_1$ we will use  the set $S(n,n_3,m):=\{n_2 \, : \, m =  |n-n_2-n_3|^2 +|n_2|^2+ |n_3|^2\}$ with  cardinality $|S_{(n,n_3,m)}|\lesssim N_1$ since this set describes a sphere whose radius is at most $N_1$. Using \eqref{sup-bound-all}  we  estimate 
\begin{equation}\label{M1later}M_1\lesssim \sum_{n_2\in S(n,n_3,m)}N_1^{-3+\varepsilon}\lesssim N_1^{-2+\varepsilon}, \end{equation} for $\omega$ outside a set of measure $e^{-\frac{1}{\delta^r}}$. 

To estimate $M_2$ we first define the set 
$$S_{(n_3,m)}:=\{(n, n', n_2) \, : \,m =  |n-n_2-n_3|^2 +|n_2|^2+ |n_3|^2, \, m =  |n'-n_2-n_3|^2 +|n_2|^2+ |n_3|^2\}$$
and note that $|S_{(n_3,m)}|\lesssim N_2^3N_1^2$. Then using \eqref{largedeviation} and proceeding using arguments similar to those presented for \eqref{quote}-\eqref{final-quote}   we have
\begin{equation}\label{M2later}M_2^2\lesssim  \delta^{-2r}N_1^{-6}|S_{(n_3,m)}|\lesssim \delta^{-2r}N_1^{-6}N_2^3N_1^2.\end{equation}
By using the estimates of $M_1$ and $M_2$ we obtain that for $\omega$ outside a set of measure $e^{-\frac{1}{\delta^r}}$, 
\begin{equation}\label{estimate1}
\mathcal{T} \lesssim \|a_{n_3}\|_{\ell^2}^2\|a_{n_2}\|_{\ell^2}^2\delta^{-r}N_1N_2 N_3^3 N_1^{-2}N_2^\frac{3}{2}\lesssim 
\|a_{n_3}\|_{\ell^2}^2\|a_{n_2}\|_{\ell^2}^2\delta^{-r}N_1^{-1}N_2^\frac{5}{2}N_3^3. \end{equation}

We will interpolate this estimate with the one we obtain below.

\begin{eqnarray}
\mathcal{T}&\lesssim&N_1N_2\sup_{m}\sum_{n \in \Z^3} \,\left|   \sum_{\substack{ n= n_1+n_2 +n_3 \\ m= |n_1|^2+|n_2|^2+ |n_3|^2}} 
 \chi_{C}(n_1)\frac{g_{n_1}(\omega)}{|n_1|^{\frac{3}{2}}}a_{n_2} a_{n_3} \right|^2  \notag\\
 &\lesssim&N_1N_2 \|a_{n_3}\|_{\ell^2}^2\sup_{m} \sum_{n, n_3 \in \Z^3}\left|  \sum_{\substack{ n = n_1+n_2 +n_3 \\ m= |n_1|^2+|n_2|^2+ |n_3|^2}} 
 \chi_{C}(n_1)\frac{g_{n_1}(\omega)}{|n_1|^{\frac{3}{2}}}a_{n_2} \right|^2 \notag\\
 &\lesssim&N_1N_2 \|a_{n_3}\|_{\ell^2}^2 N_1^{-3+\varepsilon}\sup_{m}\sum_{n, n_3 \in \Z^3}|S_{(n,n_3,m)}|\sum_{S_{(n,n_3,m)}}|a_{n_2}|^2 \notag\\
  &\lesssim& N_1N_2\|a_{n_3}\|_{\ell^2}^2 N_1^{-3+\varepsilon}\min(N_2^2,N_1)\sup_{m}\sum_{n_2}\sum_{S_{(n_2,m)}}|a_{n_2}|^2 \notag\\
  &\lesssim&N_1N_2 \|a_{n_3}\|_{\ell^2}^2\|a_{n_2}\|_{\ell^2}^2 N_1^{-3+\varepsilon}\min(N_2^2,N_1)N_1N_3^3 \notag\\
  &\sim&  N_1^{-1+\varepsilon}N_3^3N_2\min(N_2^2,N_1)\|a_{n_2}\|_{\ell^2}^2\|a_{n_3}\|_{\ell^2}^2, \label{estimate2}
\end{eqnarray}

\medskip
\noindent where $S_{(n,n_3,m)}=\{(n_1,n_2) \, /\, n= n_1+n_2 +n_3, \, \, n_1\in C, \, \, m= |n_1|^2+|n_2|^2+ |n_3|^2 \}$, with $|S_{(n,n_3,m)}|\lesssim \min(N_2^2,N_1)$, $S_{(n_2,m)}=\{(n,n_1,n_3) \, /\, n= n_1+n_2 +n_3, \, \, n_1\in C,\,\,  m= |n_1|^2+|n_2|^2+ |n_3|^2  \}$ with $|S_{(n_2,m)}|\lesssim N_1N_3^3$ and we used \eqref{sup-bound-all} for $\omega$ outside a set of measure $e^{-\frac{1}{\delta^r}}$.

The estimate of \eqref{RDD} now follows by interpolating \eqref{estimate2} with \eqref{estimate1}.

\medskip

We now move to \eqref{DRD}. Again without loss of generality we assume that $\tilde D_i=D_1, \, i=1,3$. We use duality and the change of variables   $\zeta=m-|n_1|^2=|n_2|^2+ |n_3|^3$ as in the proof of Proposition \ref{One}. We note that the variation of $\zeta$ is at most $N_2^2$ and that  $n\in \tilde C,$ a cube of side length approximately $ N_2$.  We use \eqref{sup-bound-all} for $\omega$ outside a set of measure $e^{-\frac{1}{\delta^r}}$ and Lemma \ref{lattice-counting} to reduce the bound for $\mathcal T$ to estimating,  
\begin{eqnarray*}
&&N_2^2\sup_{\zeta}\sum_{n_1 \in \Z^3} \,\left|   \sum_{\substack{ n_1= n-n_2 - n_3 \\ \zeta= |n_2|^2+ |n_3|^2}} 
 \chi_{\tilde C}(n)k_{n}\frac{g_{n_2}(\omega)}{|n_2|^{\frac{3}{2}}} a_{n_3} \right|^2\|\chi_C a_{n_1}\|_{\ell^2}^2\\
 &\lesssim&N_2^2 \|\chi_C a_{n_1}\|_{\ell^2}^2\|a_{n_3}\|_{\ell}^2\sup_{\zeta} \sum_{n_1, n_3 \in \Z^3}\left|  \sum_{\substack{ n_1= n-n_2 - n_3 \\ \zeta= |n_2|^2+ |n_3|^2}} 
 \chi_{\tilde C}(n)k_{n}\frac{g_{n_2}(\omega)}{|n_2|^{\frac{3}{2}}} \right|^2\\
 &\lesssim&N_2^2 \|\chi_C a_{n_1}\|_{\ell^2}^2\|a_{n_3}\|_{\ell^2}^2 N_2^{-3+\varepsilon}\sup_{\zeta}\sum_{n_1, n_3 \in \Z^3}|S_{(n_1,n_3,\zeta)}|\sum_{S_{(n_1,n_3,\zeta)}}\chi_{\tilde C}(n)|k_{n}|^2 \\
  &\lesssim&N_2^2 \|\chi_C a_{n_1}\|_{\ell^2}^2\|a_{n_3}\|_{\ell^2}^2 N_2^{-3+\varepsilon}N_2\sup_{\zeta}\sum_{n}\sum_{S_{(n,\zeta)}}|k_{n}|^2 \\
  &\lesssim&N_2^2 \|\chi_C a_{n_1}\|_{\ell^2}^2\|a_{n_3}\|_{\ell^2}^2 N_2^{-3+\varepsilon}N_2N_2N_3^3\|k_{n}\|_{\ell^2}^2 \\
  &\sim&  N_2^{1+\varepsilon}N_3^3\|\chi_C a_{n_1}\|_{\ell^2}^2\|a_{n_3}\|_{\ell^2}^2\|k_{n}\|_{\ell^2}^2,
\end{eqnarray*}
where $S_{(n_1,n_3,\zeta)}=\{(n,n_2) \, /\, n_1= n-n_2 - n_3, \, \, n\in \tilde C,\, \, \zeta= |n_2|^2+ |n_3|^2 \}$, with $|S_{(n_1,n_3,\zeta)}|\lesssim N_2$,  $S_{(n,\zeta)}=\{(n_1,n_2,n_3) \, /\, n_1= n-n_2 - n_3, \, \, n_1\in C, \, \, \zeta= |n_2|^2+ |n_3|^2 \}$ with $|S_{(n,\zeta)}|\lesssim N_2N_3^3$.

\smallskip
\end{proof}

\begin{proposition}{\label{Five}} Let  $R_k$ be as above and fix $N_1\geq N_2\geq N_3$, $r, \, \delta>0$ and $C \in \mathcal{C}_{N_2}$. Then there exists $\mu>0$ and a set $\Omega_\delta\in A$ such that $\pr(\Omega_\delta^c)\leq 
e^{-\frac{1}{\delta^r}}$ such that for any $\omega\in \Omega_\delta$ we have $\eqref{barRbarRR}, \eqref{barRRbarR}$ and $\eqref{barRRR}$.

\end{proposition}

 %
 \begin{proof}We start by estimating \eqref{barRbarRR}.  We consider 
\begin{equation}\label{T1RRR}
\mathcal{T}:=  \sum_{m\in \Z, n \in \Z^3} \,\left|   \sum_{\substack{ n= n_1+ n_2 - n_3 \\n_1, n_2\ne n_3
\\ m = |n_1|^2 +|n_2|^2- |n_3|^2}} \chi_C(n_1)\frac{\cg_{n_1}(\omega)}{|n_1|^\frac{3}{2}}
 \frac{\cg_{n_2}(\omega)}{|n_2|^{\frac{3}{2}}} \frac{g_{n_3}(\omega)}{|n_3|^{\frac{3}{2}}} \right|^2.
\end{equation}

Note that if $n_1=n_2$ we get -say- $(\cg_{n_1}(\omega)^2)$ which are still independent and mean zero since the $g_{n_i}(\omega)$ are complex Gaussian random variables. Hence we are still within the framework of Lemma \ref{largedeviation} for $\omega$ and so this case does not  require a separate argument. 
 
 We first remark that the variation $\Delta m\sim N_1N_2$. Then we use Lemma \ref{largedeviation} to obtain, for $\omega$ outside a set of measure $e^{-\frac{1}{\delta^r}}$, that 
$$
\mathcal{T}\lesssim \delta^{-\frac{3}{2}r}N_1 N_2  N_1^{-3}N_2^{-3}N_3^{-3} \sup_{m}|S_{(m)}| \lesssim\delta^{-\frac{3}{2}r} N_1^{-1}N_2
$$
where $S_{(m)}:=\{(n,n_1,n_2,n_3) \, /\, n= n_1+ n_2 - n_3, \, \,  n_1\in C;\, m = |n_1|^2 +|n_2|^2- |n_3|^2\}$ and 
$|S_{(m)}|\lesssim N_3^3N_2^3N_1$.

\bigskip
To estimate \eqref{barRRbarR} and \eqref{barRRR} we proceed just like above.
\end{proof}

\subsection{Bilinear Estimate} \label{bilinear-est}
We prove the following bilinear estimate which will be used in the next Section \ref{theproof}. We use the same notation as in the previous subsection \ref{trilinear-est}.

\begin{proposition}\label{bilinear} Fix  $N_1\geq N_2\geq N_3$ and $r, \delta>0$. Assume also that $C$ is a cube of sidelength $N_2$. Then there exists $\mu, \, \varepsilon>0$ and   a set $\Omega_\delta\in A$ such that $\pr(\Omega_\delta^c)\leq 
e^{-\frac{1}{\delta^r}}$ and  such that for any $\omega\in \Omega_\delta$ and $0\leq \theta\leq 1$ we have the following estimate:
\begin{eqnarray}
\label{RD} && \|P_{C}R_1D_2\|_{L^2([0, 1] \times \T^3)}\,\lesssim \, \delta^{-\mu r} N_1^{-\frac{1}{2}+\varepsilon}\min(N_1, N_2^2)^{\frac{1-\theta}{ 2}}N_2^{\frac{1}{2}+\frac{3\theta}{4}} \|D_2\|_{U^2_{\Delta} L^2_x}.
\end{eqnarray}
\end{proposition}

\begin{proof}


To prove \eqref{RD} we follow the argument presented for \eqref{RDD} after performing a Cauchy-Schwarz. In fact we have 
$$
\|P_{C}R_1 D_2\|_{L^2}^2=  \sum_{m \in \Z, n \in \Z^3} \,\left|   \sum_{\substack{ n= n_1+n_2  \\ m = |n_1|^2 +|n_2|^2}} 
 \chi_{C}(n_1)\frac{g_{n_1}(\omega)}{|n_1|^{\frac{3}{2}}} a_{n_2} \right|^2.
$$
Then  
$$
\|P_{C}R_1 D_2\|_{L^2}^2\lesssim  \sum_{m; \,  n\in \tilde C} \,\left|   \sum_{n_2}   \sigma_{n,n_2}a_{n_2}
 \right|^2,
$$
where $\tilde C$ is a  cube of side length approximately $N_2$ and 
$$ \sigma_{n, n_2} =  \begin{cases}\dfrac{g_{n-n_2}(\omega)}{|n-n_2|^{\frac{3}{2}}} \, & \mbox{ if } 
m = |n-n_2|^2 +|n_2|^2  \\ 0 &\mbox{ otherwise}\end{cases}.$$  We then have 
$$
\|P_{C}R_1 D_2\|_{L^2}^2\lesssim \|a_{n_2}\|_{\ell^2}^2N_1N_2  \|\GG\GG^*\|. 
$$
Then using  the estimates \eqref{M1later} and \eqref{M2later}, we obtain for $\omega$ outside a set of measure $e^{-\frac{1}{\delta^r}}$, 
\begin{equation}\label{RD1}
\|P_{C}R_1 D_2\|_{L^2}\lesssim \|a_{n_2}\|_{\ell^2}\delta^{- r}N_1^{-\frac{1}{2}+\varepsilon}N_2^{\frac{5}{4}}. 
\end{equation}
We also use \eqref{estimate2} to estimate   \eqref{RDD}. By repeating the argument to prove \eqref{estimate2}  in our bilinear setting,  we obtain, for $\omega$ outside a set of measure $e^{-\frac{1}{\delta^r}}$, that 
\begin{eqnarray}
\|P_{C}R_1D_2\|_{L^2}^2&\lesssim&N_1N_2\sup_{m}\sum_{n \in \Z^3} \,\left|  \sum_{\substack{ n= n_1+n_2  \\ m= |n_1|^2+|n_2|^2}}
 \chi_{C}(n_1)\frac{g_{n_1}(\omega)}{|n_1|^{\frac{3}{2}}}a_{n_2}  \right|^2 \notag\\
 &\lesssim&N_1N_2 \|a_{n_3}\|_{\ell^2}^2\sup_{m} \sum_{n \in \Z^3}\left|  \sum_{\substack{ n = n_1+n_2  \\ m= |n_1|^2+|n_2|^2}} 
 \chi_{C}(n_1)\frac{g_{n_1}(\omega)}{|n_1|^{\frac{3}{2}}}a_{n_2} \right|^2 \notag\\
 &\lesssim&N_1N_2 N_1^{-3+\varepsilon}\sup_{m}\sum_{n \in \Z^3}|S_{(n,m)}|\sum_{S_{(n,m)}}|a_{n_2}|^2 \notag\\
  &\lesssim& N_1N_2N_1^{-3+\varepsilon}\min(N_2^2,N_1)\sup_{m}\sum_{n_2}\sum_{S_{(n_2,m)}}|a_{n_2}|^2 \notag\\
  &\lesssim&N_1N_2 \|a_{n_2}\|_{\ell^2}^2 N_1^{-3+\varepsilon}\min(N_2^2,N_1)N_1 \notag\\\notag
  &\sim&  N_1^{-1+\varepsilon}N_2\min(N_2^2,N_1)\|a_{n_2}\|_{\ell^2}^2, 
\end{eqnarray}
\smallskip

\noindent where $S_{(n,m)}=\{(n_1,n_2) \, /\, n= n_1+n_2,\, \, n_1\in C, \, \,  m= |n_1|^2+|n_2|^2 \}$, with $|S_{(n,m)}|\lesssim \min(N_2^2,N_1)$, $S_{(n_2,m)}=\{(n,n_1) \, /\, n= n_1+n_2, \, \, n_1\in C, \, \, m= |n_1|^2+|n_2|^2  \}$ with $|S_{(n_2,m)}|\lesssim N_1$
and we used \eqref{sup-bound-all}.
Hence we also have 
\begin{equation}\label{RD2}
\|P_{C}R_2 D_1\|_{L^2}\lesssim \|a_{n_2}\|_{\ell^2}N_1^{-\frac{1}{2}+\varepsilon}N_2^{\frac{1}{2}}\min{(N_2^2,N_1)}^{\frac{1}{2}}. 
\end{equation}
By interpolating \eqref{RD1} and \eqref{RD2} we finally have the estimate \eqref{RD}.
\end{proof}

\begin{remark} Later we only use \eqref{RD} with $\theta=1$ while estimating in the next section the term $J_4$ defined  \eqref{j4}. 
\end{remark}

\section{Proof of Proposition \ref{main-lemma}}\label{theproof}

In this section we use a notation similar to the one introduced at the beginning of  Section \ref{blocks} to indicate deterministic and random functions. The reader should pay attention though to the fact that the new functions we define in this section have a  different normalization than the ones in Section \ref{blocks}, hence  the slight change of notation.      

\smallskip

If $ u_i$ is random, then we write 
$$\widehat{P_{N_i}u_i}(n)=\chi_{\{|n|\sim N_i\}}(n)\frac{g_n(\omega)}{|n|^{\frac{5}{2}-\alpha}}e^{i|n|^2t}\sim 
\widehat{\RR_i}(n);$$ while if  $ u_i$ is deterministic we write  
$$\widehat{P_{N_i}u_i}(n)\sim  \widehat \D_i(n)$$
where $\widehat \D_i(n)$ is supported in $\{|n|\sim N_i\}$.   Below we will make a heavy use of Proposition \ref{alltrilinear}  when functions $\RR_i$ are involved instead of $R_i$. This will not be explicitly mentioned every time, but the reader will notice that a normalization will take place in the appropriate places.

\medskip

We first estimate the terms $J_2-J_7$ then we move to $J_1$.

\smallskip

\subsection {Estimates Involving the Term  $J_2$ }\label{estJ2}

We start by estimating   the term $J_2$ as in \eqref{j2}.  This reduces to analyzing  the sum over $N_0, N_1,\dots, N_3$ of  quatri-linear forms:  
 \begin{equation}\label{trilinearexpression1} \left|\int_{\T} \int_{\T^3}   T_{\Upsilon} \left( \,P_{N_1} u_1, P_{N_2} \cu_2,P_{N_3} u_3 \right) P_{N_0} \ch  \, dx dt\right|
\end{equation} where $T_{\Upsilon}$ is the multilinear  operator defined in \eqref{upsilon}.
\smallskip

The general outline of the proof involves the use of Cauchy-Schwarz, cutting the top frequency window if necessary, the transfer principle Proposition \ref{transfer} and suitably applying  the trilinear estimates in subsection \ref{trilinear-est}. Without any loss of generality, we then fix the relative ordering $N_1 \ge N_2 \ge  N_3$ above and consider the following cases where $T_\Upsilon$ acts on :

\begin{itemize}

\item  \, {\bf{Case 1: \quad $\mbox{a)}\,\, (\bar \RR_1, \RR_2, \RR_3) \quad \mbox{b)}\,  \, (\RR_1, \bar \RR_2, \RR_3) \quad \mbox{c)}\,  \, (\RR_1, \RR_2, \bar \RR_3)$}} 

\item {\bf{Case 2: \quad $\mbox{a)}\,  \, (\bar \D_1, \RR_2, \RR_3) \quad \mbox{b)}\,  \, (\D_1, \bar \RR_2, \RR_3)\quad \mbox{c)}\,  \, (\D_1,  \RR_2,  \bar \RR_3)$  }}

\item {\bf{Case 3: \quad $\mbox{a)}\, \, (\bar \RR_1, \RR_2, \D_3) \quad  \mbox{b)}\, \, (\RR_1, \bar \RR_2,  \D_3) \quad \mbox{c)}\,\, (\RR_1, \RR_2, \bar \D_3)$}  }

\item  {\bf{Case 4: \quad $\mbox{a)}\, \, (\bar \RR_1,  \D_2,  \RR_3) \quad \mbox{b)}\, \, (\RR_1, \bar \D_2,  \RR_3) \quad \mbox{c)}\, \, (\RR_1,  \D_2,  \bar \RR_3), $}}

\item {\bf{Case 5: \quad $\mbox{a)}\, \, (\bar \D_1, \RR_2,  \D_3) \quad \mbox{b)}\, \, (\D_1, \bar \RR_2,  \D_3) \quad \mbox{c)}\,\, (\D_1, \RR_2, \bar \D_3) $  }}

\item {\bf{Case 6: \quad $\mbox{a)}\, \, (\bar \RR_1,  \D_2,  \D_3) \quad   \mbox{b)}\, \,(\RR_1, \bar \D_2,  \D_3) \quad \mbox{c)}\, \, (\RR_1, \D_2, \bar \D_3)$} }

\item  {\bf{Case 7: \quad $\mbox{a)}\, \, (\bar \D_1,  \D_2,  \RR_3) \quad   \mbox{b)}\, \,(\D_1, \bar \D_2,  \RR_3) \quad \mbox{c)}\, \, (\D_1,  \D_2,  \bar \RR_3) $} }

\item  {\bf{Case 8: \quad $\mbox{a)}\, \, (\bar \D_1,  \D_2,  \D_3) \quad   \mbox{b)}\, \,(\D_1, \bar \D_2,  \D_3) \quad \mbox{c)}\, \,(\D_1,  \D_2,  \bar \D_3) $} }

\end{itemize}

\medskip

\noindent {\bf{Case 1, a):}}  If $N_1\sim N_0$  we cut the support of $\hat h$ and hence that of $\hat \RR_1$ with respect to cubes $C$ of side length $N_2$ and use 
Cauchy-Schwarz to get   
$$\eqref{trilinearexpression1}\lesssim \|P_CP_{N_0} h\|_{L^2_{x,t}}\|T_\Upsilon(P_C\bar \RR_1, \RR_2, \RR_3)\|_{L^2_{x,t}}.$$
To estimate the second factor we use \eqref{barRRR} and  normalization we obtain the bound 
$$\eqref{trilinearexpression1}\lesssim\delta^{-\mu r}N_1^{-\frac{3}{2}+\alpha}N_2^{-\frac{1}{2}+\alpha}N_3^{1-\alpha}
\|P_CP_{N_0} h\|_{L^2_{x,t}}.$$
 Renormalizing $h$ and using the embedding \eqref{embed4} we obtain 
\begin{eqnarray}\label{trireference}&&\left|\int_0^\pi  \int_{\T^3}   T_{\Upsilon} \left( \,P_CP_{N_1} u_1, P_{N_2} \cu_2,P_{N_3} u_3 \right) P_CP_{N_0} \ch  \, dx dt\right|\\\notag
&\lesssim& \delta^{-\mu r}N_0^sN_1^{-\frac{3}{2}+\alpha}N_2^{-\frac{1}{2}+\alpha}
N_3^{-1+\alpha}\|P_CP_{N_0} h\|_{Y^{-s}}\lesssim \delta^{-\mu r} N_1^{s+\alpha-\frac{3}{2}}\|P_CP_{N_0} h\|_{Y^{-s}},
\end{eqnarray}
which suffices provided $s+\alpha<\frac{3}{2}$ and $\alpha<\frac{1}{2}$.

If $N_1\sim N_2$ the cut with the cubes $C$ above is not needed and the argument proceeds as above. The condition here is $s<2-2\alpha$.

{\bf Case 1, b), c)} are treated similarly  replacing \eqref{barRRR} respectively by \eqref{barRRbarR} and \eqref{barRbarRR}.

\medskip

\noindent {\bf{Case 2, a):}}  Let us assume that $N_0\sim N_1$. We use the argument above.  To estimate  $$\|T_\Upsilon(P_C\bar \D_1, \RR_2, \RR_3)\|_{L^2_{x,t}}.$$
we use     \eqref{DRR}  and after taking derivatives and normalizing we  obtain the bound 
$$ \eqref{trireference}\lesssim \delta^{-\mu r}N_2^{\alpha-\frac{1}{4}}\, \|P_{N_1} u_1\|_{U^2_\Delta H^s}\, \|P_CP_{N_0} h\|_{Y^{-s}}$$ 
which suffices provided  $\alpha<\frac{1}{4}$. A similar bound holds when $N_1\sim N_2$ without cutting with cubes $C$. 

{\bf Case 2, b), c)} are treated similarly  replacing  \eqref{DRR}  by \eqref{DbarRR}.

\medskip

\noindent {\bf{Case 3, a), b) c):}} We use the argument above with  \eqref{barRRD}  and \eqref{barRbarRD} accordingly. If $N_1\sim N_0$ we obtain a bound of the form 
\begin{eqnarray*}\eqref{trireference}&\lesssim &\delta^{-\mu r}N_0^s[N_1^{\alpha-\frac{7}{4}}N_2^{-\frac{1}{2}+\alpha}N_3^{\frac{5}{4}-s}+
N_1^{\alpha-\frac{3}{2}}N_2^{-\frac{1}{2}+\alpha}N_3^{\frac{3}{4}-s}]\|P_{N_3} u_3\|_{U^2_\Delta H^s}\|P_CP_{N_0} h\|_{Y^{-s}}\\
&\lesssim&\delta^{-\mu r}N_1^{-\beta(s,\alpha)}\|P_{N_3} u_3 \|_{U^2_\Delta H^s}\|P_CP_{N_0} h\|_{Y^{-s}}
\end{eqnarray*} 
provided  $\alpha<\frac{1}{4}$ and $s+\alpha<\frac{3}{2}$. A similar bound holds when $N_1\sim N_2$ without cutting with cubes $C$.

\medskip

\noindent {\bf{Case 4, a), b), c):}}  We use the argument above with  \eqref{barRDR}  and \eqref{barRDbarR} accordingly. If $N_1\sim N_0$ we obtain a bound of the form 
$$ \eqref{trireference}\lesssim \delta^{-\mu r}N_1^{-\beta(s,\alpha)}\|P_{N_2} u_2\|_{U^2_\Delta H^s}\|P_CP_{N_0} h\|_{Y^{-s}}$$ 
 provided  $\alpha<\frac{1}{4}$ and $s+\alpha<\frac{3}{2}$. A similar bound holds when $N_1\sim N_2$ without cutting with cubes $C$.

\noindent {\bf{Case 5, a), b), c):}} We use the argument above with  \eqref{DRD}. If $N_1\sim N_0$ we obtain a bound of the form 
$$  \eqref{trireference}\lesssim \delta^{-\mu r}N_2^{1+\alpha-s}\|P_{N_2} u_2\|_{U^2_\Delta H^s}\|P_{N_3} u_3\|_{U^2_\Delta H^s}\|P_CP_{N_0} h\|_{Y^{-s}}$$ 
which suffices provided  $s>1+\alpha$. A similar bound holds when $N_1\sim N_2$ without cutting with cubes $C$.

\medskip

\noindent {\bf{Case 6 a):}}  If $N_1\sim N_0$ we proceed as above and we bound
\begin{eqnarray*}&&\left|\int_{\T}  \int_{\T^3}   T_{\Upsilon} \left( \,P_CP_{N_1}\bar \RR _1, P_{N_2} \D_2,P_{N_3} \D_3 \right) P_CP_{N_0} \ch  \, dx dt\right|\\&\lesssim&
 \|T_{\Upsilon} \left( \,P_CP_{N_1}\bar \RR _1, P_{N_2} \D_2,P_{N_3} \D_3 \right)\|_{L^2_{xt}}\|P_CP_{N_0} h\|_{L^2_{xt}}.\end{eqnarray*}
Then we use \eqref{RDD}, normalization and the embedding \eqref{embed4} to obtain  the bound 
\begin{eqnarray*}
\eqref{trilinearexpression1}&\lesssim& \delta^{-\mu r}N_1^{s-\frac{3}{2}+\alpha+\varepsilon}N_2^{\frac{1}{2}+\frac{3\theta}{4}-s}\min(N_1,N_2^2)^\frac{1-\theta}{2} N_3^{-s+\frac{3}{2}}\\&\times&\|P_{N_3} u_3\|_{U^2_\Delta H^s}\| P_{N_2} u_2\|_{U^2_\Delta H^s}\|P_CP_{N_0} h\|_{Y^{-s}}.\end{eqnarray*}
If $N_1\geq N_2^2$ then 
$$N_1^{s-\frac{3}{2}+\alpha+\varepsilon}N_2^{\frac{1}{2}+\frac{3\theta}{4}-s}\min(N_1,N_2^2)^\frac{1-\theta}{2} N_3^{-s+\frac{3}{2}}\leq N_1^{s-\frac{3}{2}+\alpha+\varepsilon}
N_2^{3-2s-\frac{\theta}{4}}\leq N_1^{\alpha+\varepsilon-\frac{\theta}{8}}$$
provided that $s<\frac{3}{2}-\frac{\theta}{8}$ which forces $\alpha<\frac{\theta}{8}$.

On the other hand if $N_1< N_2^2$ we have that, 
$$N_1^{s-\frac{3}{2}+\alpha+\varepsilon}N_2^{\frac{1}{2}+\frac{3\theta}{4}-s}\min(N_1,N_2^2)^\frac{1-\theta}{2} N_3^{-s+\frac{3}{2}}\leq N_1^{s-1+\alpha+\varepsilon-\frac{\theta}{2}}
N_2^{2 -2s+\frac{3\theta}{4}}\leq N_2^{2\alpha+\varepsilon-\frac{\theta}{4}}$$
provided $s>1+\frac{\theta}{2}-\alpha$. By letting, for example, $\theta=10\alpha$ we obtain that $1+4\alpha<s<\frac{3}{2}-2\alpha$ in this case, while still satisfying the requirement 
that $\alpha<\frac{\theta}{8}$ from Case a). 

If $N_1\sim N_2$ the argument is similar and easier. For Case 6 b) and c) we repeat the argument since  \eqref{RDD} is not sensitive to conjugation on the random function.

\medskip

\noindent {\bf{Case 7 a):}}   In this case we would like to use the Strichartz estimate \eqref{strichartz-mix}. But since 
$$T_{\Upsilon}(\bar \D_1,\D_2,R_3)\ne \bar \D_1\D_2\RR_3$$
we need to add back the frequencies that have been removed, i.e. allow for $n_2$ or $n_3$ to be equal to $n_1$.
If we were working with spaces whose norms are based on the absolute value of the time-space Fourier coefficients, like the $X^{s,b}$ space,  this would not be an issue, but since we are using $U^pL^2$ spaces we need to put back those missing frequencies. We show below that reintroducing these frequencies will not bring back the whole linear term that we have gauged away but only a part that has sufficient regularity to be controlled.

We start by assuming that the Fourier coefficient associated to $\D_1(t)$ is $a_{n_1}(t)$,  to $\D_2(t)$ is $b_{n_2}(t)$ and  to $\RR_3(t)$ is $c_{n_3}(t)$. Then we write
\begin{eqnarray*}
&&\sum_{n=-n_1+ n_2+ n_3, \, n_2,n_3 \ne n_1}\chi_{N_1}a_{n_1}\chi_{N_2}b_{n_2}\chi_{N_3}c_{n_3}=
-\chi_{N_3}c_{n}\left(\sum_{n_1}\chi_{N_1}a_{n_1}\chi_{N_2}b_{n_1}\right)\\
&&-\chi_{N_2}b_{n}\left(\sum_{n_3}\chi_{N_1}a_{n_3}\chi_{N_3}c_{n_3}\right)+
\chi_{N_1}a_{n}\chi_{N_2}b_{n}\chi_{N_3}c_{n}\\
&& +\sum_{n=-n_1 +n_2+n_3}\chi_{N_1}a_{n_1}\chi_{N_2}b_{n_2}\chi_{N_3}c_{n_3}=A_1(n)+A_2(n)
+A_3(n)+A_4(n).
\end{eqnarray*}
Then we have that 
$$\eqref{trilinearexpression1}\lesssim \sum_{i=1}^4\left|\int_{\T}\int_{\T^3}\FF^{-1}(A_i)(x,t)P_{N_0} \ch(x,t)\, dx\, dt\right|.$$

We now start with the estimate of $A_1$.  Using Plancherel and Cauchy-Schwarz we have
$$\left|\int_{\T}\int_{\T^3}\FF^{-1}(A_1)(x,t)P_{N_0} \ch(x,t)\, dx\, dt\right|\lesssim 
\|A_1(n)\|_{L^2(\T,\ell^2)}\|P_{N_0} h(x,t)\|_{L^2_{x,t}}.$$

We first notice that $A_1$ is not zero only if $N_3\sim N_1$. Then 
$$\|A_1(n)\|_{L^2(\T,\ell^2)}\lesssim \|\D_1 \|_{L^\infty_tL^2_x} \|\D_2 \|_{L^\infty_tL^2_x}\|\RR_3 \|_{L^2(\T,L^2(\T^3))}.$$
By renormalizing and using the embedding \eqref{embed4} we obtain that 
$$\left|\int_\T\int_{\T^3}\FF^{-1}(A_1)(x,t)P_{N_0} \ch(x,t)\, dx\, dt\right|\lesssim N_2^{-s-1+\alpha}\|P_{N_1} u_1\|_{U^2_\Delta H^s}\|P_{N_2} u_2\|_{U^2_\Delta H^s}\|P_{N_0} h\|_{Y^{-s}}.$$ 

We now note that $A_2=0$ unless $N_0\sim N_1\sim N_2$ and 
$$\|A_2(n)\|_{L^2([0,\pi],\ell^2)}\lesssim \|\D_2 \|_{L^2(\T,L^2(\T^3))} \|\D_1 \|_{L^\infty_tL^2_x}
\|\RR_3 \|_{L^\infty_tL^2_x}. $$

Also in this case we then have 
$$\left|\int_\T\int_{\T^3}\FF^{-1}(A_2)(x,t)P_{N_0} \ch(x,t)\, dx\, dt\right|\lesssim N_2^{-s-1+\alpha}\|P_{N_1} u_1\|_{U^2_\Delta H^s}\|P_{N_2}u_2\|_{U^2_\Delta H^s}\|P_{N_0} h\|_{Y^{-s}}.$$

Now we note that $A_3=0$ unless  $N_1\sim N_2\sim N_3$. Then 
$$\|A_3(n)\|_{L^2(\T,\ell^2)}\lesssim \|\D_1 \|_{L^\infty_tL^2_x} \|\D_2 \|_{L^\infty_tL^2_x}
\|\RR_3 \|_{L^2(\T,L^2(\T^3))},$$
where we used that $\|a_n\|_{\ell^\infty}\lesssim \|a_n\|_{\ell^2}$. Hence also in this case

$$\left|\int_\T\int_{\T^3}\FF^{-1}(A_3)(x,t)P_{N_0} \ch(x,t)\, dx\, dt\right|\lesssim N_2^{-s-1+\alpha}\|P_{N_1} u_1\|_{U^2_\Delta H^s}\|P_{N_2} u_2\|_{U^2_\Delta H^s}\|P_{N_0} h\|_{Y^{-s}}.$$ 

Finally we estimate the term involving $A_4$. Assume first that $N_0\sim N_1$. Then we need to estimate 

\begin{equation} \label{A4n0n1} \left|\int_\T\int_{\T^3}\FF^{-1}(A_4)(x,t)P_CP_{N_0} \ch(x,t)\, dx\, dt\right| \end{equation}
where we cut by cubes $C$ of size length $N_2$. We use Cauchy-Schwarz, 
 \eqref{strichartz-mix}, embedding \eqref{embed4} and normalization to obtain that 
 \begin{eqnarray*} \eqref{A4n0n1}\, &\lesssim & \, 
 N_0^s N_1^{-s}N_2^{1-s}N_3^\alpha\|P_C P_{N_1} u_1\|_{U^4_\Delta H^s}\|P_{N_2} u_2\|_{U^4_\Delta H^s}\|P_CP_{N_0} h\|_{Y^{-s}}\\
&\lesssim &
 N_2^{1-s+\alpha}\|P_C P_{N_1} u_1\|_{U^4_\Delta H^s}\|P_{N_2} u_2\|_{U^4_\Delta H^s}\|P_CP_{N_0} h\|_{Y^{-s}}.\end{eqnarray*}
If $N_1\sim N_2$ then the cutting with cubes $C$ is automatic and a similar bound holds.

Cases b) and c) are similar since the argument presented above is not effected by complex conjugation.

 \medskip
\noindent {\bf{Case 8:}} This case is similar and  better than Case 7.

\medskip

\subsection{Estimates Involving the Term   $J_3$}\label{estJ3}
We start by noting that $\FF(J_3)(n)$ is comprised  by terms of the form
\begin{equation}\label{sumJ3}\sum_{\Gamma(n)_{[123]}, \, n_1,n_3\ne n_2}\hat w_{n_1}(t)\ca_{n_2}(t)b_{n_3}(t),\end{equation}
where $\hat w_{n_1}(t)=c_{n_1}(t)d_{n_1}(t)r_{n_1}(t)$. We note that in the worse case,  i.e. when  the   three factors of $\hat w$ correspond to random functions, $w(t)\in  H^{3-3\alpha}$, hence $w$ can always be thought of as a deterministic function. We estimate $J_3$ using the arguments presented for the estimate of $J_2$  in subsection \ref{estJ2}, but for the reason  just explained  we do not have to consider Case 1) of that section. For Cases 2)-6) we proceed by first applying the transfer principle to the {\it quintilinear} expression associated to \eqref{sumJ3} and then regroup into as single {\it deterministic} function those with the same frequency $n_1$. 
Then we apply the appropriate trilinear estimates in Proposition \ref{trilinear-est}.  The term involving the $\ell^2$ norm of the product of the three coefficients in $n_1$   can be bounded by the product of the $\ell^2$ norm of each. We transfer and normalize back as usual.

This same argument is also used to   estimate the $A_i(x,t), \, i=1,2,3$  of Case 7).  To estimate $A_4$ we use again Strichartz inequality in Proposition \ref{prop_stric} placing $w$ in $L^p$ with $p>4$. Then we use \eqref{ulp}.

\subsection{Estimates Involving the Term   $J_4$} 
 Let $w$ now be such that $\hat w_{n_2}(t)=a_{n_2}(t)c_{n_2}(t)d_{n_2}(t)r_{n_2}(t)$  and $v$ such that, $ \hat v(n_1)=b_{n_1}$. To estimate the contribution of $J_4$ we need to estimate a term such as
$$\int_\T\int_{\T^3}P_{N_0}(wv)\overline{P_{N_0}h}\, dx\, dt=\int_\T\int_{\T^3}P_{N_0}\left(\sum_{N_1,N_2}P_{N_1}vP_{N_2}w\right)\overline{P_{N_0}h}\, dx\, dt.$$
Since $w\in H^{4-4\alpha}$, hence much smoother than $v$, the less advantageous situation is when  $N_1\sim N_0$ and $N_2\ll N_1$ and this is the one we consider below. 
We cut the frequency support of $P_{N_0}h$ with cubes $C$ of  size $N_2$ and we write 
\begin{eqnarray*}&&\left(\int_\T\int_{\T^3}P_{N_0}P_{N_1}vP_{N_2}w\overline{P_{N_0}h}dxdt\right)^2\\
&\lesssim &\left(\sum_C\|P_CP_{N_0}h\|_{L^2_tL^2_x}^2\sup\|P_CP_{N_1}vP_{N_2}w\|_{L^2_tL^2_x}\right)^2.
\end{eqnarray*} We assume first that $v$ is random. Then the remarks in Subsection \ref{estJ3} combined with the transfer principle and  the bilinear estimate \eqref{RD} with $\theta=1$ give 
$$\|P_CP_{N_1}vP_{N_2}w\|_{L^2_tL^2_x}\lesssim \delta^{-\mu r}N_1^{-\frac{1}{2}+\varepsilon}
N_2^{\frac{5}{4}}\prod_{i\notin J}\|\D_i\|_{U^2_\Delta L^2}.$$
After normalizing we obtain  the bound 
$$N_1^{-\frac{3}{2}+s+\varepsilon+\alpha}N_2^{-\frac{11}{4}+4\alpha}$$
which entails $s+\alpha<\frac{3}{2}$

\smallskip
If $v$ is deterministic then we use the bilinear estimate \eqref{2strichartz-mix} and after normalization we obtain the bound
$$N_2^{-\frac{7}{2}+4\alpha}.$$

\subsection{Estimates  Involving the Terms  $J_5, J_6$ and $J_7$}\label{estJ5}
We work with the first  term of $J_5$, the second term being analogous.  Given a dual function $h$ we define  a new function $k$ such that 
$$\widehat k(n,t)=\chi_{N_0}a_n^1(t)a_n^2(t)\widehat h(n,t)$$
where $a_n^i(t)$are the Fourier coefficient of either a random or a deterministic function. Assume that $N_1\sim N_0$. Then we cut the support of $\hat h$ with cubes $C$ of sidelength $N_2$. 
By Plancherel and Cauchy-Schwarz we need to bound
$$\|P_Ck\|_{L^2_{xt}} \quad \mbox{and } \quad \left\|\sum_{\Gamma(n)_{[1,2,3]}}\chi_C\chi_{N_1}b_{n_1}\chi_{N_2}\bar c_{n_2}\chi_{N_3}d_{n_3}\right\|_{L^2_t\ell^2}.$$
Clerarly
$$\|P_Ck\|_{L^{2}_tL^2_x}\lesssim\|P_CP_{N_0}h\|_{L^{\infty}_tL^2_x} \prod_{i=1}^2\|\chi_{N_i}a_n^i\|_{L^\infty_t\ell^2}^2.$$
On the other hand by \eqref{strichartz-mix} we have  that 
$$\left\|\sum_{\Gamma(n)_{[1,2,3]}}\chi_C\chi_{N_1}b_{n_1}\chi_{N_2}\bar c_{n_2}\chi_{N_3}d_{n_3}\right\|_{L^2_t\ell^2}
$$
has  a bound of $N_2N_3$. By normalizing, assuming at worse that all functions are random,  we obtain
$$N_0^{s-2+2\alpha}N_1^{-1+3\alpha}.$$
If $N_1\sim N_2$ the situation is similar.

\bigskip

To estimate $J_6$ we use Cauchy-Schwarz and \eqref{wl2}, while for the two terms in $J_7$ we use respectively 
\eqref{vl2} and \eqref{ul2}.

\subsection{Estimates Involving the Term   $J_1$}\label{quintilinear}

The term $J_1$ in \eqref{j1} can be written as the sum over $N_0, N_1, \dots, N_5$ - dyadic numbers- of

\begin{equation}\label{expression}  \left|\int_\T  \int_{\T^3}   P_{N_0}T_{\Upsilon} \left( P_{N_1} \tilde u_1, P_{N_2} \tilde u_2, P_{N_3} \tilde u_3, P_{N_4} \tilde u_4, P_{N_5} \tilde u_5 \right) P_{N_0} \ch  \, dx dt\right|
\end{equation} where $T_{\Upsilon}$  is the multilinear operator associated to the multiplier $\chi_{\Upsilon}$, the indicator function over  the set $\Upsilon$ now defined by 
\begin{equation}\label{upsilon5}\Upsilon(n,m) := \left\{ (n_1, m_1; \dots n_5, m_5) :  \begin{aligned} \quad &n= (-1)^{\alpha_1} n_1 + \dots + (-1)^{\alpha_5} n_5 \\ &n_k \neq n_{\ell} \, \,\mbox{ whenever } \, \, \alpha_k \neq \alpha_{\ell}, \\ &|n_j| \sim N_j, \quad j=1, \dots 5  \\
&m= (-1)^{\alpha_1} m_1 + \dots + (-1)^{\alpha_5} m_5 \, \end{aligned} \right\}\end{equation} 
where  $\alpha_j $ are $ 0$ or $1$ for $j=1, \dots, 5$.


\subsubsection{The all deterministic case DDDDD}\label{ddddd}

Without loss of generality we assume that $u_2$ and $u_4$ are conjugated. Our goal  is to use Strichartz estimates as in \eqref{strichartz-mix}, but the operator $T_{\Upsilon} \left( P_{N_1} u_1, P_{N_2} \cu_2, P_{N_3} u_3, P_{N_4} \cu_4, P_{N_5} u_5 \right) $ is not a product of the functions involved since  in the convolution of the Fourier coefficient some frequencies have been removed.  We need to add back the frequencies that have been removed, i.e. allow for $n_2$ or $n_3$ to be equal $n_1$.
If we were working with spaces whose norms are based on the absolute value of the time-space Fourier coefficients, like the $X^{s,b}$ space,  this would not be an issue, but since we are using $U^pL^2$ spaces we need to put back those missing frequencies. We show below that reintroducing these frequencies will not bring back the whole linear term that we have gauged away but only a part that has sufficient regularity to be controlled. See also Subsection \ref{estJ2}.

From \eqref{rewrite} we see that 
\begin{eqnarray}\notag
P_{N_0}(\FF^{-1}J_1)(x,t)&=&P_{N_0}T_{\Upsilon} \left( P_{N_1} u_1, P_{N_2} \cu_2, P_{N_3} u_3, P_{N_4} \cu_4, P_{N_5} u_5 \right)(x,t)\\\notag
&=&P_{N_0}(P_{N_1} u_1 P_{N_2} \cu_2 P_{N_3} u_3 P_{N_4} \cu_4 P_{N_5} u_5)(x,t)\\\label{addons}
&-&
\sum_{i=1}^5 P_{N_0}P_{N_i} u_i(x,t) \, \int_{\T^3}\prod_{j\ne i, j\in\{1,2,3,4,5\}} P_{N_j} \tilde u_j(x,t)\, dx\\\notag
&-&\sum_{i=2}^7c_i \, P_{N_0}\FF^{-1}J_i(P_{N_1} u_1, P_{N_2} \cu_2, P_{N_3} u_3, P_{N_4} \cu_4, P_{N_5} u_5 )(x,t),
\end{eqnarray}
where $c_i$ are constants and we specified as an argument of $\FF^{-1}J_1$ the functions involved in its definition. The last sum involving $J_2$-$J_7$ has been already estimated in  Subsections \ref{estJ2}-\ref{estJ5} above.  On the other  hand the first term, which is now a product of functions can be estimated as in Proposition \ref{all-deterministic}. Finally we estimate 
\begin{equation}\label{norml2xt}
\left\|   \sum_{i=1}^5 P_{N_0}P_{N_i} \tilde u_i(x,t) \, \int_{\T^3}\prod_{j\ne i,\, \,  j\in\{1,2,3,4,5\}}P_{N_j} \tilde u_j(y,t)\, dy \right\|_{L^2_{x,t}}.
\end{equation}
We first note that each term of the sum is zero unless $N_i\sim N_0$ and that
\begin{equation}\label{norml2xt2}
\eqref{norml2xt}\lesssim  \sum_{i=1}^5\|P_{N_0}P_{N_i} u_i\|_{L^\infty_tL^2_x}\prod_{j\ne i,\, \,  j\in\{1,2,3,4,5\}}
N_j^{\frac{1}{4}}\|P_{N_j}  u_j(x,t)\|_{U^4_\Delta L^2_{x}}.
\end{equation}
and this is enough since all  are deterministic.


\subsubsection{The case $DDDDR$}\label{ddddr}
In \eqref{expression} we assume without any loss of generality that $u_5$ is random and that $N_1\geq N_2\geq N_3\geq N_4$. Also in the argument below one can check that the location of the complex conjugates does not affect the proof hence here we assume that $u_2$ and $u_4$ are complex conjugate. 

\medskip
We consider the following cases:
\begin{itemize}
\item {\bf Case a):} $N_5\sim N_0$ and $N_1\leq N_5$.
\item {\bf Case b):} $N_1\sim N_0$ and $N_2\leq N_5\leq N_1$.
\item {\bf Case c):} $N_1\sim N_5$ and $N_0\leq N_1$.
\item {\bf Case d):} $N_1\sim N_0$ and $N_5\leq N_2$.
\item {\bf Case e):} $N_1\sim N_2$ and $N_5\leq N_1$.

\end{itemize}
{\bf Case a)}: Proceeding as in the trilinear estimates we first decompose the support of $\chi_{N_0}\hat h$ with cubes $C$ of sidelength $N_1$ in \eqref{expression}. By Cauchy-Schwarz, the transfer principle and Plancherel  we are reduced to estimating
\begin{equation}\label{sum1}\sum_{(m, n) \in \Z\times \Z^3} \,\left|   \sum_{\substack{ n= n_5-n_2 + n_3-n_4+n_1,\\ n_1,n_3,n_5\ne n_2,n_4 
\\ m = |n_5|^2 -|n_2|^2+ |n_3|^2-|n_4|^2+ |n_1|^2}} 
 \chi_{C}(n_5)\frac{g_{n_5}(\omega)}{|n_5|^{\frac{3}{2}}} a_{n_1}\ca_{n_2}a_{n_3}\ca_{n_4} \right|^2.
\end{equation}
We define the set 
$$S_{(n_5,n,m)}=\left\{(n_1,n_2,n_3,n_4) \, : \,  \begin{aligned}&n= n_5-n_2 + n_3-n_4+n_1,\\ &n_1, n_3, n_5\ne n_2, n_4, \, \,  n_5\in C\\ 
&m = |n_5|^2 -|n_2|^2+ |n_3|^2-|n_4|^2+ |n_1|^2\end{aligned}\right\}
$$
and note that that $|S_{(n_5,n,m)}|\lesssim N_4^3N_3^3N_2^2$. Also note that the variation of $m$, $\Delta m\sim N_5N_1$, therefore by Lemma \ref{largedeviation}, for 
$\omega$ outside a set of measure $e^{-\frac{1}{\delta^r}}$
we have
\begin{eqnarray*}
\eqref{sum1}&\lesssim&\delta^{-2\mu r}N_5N_1N_5^{-3}\sum_{m}\sum_{n_5}\left|\sum_{S_{(n_5,n,m)}}a_{n_1}\ca_{n_2}a_{n_3}\ca_{n_4} \right|^2\\
&\lesssim&\delta^{-2\mu r}N_5^{-2}N_1\sup_{m}\sum_{n_5}\sum_{S_{(n_5,n,m)}}|S_{(n_5,n,m)}||a_{n_1}|^2|a_{n_2}|^2|a_{n_3}|^2|a_{n_4}|^2 \\
&\lesssim&\delta^{-2\mu r}N_5^{-2}N_1 N_4^3N_3^3N_2^2\sum_{n_1,n_2,n_3,n_4}|a_{n_1}|^2|a_{n_2}|^2|a_{n_3}|^2|a_{n_4}|^2|S_{(n_1,n_2,n_3,n_4,m)}|\\&\lesssim &\delta^{-2\mu r}N_5^{-1}N_1 N_4^3N_3^3N_2^2\prod_{i=1}^4\|a_{n_i}\|_{\ell^2}^2
\end{eqnarray*}
where
$$S_{(n_1,n_2,n_3,n_4,m)}=\left\{(n_1,n_5) \, : \,  \begin{aligned}&n= n_5-n_2 + n_3-n_4+n_1, \, \,  n_5\in C\\ 
&m = |n_5|^2 -|n_2|^2+ |n_3|^2-|n_4|^2+ |n_1|^2\end{aligned}\right\}
$$
and in the last inequality we used that $|S_{(n_1,n_2,n_3,n_4,m)}|\leq N_5$. After renormalizing and taking square root 
we obtain the bound of 
$N_5^{-3s+\alpha+3}$ which entails $s>1+\frac{\alpha}{3}$. 

\medskip
 {\bf Case b)}:   We will proceed by duality and a change of variables $\zeta=m-|n_1|^2$ as in the proof of Proposition \ref{One} in particular see \eqref{t1dual}.  We also cut the window $N_1$ by cubes $C$ of sidelength $N_5$. We have to bound
 \begin{equation}\label{sum2}\|\gamma\|_{\ell_\zeta^2}^2\|\chi_{C}a_{n_1}\|_{\ell^2}^2
 \sum_{(\zeta, n_1) \in \Z\times \Z^3} \,\left|   \sum_{\substack{ n= n_5-n_2 + n_3-n_4+n_1,\\ n_1,n_3,n_5\ne n_2,n_4 
\\ \zeta = |n_5|^2 -|n_2|^2+ |n_3|^2-|n_4|^2}} 
 \frac{g_{n_5}(\omega)}{|n_5|^{\frac{3}{2}}} \chi_{\tilde C}(n)k_n \ca_{n_2}a_{n_3}\ca_{n_4} \right|^2,
\end{equation}
where $\tilde C$ is of size approximately $N_5$.
We define the set 
$$S_{(n_5,n_1,\zeta)}=\left\{(n,n_2,n_3,n_4) \, : \,  \begin{aligned}&n= n_5-n_2 + n_3-n_4+n_1,\\ &n_1, n_3, n_5\ne n_2, n_4, \, \,  n\in \tilde C\\ 
&\zeta = |n_5|^2 -|n_2|^2+ |n_3|^2-|n_4|^2\end{aligned}\right\}
$$
and note that that $|S_{(n_5,n_1,\zeta)}|\lesssim N_4^3N_3^3N_2^2$. Note also that $\Delta \zeta\lesssim N_5^2$ hence we can continue for $\omega$ outside a set of measure $e^{-\frac{1}{\delta^r}}$ with 
\begin{eqnarray*}
\eqref{sum2}&\lesssim&\delta^{-2\mu}\|\gamma\|_{\ell_\zeta^2}^2\|\chi_{C}a_{n_1}\|_{\ell^2}^2N_5^{2}N_5^{-3}\sup_{\zeta}\sum_{n_1,n_5}\left|\sum_{S_{(n_5,n_1,\zeta)}}\chi_{\tilde C}(n)k_n\ca_{n_2}a_{n_3}\ca_{n_4} \right|^2\\
&\lesssim&\delta^{-2\mu}\|\gamma\|_{\ell_\zeta^2}^2\|\chi_{C}a_{n_1}\|_{\ell^2}^2N_5^{-1}\sup_{\zeta}\sum_{n_1,n_5}\sum_{S_{(n_5,n_1,\zeta)}}|S_{(n_5,n_1,\zeta)}||a_{n_2}|^2|a_{n_3}|^2|a_{n_4}|^2 |\chi_{\tilde C}(n)k_n|^2\\
&\lesssim&\delta^{-2\mu}\|\gamma\|_{\ell_\zeta^2}^2\|\chi_{C}a_{n_1}\|_{\ell^2}^2N_5^{-1} N_4^3N_3^3N_2^2\\&&\sum_{n,n_2,n_3,n_4}
|a_{n_2}|^2|a_{n_3}|^2|a_{n_4}|^2 |\chi_{\tilde C}(n)k_n|^2|S_{(n,n_2,n_3,n_4,\zeta)}|\\&\lesssim &\delta^{-2\mu}\|\chi_{C}a_{n_1}\|_{\ell^2}^2N_4^3N_3^3N_2^2\prod_{i=2}^4\|a_{n_i}\|_{\ell^2}^2\|k_n\|_{\ell^2}^2\|\gamma\|_{\ell_\zeta^2}^2
\end{eqnarray*} 
where
$$S_{(n,n_2,n_3,n_4,\zeta)}=\left\{(n,n_5) \, : \,  \begin{aligned}&n= n_5-n_2 + n_3-n_4+n_1, \, \,  n_5\in C\\ 
&\zeta = |n_5|^2 -|n_2|^2+ |n_3|^2-|n_4|^2\end{aligned}\right\}
$$
and in the last inequality we used that $|S_{(n,n_2,n_3,n_4,\zeta)}|\leq N_5$. After renormalizing and taking square root 
we obtain the bound of 
$N_5^{-3s+\alpha+3}$ which entails $s>1+\frac{\alpha}{3}$. 

\medskip
 {\bf Case c)}:   This is like Case b), but now we do not need to cut the support of the window $N_1$ by $N_5$.
 
 \medskip
 {\bf Case d)}: In this case we proceed as in subsection \ref{ddddd}, the only difference being in the treatment of the terms in \eqref{addons}. More precisely here we show how to estimate the random term in \eqref{norml2xt2}. We have
 for $v_0^\omega$ in \eqref{linear} 
 \begin{equation}\label{randompiece}
 N_0^sN_3^{-1+\alpha}\|P_{N_0}P_{N_3} v_0^\omega\|_{L^\infty_tH^{1-\alpha}_x}\prod_{j=1}^4
N_j^{\frac{1}{4}-s}\|P_{N_j}  u_j(x,t)\|_{U^4_\Delta H^s}.
\end{equation}
where we notice that $N_3\sim N_0$ otherwise the contribution would be null. This is enough to obtain the desired bound. 

\medskip
 {\bf Case e)}:   This is like Case d), but now we do not need to cut the support of the window $N_1$ by $N_2$.

 
 \subsubsection{ The $DDDRR$ Case}\label{dddrr} To estimate the expression  in \eqref{expression} we will assume without any loss of generality that $u_4, \, u_5$  are  random and $  N_4\geq N_5$. We can also assume that $N_1\geq N_2\geq N_3$.  We have two different scenarios: {\underline{Case 1}}: $u_4u_5$ or   {\underline{Case 2}}:  $\bar u_4 u_5$, the other cases being obtained by complex conjugation since we do not care about bars on deterministic functions.  The only difference between Case 1 and 2 is that in Case 2 we automatically have that $n_4 \neq n_5$ which still allows us to use Proposition \ref{chaos}, and hence the same argument as in Case 1 applies.
 We discuss Case 1 within the context of the following cases ( Case 2 being analogous after appropriately rewriting the corresponding constraints):  
\begin{itemize}

\item{\bf Case a):}
\begin{itemize}
\item {Case i):}  \, $N_4\sim N_5 \geq N_0, N_1$.
\item {Case ii):}    \, $N_4\sim N_1 \geq N_0$ .
\end{itemize}
\smallskip 

\item{\bf Case b):} \, $N_4\sim N_0$  and 
\begin{itemize}
\item {Case i):}   \,$N_5\geq N_1 $.
\item {Case ii):}    $N_4\geq N_1\geq N_5 \geq N_2$.
\item {Case iii):}  $N_4 \geq N_1$ and $N_2\geq N_5 \geq N_3$.
\item {Case iv):}   $N_4\geq N_1$ and $N_3 \geq N_5$.
\end{itemize}

\smallskip

\item{\bf Case c):} \, $N_1\sim N_0$  and 
\begin{itemize}
\item {Case i):}   \,$N_1\geq N_4, \, N_5 \geq N_2 $.
\item {Case ii):}   $N_1\geq N_4 \geq N_2 \geq N_5 \geq N_3$.
\item {Case iii):}    $N_1\geq N_4 \geq N_2 \geq N_3 \geq N_5$.
\item {Case iv):}    $N_2\geq N_4, \, N_5 \geq N_3$.
\item {Case v):}   \,$N_2\geq N_4 \geq N_3 \geq N_5$.
\item {Case vi):}    $N_3 \geq N_4$.
\end{itemize}

\smallskip 

\item{\bf Case d):} \, $N_1\sim N_2 \geq N_0, N_4$

\end{itemize}

Below we always treat Case 1 and without any loss of generality we may assume $\tilde u_1 = u_1, \tilde u_j = \cu_j, j = 2, 3$. 

\medskip

{\bf $\bullet$ Case a), i):} In this case,  $N_4\sim N_5 \geq N_0, N_1$.   By Cauchy-Schwarz, the transfer principle and Plancherel  we are reduced to estimating
\begin{equation}\label{sum5}\sum_{(m, n) \in \Z\times \Z^3} \,\left| \sum_{n_4, n_5}  \left[ \sum_{S_{(n_4, n_5,n,m)}}   a_{n_1} \ca_{n_2}\ca_{n_3}   \frac{1}{|n_4|^{\frac{3}{2}}} \frac{1}{|n_5|^{\frac{3}{2}}} \right] \,  g_{n_4}(\omega) \, g_{n_5}(\omega)  \right|^2,
\end{equation}
where 
$$S_{(n_4, n_5,n,m)}=\left\{(n_1, n_2, n_3) \, : \,  \begin{aligned}&n= n_4 + n_5 +n_1-  n_2 - n_3 ,\\ & n_2, n_3\ne n_1, n_4, n_5, \, \,  \\ 
&m = |n_4|^2 + |n_5|^2+ |n_1|^2- |n_2|^2 -  |n_3|^2 \end{aligned}\right\}
$$ with $|S_{(n_4, n_5,n,m)}| \lesssim N_3^3 N_2^2 $ and the variation of $m$, \, $\Delta m \sim N_4^2$.  We then have, for $\omega$ outside a set of measure $e^{-\frac{1}{\delta^r}}$
 that 
\begin{eqnarray} \label{sum5bound}
\eqref{sum5} \, &\lesssim& \, \delta^{-2\mu r}N_4^2  N_4^{-3} N_5^{-3} \sup_m \sum_{n_4, n_5, n} \left|\sum_{S_{(n_4, n_5,n,m)}}   a_{n_1} \ca_{n_2}\ca_{n_3} \right|^2 \notag\\
 &\lesssim& \, \delta^{-2\mu r}N_4^{-1}  N_5^{-3} \sup_m \sum_{n_4, n_5, n} \sum_{S_{(n_4, n_5,n,m)}}  \, |S_{(n_4, n_5,n,m)}|   |a_{n_1}|^2  |a_{n_2}|^2 |a_{n_3}|^2  \notag \\
  &\lesssim& \, \delta^{-2\mu r}N_4^{-1}  N_5^{-3} N_2^2  N_3^3\, \sup_m \, \sum_{ n_1, n_2, n_3 }  \,   |a_{n_1}|^2  |a_{n_2}|^2 |a_{n_3}|^2 \, | S_{(n_1, n_2, n_3,m)}|  \notag\\
  &\lesssim& \, \delta^{-2\mu r}N_4^{-1}  N_5^{-3} N_2^2  N_3^3 N_4 N_5^3  \sum_{n_1, n_2,n_3}  \,   |a_{n_1}|^2  |a_{n_2}|^2 |a_{n_3}|^2  \notag \\
  &\lesssim& \,    \delta^{-2\mu r}N_2^2  N_3^3   \, \prod_{i=1}^3\,  \| a_{n_i} \|_{\ell^2}^2 
\end{eqnarray}  where we used that  
$$S_{(n_1, n_2,n_3,m)}=\left\{(n, n_4, n_5) \, : \,  \begin{aligned}&n= n_4 + n_5 +n_1-  n_2 - n_3 ,\\ & n_2, n_3\ne n_1, n_4, n_5, \, \,  \\ 
&m = |n_4|^2 + |n_5|^2+ |n_1|^2- |n_2|^2 -  |n_3|^2 \end{aligned}\right\},$$ and 
$| S_{(n_1, n_2, n_3,m)}| \lesssim N_5^3 N_4 $. Taking square root and normalizing we then obtain the bound $N_4^{s-2+2 \alpha}$ which requires $s < 2 - 2 \alpha  $.

\medskip

{\bf $\bullet$ Case a), ii)} In this case $N_4\sim N_1\geq N_0$, we repeat the argument in Case a), i), but in this case after taking square root and normalizing we obtain the bound $N_4^{\frac{3}{2}-2s+\alpha}.$

\medskip

{\bf $\bullet$ Case b), i)}  In this case,  $N_4\sim N_0$ and  $N_5 \ge N_1$. From \eqref{expression},  we first decompose the support of $\chi_{N_0}\hat h$ with cubes $C$ of sidelength $N_5$ and then apply Cauchy-Schwarz, the transfer principle and Plancherel.  We are thus reduced to estimating an expression just as in \eqref{sum5} but where now the variation in $m$, \, $\Delta m \sim N_4N_5$ and thus we obtain instead of \eqref{sum5bound} the estimate $\delta^{-2\mu r}N_4^{-1} N_5 N_2^2  N_3^3   \, \prod_{i=1}^3\,  \| a_{n_i} \|_{\ell^2}^2 $. \, Taking square root and normalizing we now obtain the bound $N_4^{s- \frac{3}{2}+\alpha}$ provided $\alpha < \frac{1}{2}$, which in turn entails $ 1 \leq s < \frac{3}{2} -  \alpha $. 

\medskip

{\bf $\bullet$  Case b), ii)}  In this case we have $N_4\sim N_0$ and  $N_1 \ge N_5$. The proof follows that of Case b), i) except that now  we first decompose the support of $\chi_{N_0}\hat h$ -and hence the $N_4$ Fourier window-  with cubes $C$ of sidelength $N_1$. We then have that $\Delta m \sim N_4N_1$ and  we now obtain instead of \eqref{sum5bound} the estimate $\delta^{-2\mu r}N_4^{-1} N_1 N_2^2  N_3^3   \, \prod_{i=1}^3\,  \| a_{n_i} \|_{\ell^2}^2 $. \, Taking square root and normalizing we now obtain as before the bound $N_4^{s- \frac{3}{2}+\alpha}$.   

\medskip

{\bf $\bullet$  Cases b, iii) iv)}  are analogous to case {\bf Case b, ii)}.

\medskip

{\bf $\bullet$  Case c, i)}  In this case we have that $N_0\sim N_1 \geq N_4, \, N_5 \geq N_2 $.   We will proceed by duality and a change of variables $\zeta=m-|n_1|^2$ as in the proof of Proposition \ref{One}, \eqref{t1dual} and also as in \eqref{sum2} in subsection \ref{ddddr} above.  We also cut the window $N_1$ by cubes $C$ of sidelength $N_4$. We have to bound
 \begin{equation}\label{sum6}\, \, \|\chi_{C}a_{n_1}\|_{\ell^2}^2\|\gamma\|_{\ell_\zeta^2}^2
 \sum_{\zeta \in \Z,\,  n_1\in \Z^3 } \,\left|   \sum_{\substack{ n= n_1 +n_4 + n_5 -n_2 - n_3,\\ n_2,n_3\ne n_1,n_4, n_5
\\ \zeta = |n_4|^2 + |n_5|^2 -|n_2|^2 - |n_3|^2}} \frac{g_{n_4}(\omega)}{|n_4|^{\frac{3}{2}}}
 \frac{g_{n_5}(\omega)}{|n_5|^{\frac{3}{2}}} \chi_{\tilde C}(n)k_n \ca_{n_2}\ca_{n_3} \right|^2,
\end{equation}   
where $\tilde C$ is of size approximately $N_4$.
Let us now define 
\begin{equation}\label{sigma12} \sigma_{n_1, n_2} =  \sum_{\substack{n= n_1 +n_4 + n_5 -n_2 - n_3,\\ n_2,n_3\ne n_1,n_4, n_5  \\ \zeta = |n_4|^2 + |n_5|^2 -|n_2|^2 - |n_3|^2}}  \,  \chi_{\tilde C}(n)k_n \ca_{n_3} \, \frac{g_{n_4}(\omega)}{|n_4|^{\frac{3}{2}}}
 \frac{g_{n_5}(\omega)}{|n_5|^{\frac{3}{2}}},\end{equation} and note that then $\Delta \zeta \sim N_4^2$. Then 
 \begin{eqnarray}\label{matrix1}
 \eqref{sum6} \, &\lesssim& \, \|\chi_{C}a_{n_1}\|_{\ell^2}^2 \|\gamma\|_{\ell_\zeta^2}^2N_4^2  \sup_{\zeta} \, \sum_{n_1 \in C} \left| \sigma_{n_1, n_2}  \ca_{n_2} \right|^2 \notag\\
 &\, \leq \,& N_4^2 \, \|\chi_{C}a_{n_1}\|_{\ell^2}^2 
 \|\gamma\|_{\ell_\zeta^2}^2 \|a_{n_2} \|_{\ell^2}^2 \,\,  \sup_{\zeta} \,\| \mathcal{G}  \mathcal{G}^{\ast}\| \, .
\end{eqnarray}  As in Section \ref{blocks},  we write  
\begin{equation}\label{gigistar}\| \mathcal{G}  \mathcal{G}^{\ast}\|\, \lesssim \,  \max_{n_1}\sum_{n_2, \, n_2 \neq n_1}|\sigma_{n_1, n_2}|^2\, +\, \left(\sum_{n_1\ne n_1'}\left|\sum_{n_2 }\sigma_{n_1, n_2}\overline{\sigma}_{n_1', n_2}\right|^2\right)^\frac{1}{2}\, =: \, M_1\,+\,M_2,\end{equation}
and estimate each term separately. For $M_1$ we proceed as follows: 
\begin{eqnarray} \label{sum7}
M_1 &=& \sup_{n_1} \sum_{n_2, \, n_2 \neq n_1} \left| \sum_{n_4, n_5} \left[ \sum_{S_{(n_1, n_2, n_4, n_5, \zeta)} } \,  \chi_{\tilde C}(n)k_n \ca_{n_3} \, \frac{1}{|n_4|^{\frac{3}{2}}}
 \frac{1}{|n_5|^{\frac{3}{2}}}\right]  g_{n_4}(\omega) \, g_{n_5}(\omega) \, \right|^2, \notag \\
 &\lesssim& \, \delta^{-2\mu r}\sup_{n_1} \sum_{n_2 \ne n_1, n_4, n_5} N_4^{-3} N_5^{-3} \left| \sum_{S_{(n_1, n_2, n_4, n_5, \zeta)} } \,  \chi_{\tilde C}(n)k_n \ca_{n_3}  \right|^2 \notag\\
 &\lesssim& \, \delta^{-2\mu r}\sup_{n_1} \sum_{n_2, n_4, n_5} N_4^{-3} N_5^{-3} |S_{(n_1, n_2, n_4, n_5, \zeta)}| \sum_{S_{(n_1, n_2, n_4, n_5, \zeta)} } \,  |\chi_{C}(n)k_n|^2 |\ca_{n_3}|^2,
 \end{eqnarray} 
 for $\omega$ out side a set of measure $e^{-\frac{1}{\delta^r}}$,
 where 
 $$S_{(n_1, n_2, n_4, n_5, \zeta)}:= \left\{ (n, n_3) \, :\, \begin{aligned}  n&= n_1 +n_4 + n_5 -n_2 - n_3,\\ &n_2,n_3\ne n_1,n_4, n_5, \, \, n \in \tilde C   \\ \zeta &= |n_4|^2 + |n_5|^2 -|n_2|^2 - |n_3|^2 \end{aligned} \right\} $$ and $ |S_{(n_1, n_2, n_4, n_5, \zeta)}| \lesssim N_3^2$.  Hence for  $$S_{(n, n_3, \zeta)}:= \left\{ (n_2, n_4, n_5) \, :\, \begin{aligned}  n&= n_1 +n_4 + n_5 -n_2 - n_3,\\ &n_2,n_3\ne n_1,n_4, n_5, \, \, n \in \tilde C   \\ \zeta &= |n_4|^2 + |n_5|^2 -|n_2|^2 - |n_3|^2 \end{aligned} \right\} $$  we have that 
\begin{eqnarray*} 
\eqref{sum7} \, &\lesssim& \,  \delta^{-2\mu r}N_4^{-3} N_5^{-3} N_3^2  \sum_{n, n_3} \,  |\chi_{\tilde C}(n)k_n|^2 |\ca_{n_3}|^2 |S_{(n, n_3, \zeta)}| \\
&\lesssim&   \delta^{-2\mu r}N_4^{-3} N_5^{-3}N_3^2 N_2^3 N_5^3 N_4 \, \| \chi_{\tilde C}(n)k_n\|_{\ell^2}^2 \|a_{n_3}\|_{\ell^2}^2 \\
&\lesssim&\delta^{-2\mu r} N_4^{-2} N_2^3 N_3^2 \, \| \chi_{\tilde C}(n)k_n\|_{\ell^2}^2 \|a_{n_3}\|_{\ell^2}^2.
\end{eqnarray*} Hence the contribution of $M_1$ to \eqref{matrix1} is
$$ \delta^{-2\mu r}N_4^2 N_4^{-2} N_2^3 N_3^2 \,\,  \|\chi_{C}a_{n_1}\|_{\ell^2}^2 \|a_{n_2} \|_{\ell^2}^2 \|a_{n_3}\|_{\ell^2}^2 \, \|\gamma\|_{\ell_\zeta^2}^2  \| \chi_{\tilde C}(n)k_n\|_{\ell^2}^2.$$
After taking square root and normalizing we then obtain a bound of $N_4^{-1 +\alpha} N_5^{-s +\frac{1}{2} + \alpha}$ which suffices provided $s > \frac{1}{2}+ \alpha$.

 To estimate $M_2$ we first write
\begin{eqnarray} \label{m22expression}
M_2^2&=&\sum_{n_1\ne n_1'}\left|\sum_{n_2}\sigma_{n_1, n_2}\overline \sigma_{n_1', n_2}\right|^2 \notag \\ &\sim&
\sum_{n_1\ne n_1'}\left| \sum_{S_{(n_1,n_1', \zeta)}}  \chi_{\tilde C}(n)k_n  \chi_{\tilde C}(n')k_{n'} \ca_{n_3} a_{n_3'} \frac{g_{n_4}(\omega)}{|n_4|^{\frac{3}{2}}}\frac{g_{n_5}(\omega)}{|n_5|^{\frac{3}{2}}}\frac{\cg_{n_4'}(\omega)}{|n_4'|^{\frac{3}{2}}}\frac{\cg_{n_5'}(\omega)}{|n_5'|^{\frac{3}{2}}}\right|^2
\end{eqnarray}
where
\begin{equation}\label{Sn1n1'}
S_{(n_1,n_1', \zeta)}=\left\{(n, n_2, n_3, n_3', n_4,n_4', n_5, n_5') \, : \,  \begin{aligned}  & n= n_1 +n_4 + n_5 -n_2 - n_3,  \\ &n'= n'_1 +n'_4 + n'_5 -n_2 - n'_3,\\  &n_2,n_3\ne n_1,n_4, n_5; \,\, n'_2,n'_3\ne n'_1,n'_4, n'_5; \, \, \, n, n' \in \tilde C
\\ &\zeta = |n_4|^2 + |n_5|^2 -|n_2|^2 - |n_3|^2, \\  &\zeta = |n'_4|^2 + |n'_5|^2 -|n_2|^2 - |n'_3|^2 \end{aligned} \right\}.
\end{equation} To streamline the exposition let $$\mathscr{C} \, : =  \, \begin{cases}  \begin{aligned}  & n= n_1 +n_4 + n_5 -n_2 - n_3,  \quad \qquad n'= n'_1 +n'_4 + n'_5 -n_2 - n'_3; \\ &\zeta = |n_4|^2 + |n_5|^2 -|n_2|^2 - |n_3|^2,  \, \quad \zeta = |n'_4|^2 + |n'_5|^2 -|n_2|^2 - |n'_3|^2; \\  &n_2,n_3\ne n_1,n_4, n_5; \quad n'_2,n'_3\ne n'_1,n'_4, n'_5; \quad n, n' \in \tilde C.
\end{aligned}  \end{cases}$$

We need to organize the estimates according to whether some frequencies are the same or not, in all we have seven cases:
\begin{itemize}
\item {\bf Case $\beta_1$:} $n_4, n_5 \neq n_4', n_5'.$
\item {\bf Case $\beta_2$:} $n_4=n_4'; \, \, n_5\ne n_5'.$
\item {\bf Case $\beta_3$:} $n_4\ne n_4'; \, \, n_5= n_5'.$
\item {\bf Case $\beta_4$:} $n_4\ne n_5'; \, \, n_5 =  n_4'.$
\item {\bf Case $\beta_5$:} $n_4= n_5'; \, \, n_5\ne  n_4'.$
\item {\bf Case $\beta_6$:} $n_4= n_5'; \, \, n_5 =  n_4'.$ 
\item {\bf Case $\beta_7$:} $n_4= n_4'; \, \, n_5 =  n_5'.$
\end{itemize}

\noindent {\bf Case $\beta_1$:}  To estimate the contribution of $M_2$,  we first define the set 
$$S_{(n_1,n_1', n_4, n_4', n_5, n_5', \zeta)}:= \{ (n, n', n_2, n_3, n_3') \, \mbox{ satisfying } \, \mathscr{C}\},$$ 
with  $| S_{(n_1,n_1', n_4, n_4', n_5, n_5', \zeta)}| \, \lesssim \, N_3^6 \, N_2^2$.\,  Next, for $\omega$ out side a set of measure $e^{-\frac{1}{\delta^r}}$, we estimate $M_2^2$ as follows:
\begin{eqnarray*} 
\eqref{m22expression} \, &\lesssim&  \delta^{-4\mu r}\sum_{n_1 \neq n_1'}  N_4^{-6} N_5^{-6} \sum_{n_4 \ne n_4', \, n_5 \ne n_5'}\, \left[ \sum_{ S_{(n_1,n_1', n_4, n_4', n_5, n_5', \zeta)}}\, \chi_{C}(n)k_n  \chi_{C}(n')k_{n'} \ca_{n_3} a_{n_3'} \right]^2 \\
 &\lesssim&  \delta^{-4\mu r}\sum_{n_1 \neq n_1'}  N_4^{-6} N_5^{-6} N_3^6 \, N_2^2 \sum_{n_4, n_4', n_5,  n_5'}\, \sum_{ S_{(n_1,n_1', n_4, n_4', n_5, n_5', \zeta)}}\, |\chi_{\tilde C}(n)k_n|^2  |\chi_{\tilde C}(n')k_{n'}|^2 |\ca_{n_3}|^2 |a_{n_3'}|^2\\
 &\lesssim&   \delta^{-4\mu r}N_4^{-6} N_5^{-6} N_3^6 N_2^2 \, \sum_{n, n', n_3, n_3'}\, |S_{(n, n', n_3, n_3', \zeta)}| \, |\chi_{\tilde C}(n)k_n|^2  |\chi_{\tilde C}(n')k_{n'}|^2 |\ca_{n_3}|^2    
 |a_{n_3'}|^2 \\
 &\lesssim&  \delta^{-4\mu r}N_4^{-6} N_5^{-6} N_3^6  N_2^2  N_2^3 N_5^6 N_4^2 \|\chi_{\tilde C}(n)k_n\|_{\ell^2}^2  \|\chi_{\tilde C}(n')k_{n'}\|_{\ell^2}^2 \|a_{n_3}\|_{\ell^2}^2 \|a_{n_3'}\|_{\ell^2}^2\\
 &\lesssim& \delta^{-4\mu r}N_4^{-4} N_3^6 N_2^5  \|\chi_{\tilde C}(n)k_n\|_{\ell^2}^2  \|\chi_{\tilde C}(n')k_{n'}\|_{\ell^2}^2 \|a_{n_3}\|_{\ell^2}^2 \|a_{n_3'}\|_{\ell^2}^2, 
\end{eqnarray*} where we have used that 
$$S_{(n, n', n_3, n_3', \zeta)}\, :=\, \{  (n_1, n_1', n_2, n_4, n_4', n_5, n_5') \, \mbox{ satisfying } \, \mathscr{C}  \}$$
has cardinality  less than  or equal to $N_2^3 N_5^6 N_4^2$.  

All in all we then have that the contribution of $\Delta \zeta M_2$ is bounded by  $ N_3^3 N_2^{\frac{5}{2}}$.  Taking square root and normalizing we finally obtain the bound   $  N_4^{-1+ \alpha} N_5^{\frac{7}{4}- 2s + \alpha} $  in this case which suffices provided $s > \frac{7}{8}+ \frac{\alpha}{2}$.

\noindent {\bf Case $\beta_2$:} Now we have that $n_4=n_4'$ while $n_5\ne n_5'$ rendering \eqref{m22expression} equal to 
\begin{equation} \label{m22forbeta2} 
\sum_{n_1\ne n_1'}\left| \sum_{S_{(n_1,n_1', \zeta)}}  \chi_{\tilde C}(n)k_n  \chi_{\tilde C}(n')k_{n'} \ca_{n_3} a_{n_3'}\,  \frac{|g_{n_4}(\omega)|^2}{|n_4|^3} \,  \frac{g_{n_5}(\omega)}{|n_5|^{\frac{3}{2}}} \frac{\cg_{n_5'}(\omega)}{|n_5'|^{\frac{3}{2}}}\right|^2.
\end{equation}
We proceed in a similar fashion  as we did in \eqref{quote}-\eqref{final-quote} and define 
\begin{eqnarray} \label{minusone}
 \mathcal{Q}_1&:=& \, \sum_{n_1\ne n_1'}\left| \sum_{S_{(n_1,n_1', \zeta)}}  k^{\tilde C}_n \, k^{\tilde C}_{n'} \,\ca_{n_3} \, a_{n_3'}\,  \frac{(|g_{n_4}(\omega)|^2 - 1)}{|n_4|^3} \,  \frac{g_{n_5}(\omega)}{|n_5|^{\frac{3}{2}}} \frac{\cg_{n_5'}(\omega)}{|n_5'|^{\frac{3}{2}}}\right|^2 \\
 \mathcal{Q}_2 &:=&\, \sum_{n_1\ne n_1'}\left| \sum_{S_{(n_1,n_1', \zeta)}} k^{\tilde C}_n \, k^{\tilde C}_{n'} \,  \ca_{n_3} \, a_{n_3'}\,  \frac{1}{|n_4|^3} \,  \frac{g_{n_5}(\omega)}{|n_5|^{\frac{3}{2}}} \frac{\cg_{n_5'}(\omega)}{|n_5'|^{\frac{3}{2}}}\right|^2\label{plusoneagain}
 \end{eqnarray}  where we have denoted by $k^{\tilde C}_n := \chi_{\tilde C}(n)k_n $ and similarly for $k^{\tilde C}_{n'}$.
 
 \smallskip
 
To estimate $\mathcal{Q}_2$ define the set 
$$S_{(n_1,n_1', n_5, n_5', \zeta)}:= \{ (n, n', n_2, n_4, n_3, n_3') \, \mbox{ satisfying } \, \mathscr{C}  \, \},$$ with 
$| S_{(n_1,n_1', n_5, n_5', \zeta)}| \, \lesssim \, N_3^6 N_2^3 N_4$.\, Then for $\omega$ outside a set of measure $e^{-\frac{1}{\delta^r}}$
\begin{eqnarray*} 
\eqref{plusoneagain} \, &\lesssim& \delta^{-4\mu r} \sum_{n_1 \neq n_1'}  N_4^{-6} N_5^{-6}  \sum_{ n_5 \ne n_5'}\, \left[ \sum_{ S_{(n_1,n_1', n_5, n_5', \zeta)}}\, k^{\tilde C}_n  \, k^{\tilde C}_{n'} \, \ca_{n_3} \, a_{n_3'} \right]^2 \\
 &\lesssim&  \delta^{-4\mu r}\sum_{n_1 \neq n_1'}  N_4^{-6} N_5^{-6}  N_3^6 N_2^3 N_4  \sum_{n_5,  n_5'}\, \sum_{ S_{(n_1,n_1', n_5, n_5', \zeta)}}\, |k^{\tilde C}_n|^2  |k^{\tilde C}_{n'}|^2 |\ca_{n_3}|^2 |a_{n_3'}|^2\\
 &\lesssim&   \delta^{-4\mu r}N_4^{-6} N_5^{-6} N_3^6 N_2^3 N_4 \, \sum_{n, n', n_3, n_3'}\, |S_{(n, n', n_3, n_3', \zeta)}|\, |k^{\tilde C}_n|^2  |k^{\tilde C}_{n'}|^2 |\ca_{n_3}|^2    
 |a_{n_3'}|^2 \\
 &\lesssim&  \delta^{-4\mu r}N_4^{-6} N_5^{-6} N_3^6 N_2^3 N_4  N_2^3 N_5^6 N_4 \, \|\chi_{\tilde C}(n)k_n\|_{\ell^2}^2  \|\chi_{\tilde C}(n')k_{n'}\|_{\ell^2}^2 \|a_{n_3}\|_{\ell^2}^2 \|a_{n_3'}\|_{\ell^2}^2\\
 &\lesssim& \delta^{-4\mu r}N_4^{-4} N_3^6 N_2^6  \|\chi_{\tilde C}(n)k_n\|_{\ell^2}^2 \,  \|\chi_{\tilde C}(n')k_{n'}\|_{\ell^2}^2 \|a_{n_3}\|_{\ell^2}^2 \|a_{n_3'}\|_{\ell^2}^2, 
\end{eqnarray*} where we have used that 
$$S_{(n, n', n_3, n_3', \zeta)}\, :=\, \{  (n_1, n_1', n_2, n_4, n_5, n_5') \, \mbox{ satisfying } \, \mathscr{C}   \}$$
has cardinality  less than  or equal to $N_2^3 N_5^6 N_4$.  

The bound for $\mathcal{Q}_1$ is smaller, just as we saw it was the case in the proof of Proposition \ref{One},  \eqref{quote}-\eqref{final-quote} in Section \ref{blocks}. We omit the details.

Thus the contribution of $\Delta\zeta \,M_2$,  is bounded by $N_3^3 N_2^3$ which after taking the square root and normalizing gives a bound of $N_4^{-1+\alpha} N_5^{2- 2 s+ \alpha}$ which suffices provided $s > 1 + \frac{\alpha}{2}$.

\noindent {\bf Case $\beta_3$:} Now we have that $n_4 \ne n_4'$ while $n_5 = n_5'$ rendering \eqref{m22expression} equal to \begin{equation} \label{m22forbeta3} 
\sum_{n_1\ne n_1'}\left| \sum_{S_{(n_1,n_1', \zeta)}}  k^{\tilde C}_n  \, k^{\tilde C}_{n'} \ca_{n_3} a_{n_3'}\,  \frac{|g_{n_5}(\omega)|^2}{|n_5|^3} \,  \frac{g_{n_4}(\omega)}{|n_4|^{\frac{3}{2}}} \frac{\cg_{n_4'}(\omega)}{|n_4'|^{\frac{3}{2}}}\right|^2.
\end{equation}  We proceed as above, defining analogous $\mathcal{Q}_1$ and $\mathcal{Q}_2$ terms bounding \eqref{m22forbeta3} in this case. 

\smallskip 

To estimate $\mathcal{Q}_2$ we define the set, 
$$S_{(n_1,n_1', n_4, n_4', \zeta)}:= \{ (n, n', n_2, n_5, n_3, n_3') \, \mbox{ satisfying } \, \mathscr{C}  \, \},$$ with 
$| S_{(n_1,n_1', n_4, n_4', \zeta)}| \, \lesssim \, N_3^6 N_2^3 \min(N_5^2, N_4) \leq N_3^6 N_2^3 N_5^2$.\, Then for $\omega$ outside a set of measure $e^{-\frac{1}{\delta^r}}$
\begin{eqnarray} \label{ref1Q2} 
\quad \mathcal{Q}_2  &\lesssim& \delta^{-4\mu r} \sum_{n_1 \neq n_1'}  N_4^{-6} N_5^{-6}  \sum_{ n_4 \ne n_4'}\, \left[ \sum_{ S_{(n_1,n_1', n_4, n_4', \zeta)}}\, k^{\tilde C}_n  \, k^{\tilde C}_{n'}\ca_{n_3} \, a_{n_3'} \right]^2 \\
 &\lesssim&   \delta^{-4\mu r}\sum_{n_1 \neq n_1'}  N_4^{-6} N_5^{-6}  N_3^6 N_2^3  N_5^2 \sum_{n_4,  n_4'}\, \sum_{ S_{(n_1,n_1', n_4, n_4', \zeta)}}\, |k^{\tilde C}_n|^2  |k^{\tilde C}_{n'}|^2 |\ca_{n_3}|^2 |a_{n_3'}|^2 \notag\\
 &\lesssim&     \delta^{-4\mu r}N_4^{-6} N_5^{-6}  N_3^6 N_2^3 N_5^2  \, \sum_{n, n', n_3, n_3'}\, |S_{(n, n', n_3, n_3', \zeta)}| \, |k^{\tilde C}_n|^2  |k^{\tilde C}_{n'}|^2 |\ca_{n_3}|^2    
 |a_{n_3'}|^2  \notag\\
 &\lesssim&   \delta^{-4\mu r}N_4^{-6} N_5^{-6}  N_3^6 N_2^3  N_5^2 \, N_2^3 N_5^3 N_4^2 \,  \|k^{\tilde C}_n\|_{\ell^2}^2  \|k^{\tilde C}_{n'}\|_{\ell^2}^2 \|a_{n_3}\|_{\ell^2}^2 \|a_{n_3'}\|_{\ell^2}^2 \notag\\
 &\lesssim&  \delta^{-4\mu r}N_4^{-4} N_5^{-1} \, N_3^6 N_2^6 \,  \|\chi_{\tilde C}(n)k_n\|_{\ell^2}^2  \|\chi_{\tilde C}(n')k_{n'}\|_{\ell^2}^2 \|a_{n_3}\|_{\ell^2}^2 \|a_{n_3'}\|_{\ell^2}^2, \notag
\end{eqnarray} where we have now used that, 
$$S_{(n, n', n_3, n_3', \zeta)}\, :=\, \{  (n_1, n_1', n_2, n_4, n_4', n_5) \, \mbox{ satisfying } \, \mathscr{C}  \,  \}$$ has cardinality  less than  or equal to $N_2^3 N_5^3 N_4^2$.  Note this is a better bound than that obtained in Case $\beta_2$. 

Since  just as before, the bound for $\mathcal{Q}_1$ is smaller, we have that the contribution of $\Delta\zeta \,M_2$,  is bounded by $N_3^3 N_2^3$. After taking the square root and normalizing the latter gives the same abound as in Case $\beta_2$. 

\smallskip

\noindent {\bf Case $\beta_4$:} In this case we have that $n_4 \ne n_5'$ while $n_5 = n_4'$ rendering \eqref{m22expression} equal to 
\begin{equation} \label{m22forbeta4} 
\sum_{n_1\ne n_1'}\left| \sum_{S_{(n_1,n_1', \zeta)}}  k^{\tilde C}_n  \, k^{\tilde C}_{n'} \ca_{n_3} a_{n_3'}\,  \frac{|g_{n_5}(\omega)|^2}{|n_5|^3} \,  \frac{g_{n_4}(\omega)}{|n_4|^{\frac{3}{2}}} \frac{\cg_{n_5'}(\omega)}{|n_5'|^{\frac{3}{2}}}\right|^2.
\end{equation}  Once again, we proceed by defining the corresponding $\mathcal{Q}_1$ and $\mathcal{Q}_2$ terms bounding \eqref{m22forbeta4} and note the estimate for 
$\mathcal{Q}_1$ is better than that for $\mathcal{Q}_2$. In the latter case, we proceed as in \eqref{ref1Q2} in case $\beta_3$ but now we have 
$$S_{(n_1,n_1', n_4, n_5', \zeta)}:= \{ (n, n', n_2, n_5, n_3, n_3') \, \mbox{ satisfying } \, \mathscr{C}  \},$$ which has cardinality,  $| S_{(n_1,n_1', n_4, n_5', \zeta)}| \, \lesssim \, N_3^6 N_2^3 N_4 $.  Furthermore, since $n_5 = n_4'$, we have from the definition of $\mathscr{C}$ that 
$\Delta \zeta \lesssim N_5^2$.

Then for $\omega$ outside a set of measure $e^{-\frac{1}{\delta^r}}$,
\begin{eqnarray} \label{ref2Q2} 
\quad \mathcal{Q}_2  &\lesssim&  \delta^{-4\mu r}\sum_{n_1 \neq n_1'}  N_4^{-3} N_5^{-9}  \sum_{ n_4 \ne n_5'}\, \left[ \sum_{ S_{(n_1,n_1', n_4, n_5', \zeta)}}\, k^{\tilde C}_n  \, k^{\tilde C}_{n'} \, \ca_{n_3} \, a_{n_3'} \right]^2 \\
 &\lesssim&  \delta^{-4\mu r}\sum_{n_1 \neq n_1'} N_4^{-3} N_5^{-9}  N_3^6 N_2^3 N_4    \sum_{n_4,  n_5'}\, \sum_{ S_{(n_1,n_1', n_4, n_5', \zeta)}}\, |k^{\tilde C}_n|^2  |k^{\tilde C}_{n'}|^2 |\ca_{n_3}|^2 |a_{n_3'}|^2 \notag\\
 &\lesssim&    N_4^{-3} N_5^{-9}  N_3^6 N_2^3 N_4   \, \sum_{n, n', n_3, n_3'}\,|S_{(n, n', n_3, n_3', \zeta)}|\, |k^{\tilde C}_n|^2  |k^{\tilde C}_{n'}|^2 |\ca_{n_3}|^2    
 |a_{n_3'}|^2  \notag\\
 &\lesssim&  \delta^{-4\mu r}N_4^{-3} N_5^{-9}  N_3^6 N_2^3 N_4  \, N_2^3 N_5^5 N_4 \,  \|k^{\tilde C}_n \|_{\ell^2}^2  \|k^{\tilde C}_{n'}\|_{\ell^2}^2 \|a_{n_3}\|_{\ell^2}^2 \|a_{n_3'}\|_{\ell^2}^2 \notag\\
 &\lesssim& \delta^{-4\mu r} N_4^{-1} N_5^{-4} N_3^6 N_2^6 \,  \|\chi_{\tilde C}(n)k_n\|_{\ell^2}^2  \|\chi_{\tilde C}(n')k_{n'}\|_{\ell^2}^2 \|a_{n_3}\|_{\ell^2}^2 \|a_{n_3'}\|_{\ell^2}^2, \notag
\end{eqnarray} where now 
$$S_{(n, n', n_3, n_3', \zeta)}\, :=\, \{  (n_1, n_1', n_2, n_4, n_5', n_5) \, \mbox{ satisfying } \, \mathscr{C}   \}$$ has cardinality  less than  or equal to  $ N_2^3 N_5^3  N_4$. 

Thus $$\Delta\zeta \,M_2 \lesssim  \delta^{-4\mu r} N_5^2 N_4^{-\frac{1}{2}} N_5^{-2} N_3^3 N_2^3 \,  \|\chi_{\tilde C}(n)k_n\|_{\ell^2}  \|\chi_{\tilde C}(n')k_{n'}\|_{\ell^2} \|a_{n_3}\|_{\ell^2} \|a_{n_3'}\|_{\ell^2},$$ whence after taking square root and normalizing we obtain in this case a bound of $N_4^{-\frac{5}{4}+\alpha} N_5^{2- 2 s+ \alpha}$ which suffices provided $s > 1+ \frac{\alpha}{2}$.

\smallskip

\noindent {\bf Case $\beta_5$:} In this case we have that $n_4 = n_5'$ while $n_5 \ne n_4'$ rendering \eqref{m22expression} equal to 
\begin{equation} \label{m22forbeta5} 
\sum_{n_1\ne n_1'}\left| \sum_{S_{(n_1,n_1', \zeta)}}  k^{\tilde C}_n  \, k^{\tilde C}_{n'} \ca_{n_3} a_{n_3'}\,  \frac{|g_{n_4}(\omega)|^2}{|n_4|^3} \,  \frac{g_{n_4'}(\omega)}{|n_4'|^{\frac{3}{2}}} \frac{\cg_{n_5}(\omega)}{|n_5|^{\frac{3}{2}}}\right|^2.
\end{equation}  Once again, we proceed by defining the corresponding $\mathcal{Q}_1$ and $\mathcal{Q}_2$ terms bounding \eqref{m22forbeta5}.  We treat the estimate $\mathcal{Q}_2$ as in \eqref{ref1Q2} in case $\beta_3$ but with the set 
$$S_{(n_1,n_1', n_4', n_5, \zeta)}:= \{ (n, n', n_2, n_4, n_3, n_3') \, \mbox{ satisfying } \, \mathscr{C}  \},$$ instead with cardinality,  $| S_{(n_1,n_1', n_4', n_5, \zeta)}| \, \lesssim \,  N_3^6 N_2^3 N_4$. 
 
Then for $\omega$ outside a set of measure $e^{-\frac{1}{\delta^r}}$,
\begin{eqnarray} \label{ref3Q2} 
\quad \mathcal{Q}_2  &\lesssim&   \delta^{-4\mu r}\sum_{n_1 \neq n_1'}  N_4^{-9} N_5^{-3}  \sum_{ n_4' \ne n_5}\, \left[ \sum_{ S_{(n_1,n_1', n_4', n_5, \zeta)}}\, k^{\tilde C}_n  \, k^{\tilde C}_{n'} \, \ca_{n_3} \, a_{n_3'} \right]^2 \\
 &\lesssim&  \delta^{-4\mu r} \sum_{n_1 \neq n_1'} N_4^{-9} N_5^{-3}  N_3^6 N_2^3 N_4   \sum_{n_4',  n_5}\, \sum_{ S_{(n_1,n_1', n_4', n_5, \zeta)}}\, |k^{\tilde C}_n|^2  |k^{\tilde C}_{n'}|^2 |\ca_{n_3}|^2 |a_{n_3'}|^2 \notag\\
 &\lesssim&   N_4^{-9} N_5^{-3}  N_3^6 N_2^3 N_4  \, \sum_{n, n', n_3, n_3'}\, |S_{(n, n', n_3, n_3', \zeta)}|\, |k^{\tilde C}_n|^2  |k^{\tilde C}_{n'}|^2 |\ca_{n_3}|^2    
 |a_{n_3'}|^2  \notag\\
 &\lesssim&  \delta^{-4\mu r} N_4^{-9} N_5^{-3}  N_3^6 N_2^3  N_4     \,  N_2^3 N_5^3 N_4^2\,  \|k^{\tilde C}_n\|_{\ell^2}^2  \|k^{\tilde C}_{n'}\|_{\ell^2}^2 \|a_{n_3}\|_{\ell^2}^2 \|a_{n_3'}\|_{\ell^2}^2 \notag\\
 &\lesssim&   \delta^{-4\mu r}N_4^{-6} N_3^6 N_2^6 \,  \|\chi_{\tilde C}(n)k_n\|_{\ell^2}^2  \|\chi_{\tilde C}(n')k_{n'}\|_{\ell^2}^2 \|a_{n_3}\|_{\ell^2}^2 \|a_{n_3'}\|_{\ell^2}^2, \notag
\end{eqnarray} where now 
$$S_{(n, n', n_3, n_3', \zeta)}\, :=\, \{  (n_1, n_1', n_2, n_4, n_5', n_5) \, \mbox{ satisfying } \, \mathscr{C}  \}$$ has cardinality  less than  or equal to  $ N_2^3 N_5^3 N_4^2 $. 

Thus the contribution of $\Delta\zeta \,M_2$,  is bounded by $N_4^{-1} N_5^{-2}N_3^3 N_2^3$. After taking the square root and normalizing we obtain in this case a bound of $N_4^{-\frac{3}{2}+\alpha} N_5^{2- 2 s+ \alpha}$ which suffices provided $s > 1 + \frac{\alpha}{2}$.

\smallskip

\noindent {\bf Case $\beta_6$:} In this case we have that  $n_4=n_5'$ and that $n_5 = n_4'$ and \eqref{m22expression} now has enough decay to use Lemma \ref{supg}. We define $$S_{(n_1,n_1', \zeta)} \, :=\, \{  (n, n', n_2, n_3, n_3', n_4, n_4') \, \mbox{ satisfying } \, \mathscr{C}   \}$$ with cardinality $|{S_{(n_1,n_1', \zeta)}}| \lesssim N_3^6 N_2^3 N_4^4$ and proceed as follows for $\omega$ outside a set of measure $e^{-\frac{1}{\delta^r}}$:
\begin{eqnarray} \label{m22forbeta6} 
&&\sum_{n_1\ne n_1'}\left| \sum_{S_{(n_1,n_1', \zeta)}}  k^{\tilde C}_n  \, k^{\tilde C}_{n'} \ca_{n_3} a_{n_3'}\,  \frac{|g_{n_4}(\omega)|^2}{|n_4|^3} \,  \frac{|g_{n_4'}(\omega)|^2}{|n_4'|^3} \right|^2 \\
&\lesssim& N_4^{-12+\varepsilon} \sum_{n_1\ne n_1'} \, |{S_{(n_1,n_1', \zeta)}}|  \, \sum_{S_{(n_1,n_1', \zeta)}} |k^{\tilde C}_n|^2  |k^{\tilde C}_{n'}|^2 |\ca_{n_3}|^2    
 |a_{n_3'}|^2,  \notag \\
 &\lesssim& N_4^{-12+\varepsilon} N_3^6 N_2^3 N_4^4  N_2^3 N_4^4 \sum_{n, n', n_3, n_3'}\, | S_{(n, n', n_3, n_3', \zeta)}| |k^{\tilde C}_n|^2  |k^{\tilde C}_{n'}|^2 |\ca_{n_3}|^2     |a_{n_3'}|^2  \notag\\
  &\lesssim& N_4^{-4+\varepsilon} N_3^6 N_2^6   \,  \|\chi_{\tilde C}(n)k_n\|_{\ell^2}^2  \|\chi_{\tilde C}(n')k_{n'}\|_{\ell^2}^2 \|a_{n_3}\|_{\ell^2}^2 \|a_{n_3'}\|_{\ell^2}^2, \label{boundforbeta6}
 \end{eqnarray} where now to obtain \eqref{boundforbeta6} we have used that 
$$S_{(n, n', n_3, n_3', \zeta)}\, :=\, \{  (n_1, n_1', n_2, n_4, n_4') \, \mbox{ satisfying } \, \mathscr{C}  \}$$ has cardinality  less than  or equal to  $ N_2^3 N_4^4 $. 

Thus the contribution of $\Delta\zeta \,M_2$,  is bounded by $N_4^{\varepsilon} N_3^3 N_2^3$. After taking the square root and normalizing we obtain in this case a bound of $N_4^{-1+\alpha+ \varepsilon} N_5^{2- 2 s+ \alpha}$ which suffices provided $s > 1 + \frac{\alpha}{2}$.

\smallskip
\noindent {\bf Case $\beta_7$:} In this case we have that  $n_4=n_4'$ and that $n_5 = n_5'$ and once again \eqref{m22expression} has enough decay to use Lemma \ref{supg}. Define $$S_{(n_1,n_1', \zeta)} \, :=\, \{  (n, n', n_2, n_3, n_3', n_4, n_5) \, \mbox{ satisfying } \, \mathscr{C}   \}$$ with cardinality $|{S_{(n_1,n_1', \zeta)}}| \lesssim N_3^6 N_2^3 N_5^3 N_4$ and proceed as follows for $\omega$ outside a set of measure $e^{-\frac{1}{\delta^r}}$ :

\begin{eqnarray} \label{m22forbeta7} 
&&\sum_{n_1\ne n_1'}\left| \sum_{S_{(n_1,n_1', \zeta)}}  k^{\tilde C}_n  \, k^{\tilde C}_{n'} \ca_{n_3} a_{n_3'}\,  \frac{|g_{n_4}(\omega)|^2}{|n_4|^3} \,  \frac{|g_{n_5}(\omega)|^2}{|n_5|^3} \right|^2 \\
&\lesssim& N_4^{-6+\varepsilon} N_5^{-6} \sum_{n_1\ne n_1'} \, |{S_{(n_1,n_1', \zeta)}}|  \, \sum_{S_{(n_1,n_1', \zeta)}} |k^{\tilde C}_n|^2  |k^{\tilde C}_{n'}|^2 |\ca_{n_3}|^2    
 |a_{n_3'}|^2,  \notag \\
 &\lesssim&   N_4^{-6+ \varepsilon} N_5^{-6} \, N_3^6 N_2^3 N_5^3 N_4 \sum_{n, n', n_3, n_3'}\, |S_{(n, n', n_3, n_3', \zeta)}| |k^{\tilde C}_n|^2  |k^{\tilde C}_{n'}|^2 |a_{n_3}|^2     |a_{n_3'}|^2  \notag\\
  &\lesssim&   N_4^{-4+ \varepsilon}  N_3^6 N_2^6  \,  \|\chi_{\tilde C}(n)k_n\|_{\ell^2}^2  \|\chi_{\tilde C}(n')k_{n'}\|_{\ell^2}^2 \|a_{n_3}\|_{\ell^2}^2 \|a_{n_3'}\|_{\ell^2}^2, \label{boundforbeta7}
 \end{eqnarray} where now to obtain \eqref{boundforbeta7} we have used that 
$$S_{(n, n', n_3, n_3', \zeta)}\, :=\, \{  (n_1, n_1', n_2, n_4, n_5) \, \mbox{ satisfying } \, \mathscr{C}  \, \mbox{ for fixed  } \, (n, n', n_3, n_3', \zeta) \}$$ has cardinality  less than  or equal to  $ N_2^3 N_5^3 N_4 $. 

Thus the contribution of $\Delta\zeta \,M_2$,  is bounded by $N_4^{\varepsilon} N_3^3 N_2^3$. After taking the square root and normalizing we obtain in this case a bound of $N_4^{-1+\alpha+  \varepsilon} N_5^{2- 2 s+ \alpha}$ which suffices provided $s > 1 + \frac{\alpha}{2}$.

\medskip 

{\bf $\bullet$  Case c, ii)}  In this case we have that $N_0\sim N_1 \geq N_4 \geq N_2 \geq N_5 \geq N_3$.  As in {\bf Case c) i)} after duality, changing variables $\zeta:=m-|n_1|^2$ and cutting the window $N_1$ by cubes $C$ of sidelength $N_4$ we have to estimate expression \eqref{sum6}. Since $\Delta \zeta \sim N_4^2$ we once again bound \eqref{sum6} by $$\|\chi_{C}a_{n_1}\|_{\ell^2}^2 \|\gamma\|_{\ell_\zeta^2}^2N_4^2  \sup_{\zeta} \, \sum_{n_1 \in C} \left| \sigma_{n_1, n_2}  \ca_{n_2} \right|^2 \notag\\
 \, \leq \, N_4^2 \, \|\chi_{C}a_{n_1}\|_{\ell^2}^2 
 \|\gamma\|_{\ell_\zeta^2}^2 \|a_{n_2} \|_{\ell^2}^2 \,\,  \sup_{\zeta} \,\| \mathcal{G}  \mathcal{G}^{\ast}\|, \, $$ where $ \sigma_{n_1, n_2}$ is defined as in \eqref{sigma12} and  $\GG$ denotes, as usual, the matrix of entries $ \sigma_{n_1, n_2}$.  Just as in  {\bf Case c) i)}  we are then reduced to estimating $M_1$ and $M_2$ as defined
\eqref{gigistar}. 

To estimate $M_1$ we proceed just as in \eqref{sum7} to obtain for $\omega$ outside a set of measure $e^{-\frac{1}{\delta^r}}$ the same bound 
\begin{equation} \label{sum9}
M_1  \lesssim \,   \delta^{-2\mu r}N_4^{-2} N_2^3 N_3^2 \, \| \chi_{\tilde C}(n)k_n\|_{\ell^2}^2 \|a_{n_3}\|_{\ell^2}^2.
\end{equation} 
Hence $\Delta \zeta M_1$ is bounded once again by 
$$\delta^{-2\mu r} N_2^3 N_3^2 \,\,  \|\chi_{C}a_{n_1}\|_{\ell^2}^2 \|a_{n_2} \|_{\ell^2}^2 \|a_{n_3}\|_{\ell^2}^2 \, \|\gamma\|_{\ell_\zeta^2}^2  \| \chi_{\tilde C}(n)k_n\|_{\ell^2}^2.$$ which after taking square root and normalizing gives  in this case the bound $N_4^{-s + \frac{1}{2} + \alpha}$,  which suffices provided $s > \frac{1}{2}+ \alpha$. 

The estimate for $M_2$ proceeds as in  {\bf Case c) i)} by analyzing cases $\beta_1 - \beta_7$ as stated there, yielding the same bounds  for $\Delta \zeta M_2$. We do not repeat the arguments but  rather indicate the bound we obtain in each case after taking square root and normalizing since now $N_2 \ge N_5 \ge N_3$,  so we need to trade the slower decay of the random term $\tilde u_5$  for the better regularity of the deterministic function $ \tilde u_2$ . 

\noindent {\bf Case $\beta_1$:}  In this case we have that the contribution of $\Delta \zeta M_2$ is bounded by $N_3^3 N_2^{\frac{5}{2}}$. Taking square root and normalizing we now obtain the bound  $N_4^{-1 +\alpha}N_2^{\frac{5}{4} -s} N_5^{\frac{1}{2} -s + \alpha}$,  which suffices provided $s > \frac{1}{2}+ \alpha$ and $\alpha <1$.

\medskip

\noindent {\bf Case $\beta_2$ and  $\beta_3$:}  In these cases we have that the contribution of $\Delta \zeta M_2$ is bounded by  $N_3^3 N_2^3$. Taking square root and normalizing we thus obtain the bound  $N_4^{\frac{1}{2} -s + \alpha} N_5^{\frac{1}{2} -s + \alpha}$,  which suffices provided $s > \frac{1}{2}+ \alpha$.

\medskip

\noindent {\bf Case $\beta_4$:} In this case we have that the contribution of $\Delta \zeta M_2$ is bounded by $ N_4^{-\frac{1}{2}}N_2^3 N_3^3 $. Taking square root and normalizing gives the bound  $ N_4^{\frac{1}{4}-s + \alpha} N_5^{\frac{1}{2}-s + \alpha}  $,  which suffices provided $s > \frac{1}{2}+ \alpha$.

\medskip

\noindent {\bf Case $\beta_5$:}  In this case we have that the contribution of $\Delta \zeta M_2$ is bounded by $ N_4^{-1} N_5^{-2}N_2^3 N_3^3 $ which is smaller than the bound in   {\bf Case $\beta_4$}. 

\medskip

\noindent {\bf Case $\beta_6$} and {\bf Case $\beta_7$:}  In this case we have that the contribution of $\Delta \zeta M_2$ is bounded by $ N_4^{\varepsilon} N_2^3 N_3^3 $.  After taking square root and normalizing we get the bound  $N_4^{\frac{1}{2} -s + \alpha+ \varepsilon} N_5^{\frac{1}{2} -s + \alpha}$,  which once again suffices provided $s > \frac{1}{2}+ \alpha$.

\bigskip

{\bf $\bullet$ Case c, iii)}  In this case we have that $N_0\sim N_1 \geq N_4 \geq N_2 \geq N_3 \geq N_5$. Since $\Delta \zeta M_1$ is bounded by 
$$\delta^{-2\mu r}N_2^3 N_3^2 \,\,  \|\chi_{C}a_{n_1}\|_{\ell^2}^2 \|a_{n_2} \|_{\ell^2}^2 \|a_{n_3}\|_{\ell^2}^2 \, \|\gamma\|_{\ell_\zeta^2}^2  \| \chi_{\tilde C}(n)k_n\|_{\ell^2}^2,$$ we have, after taking square root and normalizing the bound $N_4^{-s + \frac{1}{2} + \alpha}$  just as before. The latter suffices provided $s > \frac{1}{2}+ \alpha$. 
For $M_2$, following the scheme presented above for {\bf Case c, ii)} we now have:

\noindent {\bf Case $\beta_1$:}   Since the contribution of $\Delta \zeta M_2$ is bounded by $N_3^3 N_2^{\frac{5}{2}}$, we obtain, after taking square root and normalizing,  the bound  $N_4^{\frac{7}{4} -2 s +\alpha}$  which suffices provided $s > \frac{7}{8}+ \frac{\alpha}{2}$. 

\medskip

\noindent {\bf Case $\beta_2$ and  $\beta_3$:}  In these cases we have that the contribution of $\Delta \zeta M_2$ is bounded by  $N_3^3 N_2^3$. Taking square root and normalizing we thus obtain the bound  $N_4^{2 -2s + \alpha}$,  which suffices provided $s > 1 + \frac{\alpha}{2}$.

\medskip

\noindent {\bf Case $\beta_4$:} In this case we have that the contribution of $\Delta \zeta M_2$ is bounded by $ N_4^{-\frac{1}{2}}N_2^3 N_3^3 $ which is smaller than the bound in   {\bf Cases $\beta_2$, $\beta_3$}.

\medskip

\noindent {\bf Case $\beta_5$:}  In this case we have that the contribution of $\Delta \zeta M_2$ is bounded by $ N_4^{-1} N_5^{-2}N_2^3 N_3^3 $ which is smaller than the bound in   {\bf Case $\beta_4$}. 

\medskip

\noindent {\bf Case $\beta_6$} and {\bf Case $\beta_7$:}  In these cases we have that the contribution of $\Delta \zeta M_2$ is bounded by $ N_4^{\varepsilon} N_2^3 N_3^3 $.  After taking square root and normalizing we get the bound  $N_4^{2 - 2s+ \alpha+ \varepsilon}$,  which once again suffices provided $s >  1+ \frac{\alpha}{2}$.
 
\bigskip

{\bf $\bullet$  Case c), iv), v), vi)} $N_1 \sim N_0$ and $N_2 \ge N_4$ and {\bf Case d):} \, $N_1\sim N_2 \geq N_0, N_4$

\medskip 

In this case we proceed as in Subsection \ref{ddddd}. Assume $N_0\sim N_1\geq N_2\geq N_4$,  Case d) having similar or better bounds. The estimates of the trilinear expressions will give after normalization
$$N_0^sN_1^{-s}N_2^{-s+1}N_3^{1-s} N_4^\alpha N_5^\alpha$$
and we assume that $s>1+\alpha$.

One also needs to estimate the terms in \eqref{addons}.  Here we show how to estimate the term involving the random  function at frequency $N_4$ in \eqref{norml2xt2}. We first observe that in order for this term not to be zero it must be that $N_4\sim N_0$. Then 
 for $v_0^\omega$ in \eqref{linear}, after normalization we have  the bound
\begin{eqnarray*}
&& N_0^sN_1^{-s}N_2^{-s}N_3^{-s}N_4^{-1+\alpha}N_5^{-1+\alpha}\\
&\times&\|P_{N_0}P_{N_4} v_0^\omega\|_{L^\infty_tH^{1-\alpha}_x}N_5^{\frac{1}{4}}\|D^{1-\alpha}(P_{N_0}P_{N_5} v_0^\omega)\|_{L^4_tL^4_x}
 \prod_{j=1,2,3}
N_j^{\frac{1}{4}}\|P_{N_j}  u_j(x,t)\|_{U^4_\Delta H^s}.
\end{eqnarray*}
 The latter together with the Strichartz estimate \eqref{Strichartz-1} are enough to obtain the desired bound since for $\alpha<\frac{3}{4}$, we have:
$$N_0^sN_0^{-s+\frac{1}{4}-1+\alpha}\sim N_0^{-\frac{3}{4}+\alpha}<1.$$

\bigskip
 
\subsubsection{ The $DDRRR$ Case}\label{ddrrr} To estimate the expression  in \eqref{expression} we first observe that 
in terms of bars we need to estimate only  the following cases: \underbar{Case 1:} $u_1, u_3, u_5$ are random, that is none of the random functions are conjugated, or {\underbar{Case 2:} only one of these functions is conjugated, the other cases are obtained by conjugating the whole expression in \eqref{expression}. We will remark later on how the estimates change depending on these two Cases. 

 We now assume that the first three functions are random and the last two are deterministic. We also assume that $N_1\geq N_2\geq N_3$ and $N_4\geq N_5$. 
We then have the following subcases:
\begin{itemize}
\item{\bf Case a):} $N_4=\max{(N_1,N_4)}$
\begin{itemize}
\item Case i) $N_2\leq N_5\leq N_4$
\item Case ii) $N_2\leq N_5\leq N_1\leq N_4$
\item Case iii) $N_3\leq N_5\leq N_2\leq N_1\leq N_4$
\item Case iv) $N_5\leq N_3\leq N_2\leq N_1\leq N_4$
\end{itemize}

\medskip
\item{\bf Case b):} $N_1=\max{(N_1,N_4)}$ and $N_2\geq N_4$
\begin{itemize}
\item Case i) $N_3\geq N_4$
\item Case ii) $N_5\leq N_3\leq N_4\leq N_2$
\item Case iii) $N_3\leq N_5\leq  N_4\leq N_2$

\end{itemize}

\medskip
\item{\bf Case c):} $N_1=\max{(N_1,N_4)}$ and $N_4\geq N_2$
\begin{itemize}
\item Case i) $N_2\leq N_5\leq N_4\leq N_1$
\item Case ii) $N_3\leq N_5\leq N_2\leq N_4\leq N_1$
\item Case iii) $N_5\leq N_3\leq N_2\leq N_4\leq N_1$
\end{itemize}
\end{itemize}

\noindent
{\bf Case a), i):} In this case we proceed as in Subsection \ref{ddddd}. Assume for simplicity that $N_0\sim N_4$, the other cases are smoother. The estimates of the trilinear expressions will give after normalization
$$N_0^sN_4^{-s}N_5^{-s+1}N_3^\alpha N_2^\alpha N_1^\alpha$$
and we assume that $s>1+3\alpha$.
One also needs to estimate the terms in \eqref{addons}.  Here we show how to estimate the factor involving the random term at frequency $N_1$ in \eqref{norml2xt2}. We have
 for $v_0^\omega$ in \eqref{linear} 
 \begin{equation}\label{randompiece1}
 N_0^sN_1^{-1+\alpha}N_2^{-1+\alpha}N_3^{-1+\alpha}\|P_{N_0}P_{N_3} v_0^\omega\|_{L^\infty_tH^{1-\alpha}_x}\prod_{j=4,5}
N_j^{\frac{1}{4}-s}\|P_{N_j}  u_j(x,t)\|_{U^4_\Delta H^s}.
\end{equation}
where we notice that $N_1\sim N_0$ otherwise the contribution would be null. This is enough to obtain the desired bound since 
$$N_0^sN_0^{-1+\alpha}N_4^{\frac{1}{4}-s}\sim N_4^{-\frac{3}{4}+\alpha}.$$

Also note that this case is not effected by conjugation hence it is the same both in Case 1 and Case 2.

\medskip
\noindent
{\bf Case a), ii).} We also assume that $N_4\sim N_0$, this is the least favorable situation.
We proceed by duality and a change of variables $\zeta=m\pm|n_4|^2$ as in the proof of Proposition \ref{One} in particular see \eqref{t1dual}.  
We have to bound
 \begin{equation}\label{sum1ddrrr}\|\gamma\|_{\ell_\zeta^2}^2\|a_{n_4}\|_{\ell^2}^2
 \sum_{(\zeta, n_4) \in \Z\times \Z^3} \,\left|   \sum_{\substack{ n= \pm n_1\pm n_2 \pm  n_3\pm n_4\pm n_5,\\ n_i,n_j,n_k\ne n_r,n_p 
\\ \zeta = \pm |n_1|^2  \pm |n_2|^2\pm  |n_3|^2\pm |n_5|^2}} 
\frac{\tilde g_{n_1}(\omega)}{|n_1|^{\frac{3}{2}}} \frac{\tilde g_{n_2}(\omega)}{|n_2|^{\frac{3}{2}}} \frac{\tilde g_{n_3}(\omega)}{|n_3|^{\frac{3}{2}}}  k_n\tilde a_{n_5} \right|^2.
\end{equation}
We now consider two cases:
\begin{itemize}
\item {\bf Case $A_0$:} $n_1,n_2,n_3$ are all different from each other.
\item {\bf Case $A_1$:} At least two of the frequencies $n_1,n_2,n_3$ are equal.
\end{itemize}

\noindent
{\bf Case $A_0$:} We define the set 
$$S_{(\zeta, n_4,n_1,n_2,n_3)}=\left\{(n,n_5) \, : \,  \begin{aligned}&n= \pm n_1\pm n_2 \pm  n_3\pm n_4\pm n_5,\\ &n_i,n_j,n_k\ne n_r,n_p, \\ 
&\zeta =\pm |n_1|^2  \pm |n_2|^2\pm  |n_3|^2\pm |n_5|^2\end{aligned}\right\}
$$
with $|S_{(\zeta, n_4,n_1,n_2,n_3)}|\lesssim N_5^2$ and we write
\begin{eqnarray*}\eqref{sum1ddrrr}&\lesssim&\|\gamma\|_{\ell_\zeta^2}^2\|a_{n_4}\|_{\ell^2}^2\\
 &&\sum_{(\zeta, n_4) \in \Z\times \Z^3} \,N_1^{-3}N_2^{-3}N_3^{-3}\left|   \sum_{n_1,n_2,n_3} 
\tilde g_{n_1}(\omega)\tilde g_{n_2}(\omega)\tilde g_{n_3}(\omega)\sum_{S_{(\zeta, n_4,n_1,n_2,n_3)}}   k_n\tilde a_{n_5} \right|^2
\end{eqnarray*}
By using Lemma \ref{largedeviation} we can continue, for $\omega$ outside a set of measure $e^{-\frac{1}{\delta^r}}$, with 
\begin{eqnarray*}&\lesssim &\delta^{-2\mu r} \|\gamma\|_{\ell_\zeta^2}^2\|a_{n_4}\|_{\ell^2}^2\\
&&\sum_{(\zeta, n_4) \in \Z\times \Z^3} \,N_1^{-3}N_2^{-3}N_3^{-3}  \sum_{n_1,n_2,n_3} 
\sum_{S_{(\zeta, n_4,n_1,n_2,n_3)}} |S_{(\zeta, n_4,n_1,n_2,n_3)}| |k_n|^2|a_{n_5}|^2\\
&\lesssim &\delta^{-2\mu r}\|\gamma\|_{\ell_\zeta^2}^2\|a_{n_4}\|_{\ell^2}^2\,N_1^{-3}N_2^{-3}N_3^{-3} N_5^2 \sum_{n,n_5}|k_n|^2|a_{n_5}|^2|S_{(n_5,n)} |
\end{eqnarray*}
where
$$S_{(n,n_5)}=\left\{(\zeta, n_4, n_1,n_2,n_3) \, : \,  \begin{aligned}&n= \pm n_1\pm n_2 \pm  n_3\pm n_4\pm n_5,\\ &n_i,n_j,n_k\ne n_r,n_p, \\ 
&\zeta =\pm |n_1|^2  \pm |n_2|^2\pm  |n_3|^2\pm |n_5|^2\end{aligned}\right\}
$$
and $|S_{(n,n_5)}|\lesssim N_1^3N_2^3N_3^3$, where we used that   $\Delta \zeta\lesssim N_1^2$. Hence we can continue with 
$$\lesssim \delta^{-2\mu r}\|\gamma\|_{\ell_\zeta^2}^2\|a_{n_4}\|_{\ell^2}^2 N_5^2 \|k_n\|_{\ell^2}^2\|a_{n_5}\|_{\ell^2}^2$$
and after taking square root and normalizing we obtain the bound
$$N_5^{1-s}N_1^{-1+\alpha}N_2^{-1+\alpha}N_3^{-1+\alpha}.$$ We note that this case is the same both in Case 1 and Case 2.

\medskip
\noindent
{\bf Case $A_1$:} We first assume that only two frequencies are equal. The important remark is that we have removed the frequencies that would give rise to $|g_n(\omega)|^2$ so in \eqref{sum1ddrrr} we would see
either $(\tilde g_{n_1})^2(\omega)\tilde g_{n_3}(\omega)$ or $\tilde g_{n_1}(\omega)(\tilde g_{n_2})^2(\omega)$. In both cases we can still use  Lemma \ref{largedeviation}   and proceed as above to obtain in fact  better estimates since the cardinality of the sets involved are smaller due to the collapse of the frequencies that are equal. 

\medskip

If all three frequencies are equal, and this can happen only in Case 2,  then we have that $N_1\sim N_2\sim N_3$ and 
\begin{eqnarray*}\eqref{sum1ddrrr}&\lesssim&\|\gamma\|_{\ell_\zeta^2}^2\|a_{n_4}\|_{\ell^2}^2
 \sum_{(\zeta, n_4) \in \Z\times \Z^3} \,N_1^{-12} \sum_{n_3} 
|g_{n_3}(\omega)|^3\left|  \sum_{S_{(\zeta, n_4,n_3)}}  k_n\tilde a_{n_5} \right|^2
\end{eqnarray*}
where $$S_{(\zeta, n_4,n_3)}=\left\{(n,n_5) \, : \,  \begin{aligned}&n= \pm  3 n_3\pm n_4\pm n_5, \\ 
&\zeta =\pm 3 |n_3|^2  \pm |n_5|^2\end{aligned}\right\}.$$
 Then by using Lemma \ref{supg} we can continue, for $\omega$ outside a set of measure $e^{-\frac{1}{\delta^r}}$,  with
$$\lesssim \|\gamma\|_{\ell_\zeta^2}^2\|a_{n_4}\|_{\ell^2}^2
 \sum_{(\zeta, n_4) \in \Z\times \Z^3} \,N_1^{-12+\varepsilon}  \sum_{S_{(\zeta, n_4)}} 
 |k_n|^2|a_{n_5}|^2|S_{(\zeta, n_4)}|,$$
where
$$S_{(\zeta, n_4)}=\left\{(n,n_3,n_5) \, : \,  \begin{aligned}&n= \pm 3 n_3\pm n_4\pm n_5, \\ 
&\zeta =\pm 3  |n_3|^2  \pm |n_5|^2\end{aligned}\right\}.$$
with $|S_{(\zeta, n_4)}|\lesssim N_5^2N_3$, and we continue with
$$\lesssim \|\gamma\|_{\ell_\zeta^2}^2\|a_{n_4}\|_{\ell^2}^2
N_1^{-12+\varepsilon}N_5^2N_3 \sum_{n,n_5} |k_n|^2|a_{n_5}|^2\,  |S_{(n, n_5)}|,$$
where
$$S_{(n, n_5)}=\left\{(\zeta, n, n_3, n_4) \, : \,  \begin{aligned}&n= \pm  3 n_3\pm n_4\pm n_5, \\ 
&\zeta =\pm 3  |n_3|^2  \pm |n_5|^2\end{aligned}\right\}.$$
with $|S_{(n, n_5)}|\lesssim N_3^3$. We obtain the bound $N_1^{-6+\varepsilon}$ which clearly suffices without any further restriction when we take square root and normalize. 

\medskip
We now  observe that {\bf Cases a) iii) } and {\bf iv) }can be analyzed just like {\bf  Cases a) i) } since $N_4$ and $N_1$ are still the top frequencies and the order of the rest is not relevant.

\medskip
{\bf Case b), i)} We assume first that  $N_1\sim N_0$. We cut the frequency windows $N_0$ and $N_1$ by cubes  $C$ of sidelength $N_2$. After using Cauchy-Schwarz we need to estimate
\begin{equation} \label{rrrdd2}
\sum_{m\in \Z, n \in C} \,\left|  \sum_{ n_1, n_2, n_3; \, \, n_1\in C} \tilde g_{n_1}(\omega)\tilde g_{n_2}(\omega)\tilde g_{n_3}(\omega)
\left[\sum_ {S_{(m,n,n_1,n_2,n_3)}} \frac{1}{|n_1|^{\frac{3}{2}}}\frac{1}{|n_2|^{\frac{3}{2}}}\frac{1}{|n_3|^{\frac{3}{2}}}\tilde a_{n_5}\tilde a_{n_4}\right]  
\right|^2 
\end{equation}
where
$$ S_{(m,n,n_1,n_2,n_3)}=\left\{(n_4, n_5) \, : \,  \begin{aligned}&n= \pm n_1\pm n_2 \pm  n_3\pm n_4\pm n_5,\\ &n_i,n_j,n_k\ne n_r,n_p, \, 
m=\pm |n_1|^2  \pm |n_2|^2\pm  |n_3|^2\pm |n_5|^2\pm |n_4|^2\end{aligned}\right\},$$
with $|S_{(m,n,n_1,n_2,n_3)}|\lesssim N_5^2$.
 We now consider two cases:
\begin{itemize}
\item {\bf Case $A_0$:} $n_1,n_2,n_3$ are all different from each other.
\item {\bf Case $A_1$:} At least two of the frequencies $n_1,n_2,n_3$ are equal.
\end{itemize}

\noindent
{\bf Case $A_0$:} We use Lemma \ref{largedeviation} and, for $\omega$ outside a set of measure $e^{-\frac{1}{\delta^r}}$,  we have  
\begin{eqnarray*} \eqref{rrrdd2}&\lesssim&\delta^{-2\mu r}
\sum_{m\in \Z, n \in C} \, \sum_{ n_1, n_2, n_3; \, \, n_1\in C} \left[\sum_ {S_{(m,n,n_1,n_2,n_3)}} \frac{1}{|n_1|^{\frac{3}{2}}}\frac{1}{|n_2|^{\frac{3}{2}}}\frac{1}{|n_3|^{\frac{3}{2}}}\tilde a_{n_5}\tilde a_{n_4}\right]^2 \\
&\lesssim& \delta^{-2\mu r}N_1^{-3} N_2^{-3} N_3^{-3}\sum_{m\in \Z, n \in C} \, \sum_{S_{(m,n)}} \#S_{(m,n,n_1,n_2,n_3)}|a_{n_5}|^2 |a_{n_4}|^2\\
&\lesssim& \delta^{-2\mu r}N_1^{-3} N_2^{-3} N_3^{-3}N_5^2\sum_{n_4,n_5} |a_{n_5}|^2 |a_{n_4}|^2|S_{(n_4,n_5)}|
\end{eqnarray*}
where
$$ S_{(n_4,n_5)}=\left\{(m, n, n_1,n_2, n_3) \, : \,  \begin{aligned}&n= \pm n_1\pm n_2 \pm  n_3\pm n_4\pm n_5,\\
&m=\pm |n_1|^2  \pm |n_2|^2\pm  |n_3|^2\pm |n_5|^2\pm |n_4|^2\end{aligned}\right\},$$
with $|S_{(n_4,n_5)}|\lesssim N_1N_2N_1N_2^3N_3^3$, which finally gives
$$\eqref{rrrdd2}\lesssim \delta^{-2\mu r}N_1^{-1}N_2\|a_{n_5}\|_{\ell^2}^2 \|a_{n_4}\|_{\ell^2}^2.$$
By taking square root and normalizing we require that 
$$N_0^sN_1^{-\frac{3}{2}+\alpha}N_2^{-\frac{1}{2}+\alpha}N_3^{-1+\alpha}N_4^{-s}N_5^{-s}\lesssim N_1^{-\beta}$$
and this follows from assuming $s<\frac{3}{2}-\alpha$.

\medskip

\noindent
{\bf Case $A_1$:} We proceed just like in  the same case for {\bf Case a), ii)}. Here we only work out the details for the case when all frequencies are equal, again this can happen only in Case 2.   We have that $N_1\sim N_2\sim N_3$ and 
\begin{eqnarray*}\eqref{rrrdd2}&\lesssim& \sum_{(m, n) \in \Z\times \Z^3} \,N_1^{-12}\left|   \sum_{n_3} 
|g_{n_3}(\omega)|^3\sum_{S_{(m, n, n_3)}}  \tilde a_{n_4}\tilde a_{n_5} \right|^2
\end{eqnarray*}
where $$S_{(m, n, n_3)}=\left\{(n_4, n_5) \, : \,  \begin{aligned}&n= \pm  3 n_3\pm n_4\pm n_5, \\ 
&m =\pm 3 |n_3|^2 \pm |n_4|^2 \pm |n_5|^2\end{aligned}\right\}.$$
 Then by using Lemma \ref{supg}, for $\omega$ outside a set of measure $e^{-\frac{1}{\delta^r}}$,  we can continue with
$$\lesssim \sum_{(m, n) \in \Z\times \Z^3} \,N_1^{-12+\varepsilon}  \sum_{S_{(m, n)}} 
 |a_{n_4}|^2|a_{n_5}|^2|S_{(m, n)}|,$$
where
$$S_{(m, n)}=\left\{(n_3,n_4,n_5) \, : \,  \begin{aligned}&n= \pm 3 n_3\pm n_4\pm n_5, \\ 
&m =\pm 3  |n_3|^2  \pm |n_4|^2 \pm |n_5|^2\end{aligned}\right\}.$$
with $|S_{(m, n)}|\lesssim N_5^2N_3$, and we continue with
$$\lesssim N_1^{-12+\varepsilon}N_5^2N_3 \sum_{n_4,n_5} |a_{n_4}|^2|a_{n_5}|^2\,  |S_{(n_4, n_5)}|,$$
where
$$S_{(n_4, n_5)}=\left\{(m, n, n_3) \, : \,  \begin{aligned}&n= \pm  3 n_3\pm n_4\pm n_5, \\ 
&m =\pm 3  |n_3|^2   \pm |n_4|^2\pm |n_5|^2\end{aligned}\right\}.$$
with $|S_{(n_4, n_5)}|\lesssim N_1^2N_3$. We obtain the bound $N_1^{-6+\varepsilon}$ which clearly suffices without any further restriction when we take square root and normalize. 

\medskip 
\noindent
Now assume that  $N_1\sim N_2$. Here we do not need to cut with cubes $C$, but the argument  and the estimates are similar as the ones we just analyzed.

\medskip
{\bf Case b), ii), iii):} These cases are estimated just like the case we just analyzed since the two highest frequencies are still $N_1$ and $N_2$ and the order of the others is not relevant.

\medskip
{\bf Case c), i):} Assume first $N_0\sim N_1$. This case is handled like {\bf Case b) i)} above. Here we cut with cubes $C$ of sidelength $N_4$.  This  gives in particular that $\Delta m\lesssim N_1N_4$.

\noindent
{\bf Case $A_0$:} Just like in {\bf Case  b) i)}  we have, for $\omega$ outside a set of measure $e^{-\frac{1}{\delta^r}}$, that 
\begin{eqnarray*}
\eqref{rrrdd2}&\lesssim&\delta^{-2\mu r}
N_1^{-3} N_2^{-3} N_3^{-3}N_5^2\sum_{n_4,n_5} |a_{n_5}|^2 |a_{n_4}|^2|S_{(n_4,n_5)}|
\end{eqnarray*}
where now $|S_{(n_4,n_5)}|\lesssim N_1N_4N_1N_2^3N_3^3$, since $\Delta m\lesssim N_1N_4$. This  finally gives
$$\eqref{rrrdd2}\lesssim \delta^{-2\mu r}N_1^{-1}N_4\|a_{n_5}\|_{\ell^2}^2 \|a_{n_4}\|_{\ell^2}^2.$$
By taking square root and normalizing we require that 
$$N_0^sN_1^{-\frac{3}{2}+\alpha}N_2^{-1+\alpha}N_3^{-1+\alpha}N_4^{-s+\frac{1}{2}}N_5^{-s}\lesssim N_1^{-\beta}$$
and this follows from assuming again $s<\frac{3}{2}-\alpha$.

\medskip
\noindent
{\bf Case $A_1$:} This is like the same case for {\bf Case b) i)}.

\medskip
\noindent
{\bf Case c), i):} Now assume  $N_4\sim N_1$.  Here we do not need to cut and the same estimates as the ones we just presented hold. 

\medskip
\noindent
{\bf Case c), ii), iii):} These cases are estimated just like the case we just analyzed since the two highest frequencies are still $N_1$ and $N_4$ and the order of the others is not relevant.

\subsubsection{ The $DRRRR$ Case}\label{drrrr} To estimate the expression  in \eqref{expression} we assume without any loss of generality
 that $u_5$ is the deterministic function and it is not conjugated. By Cauchy-Schwarz and and Proposition  \ref{transfer} we are reduced to estimate 
 
 \smallskip
 \begin{equation} \label{rrrrd1}
\sum_{m\in \Z, n \in C} \,\left|  \sum_{\substack{ n_1, n_2, n_3, n_4\\ n_1,n_3\ne n_2, n_4 }} 
\left[\sum_ {\substack{ n_5= -n_1+n_2 - n_3+n_4-n,\\ n_5\ne n_2,n_4 
\\ m = |n_1|^2 -|n_2|^2+ |n_3|^2-|n_4|^2+ |n_5|^2}}  a_{n_5}\right] \frac{g_{n_1}(\omega)}{|n_1|^{\frac{3}{2}}} 
\frac{\cg_{n_2}(\omega)}{|n_2|^{\frac{3}{2}}}\frac{g_{n_3}(\omega)}{|n_3|^{\frac{3}{2}}}\frac{\cg_{n_4}(\omega)}{|n_4|^{\frac{3}{2}}}\right|^2, 
\end{equation} 
 where we have assumed that $\widehat{u_5}(n_5,t)=e^{it|n_5|^2}a_{n_5}$ and $C$ is a cube of sidelength to be determined later. 
 
 \medskip
Since we have removed the frequencies $n_1,n_3=n_2$ or $n_1,n_3=n_4$, which  would give rise to terms of the form $|g_i(\omega)|^2$,  we can  invoke Lemma 
\ref{largedeviation} and proceed by further considering 
 the following subcases. For $i, j \in\{1, 2, 3, 4\},$
 \begin{itemize}
 \item {\bf Case a):} There exists $j$ such that $N_0\sim N_j, \, N_5\lesssim N_j$. 
 \item {\bf Case b):} There exist $j\ne i $ such that $N_i\sim N_j$ and $ N_5, N_0\lesssim N_i$.
 \item {\bf Case c):} $N_0\sim N_5$ and $N_j\lesssim N_5$.
 \item {\bf Case d):} There exist $j\ne i$ such that $N_5\sim N_j$ and $N_0, N_i\lesssim N_j$.
 \end{itemize}
 
 \smallskip
 
 \noindent {\bf Case a):}  Assume $N_k, \, k\in \{1,2,3,4,5\}, \, k\neq j$ is the second largest frequency. Then let $C$ be of sidelength $N_k$ and let
 $$S_{(n, m, n_1, n_2, n_3, n_4)}=\left\{n_5\, : \, \begin{aligned} &n_5= -n_1+n_2 - n_3+n_4-n,\\ &n_5\ne n_2,n_4, \, \, n_j\in C
\\ &m = |n_1|^2 -|n_2|^2+ |n_3|^2-|n_4|^2+ |n_5|^2 \end{aligned}\right\}.$$
  By Lemma \ref{largedeviation}, for $\omega$ outside a set of measure $e^{-\frac{1}{\delta^r}}$, we have
 \begin{eqnarray*}
\eqref{rrrrd1}&\lesssim&\delta^{-2\mu r}\sum_{m\in \Z, n \in C} \, N_1^{-3}N_2^{-3}N_3^{-3}N_4^{-3}\sum_{\substack{ n_1, n_2, n_3, n_4 }} 
\left|\sum_ {S_{(n, m, n_1, n_2, n_3, n_4)}} a_{n_5}\right|^2\\
&\lesssim&\delta^{-2\mu r}\sum_{m \in \Z, \, n \in C} \, N_1^{-3}N_2^{-3}N_3^{-3}N_4^{-3}
\sum_{S_{(n,m)}} |a_{n_5}|^2 \end{eqnarray*}
 where
 $$S_{(n,m)}=\left\{(n_1,n_2,n_3,n_4,n_5) \, : \,  \begin{aligned}&n= n_1-n_2 +n_3-n_4+n_5,\\ &n_j\in C 
\\& m = |n_1|^2 -|n_2|^2+ |n_3|^2-|n_4|^2+ |n_5|^2\end{aligned}\right\}.$$
 We now define the set 
 $$S_{(n_5)}=\left\{(m, n, n_1, n_2,n_3, n_4) \, :\, \begin{aligned}&n= n_1-n_2 +n_3-n_4+n_5,\\ &n_j\in C 
\\& m = |n_1|^2 -|n_2|^2+ |n_3|^2-|n_4|^2+ |n_5|^2\end{aligned}\right\}$$
where $|S_{(n_5)}|\lesssim  N_j^2N_k^4N_p^3N_q^3$. Then we continue with 
 $$\eqref{rrrrd1}\lesssim \delta^{-2\mu r}N_1^{-3}N_2^{-3}N_3^{-3}N_4^{-3}\sum_{n_5}|a_{n_5}|^2\, |S_{(n_5)}|\lesssim \delta^{-2\mu r}N_j^{-1}N_k\|a_{n_5}\|^2_{\ell^2},$$
By taking square root and normalizing we obtain the bound 
 $N_j^{s+\alpha- \frac{3}{2}} N_k^{-\frac{1}{2}+\alpha}$ 
which entails  $s+\alpha<\frac{3}{2}  $ and $\alpha<\frac{1}{2}$.

 \medskip
{\bf Case b):} We go back to \eqref{rrrrd1} and we let $C$ be of its natural sidelength $N_0$. We then repeat the argument above with the role of $N_k$
played by $N_j$ and we count the set 
$$S_{(n_5)}=\left\{(m, n, n_1, n_2,n_3, n_4) \, :\, \begin{aligned}&n= n_1-n_2 +n_3-n_4+n_5, 
\\& m = |n_1|^2 -|n_2|^2+ |n_3|^2-|n_4|^2+ |n_5|^2\end{aligned}\right\}$$
where $|S_{(n_5)}|\lesssim  N_j^3N_i^3N_p^3N_q^3$.
By taking square root and normalizing we obtain the bound  $N_j^{s+2\alpha- 2}$ 
which entails  $s+2\alpha< 2 $.

\medskip
{\bf Case c):} 
We proceed as in {\bf Case b)} of Subsection \ref{ddddr}. More precisely by duality and a change of variables $\zeta=m-|n_5|^2$ as in the proof of Proposition
 \ref{One} in particular see \eqref{t1dual}.  Here we let $C$ be of its natural sidelength $N_0$. Let also $N_k$, where $N_k, \, k\in \{1,\dots, 4\}$ be the second
  largest frequency. We have to bound
 \begin{equation}\label{rrrrd2}\|\gamma\|_{\ell_\zeta^2}^2\|\chi_{C}a_{n_5}\|_{\ell^2}^2
 \sum_{(\zeta, n_5) \in \Z\times \Z^3} \,\left|   \sum_{\substack{ n= n_5-n_2 + n_3-n_4+n_1,\\ n_1,n_3,n_5\ne n_2,n_4 
\\ \zeta = |n_1|^2 -|n_2|^2+ |n_3|^2-|n_4|^2}} 
 \chi_{C}(n)k_n  \frac{g_{n_1}(\omega)}{|n_1|^{\frac{3}{2}}} 
\frac{\cg_{n_2}(\omega)}{|n_2|^{\frac{3}{2}}}\frac{g_{n_3}(\omega)}{|n_3|^{\frac{3}{2}}}\frac{\cg_{n_4}(\omega)}{|n_4|^{\frac{3}{2}}}\right|^2.
\end{equation}
We proceed again as above where now we have to replace $S_{(n_5)}$ by
$$S_{(n)}=\left\{(\zeta, n_1, n_2, n_3, n_4, n_5) \, :\, \begin{aligned}&n= n_1-n_2 +n_3-n_4+n_5, 
\\& \zeta = |n_1|^2 -|n_2|^2+ |n_3|^2-|n_4|^2\end{aligned}\right\}$$
where $|S_{(n)}|\lesssim  N_k^3N_i^3N_p^3N_q^3$, where we used that $\Delta \zeta\lesssim N_k^2$.
By taking square root and normalizing we obtain the bound  $N_k^{\alpha- 1}$.

\medskip
{\bf Case d):} This case is analogous to {\bf Case c)}.




\subsubsection{The all random $R\cR R \cR R$ Case}\label{rrrrr}

Since we have removed the frequencies $n_1,n_3=n_2$ or $n_1,n_3=n_4$, which  would give rise to terms of the form $|g_i(\omega)|^2$,  we can  invoke Lemma 
\ref{largedeviation} and proceed 
to estimate the expression in \eqref{expression}  
by further considering    the following two subcases, 
 {\bf Case a):}  $N_0 \sim N_i$ for some $i=1,\dots,5$ and {\bf Case b):}  $N_i\sim N_j$ for $i, j \neq 0$.
 
\smallskip 
 
\noindent 
{\bf Case a):}  Let $N_j$ be the second largest frequency size after $N_i$. We cut the window $N_0$ with cubes $C$ of sidelength $N_j$. 
By Cauchy-Schwarz and Plancherel we estimate 

\begin{equation} \label{r01}
\sum_{m\in \Z, n \in \Z^3} \,\left|   \sum_ {\substack{ n= n_1-n_2 + n_3-n_4+n_5,\\ n_1,n_3,n_5\ne n_2,n_4 
\\ m = |n_1|^2 -|n_2|^2+ |n_3|^2-|n_4|^2+ |n_5|^2}}   \frac{g_{n_1}(\omega)}{|n_1|^{\frac{3}{2}}} \frac{\cg_{n_2}(\omega)}{|n_2|^{\frac{3}{2}}}\frac{g_{n_3}(\omega)}{|n_3|^{\frac{3}{2}}}\frac{\cg_{n_4}(\omega)}{|n_4|^{\frac{3}{2}}}\frac{g_{n_5}(\omega)}{|n_5|^{\frac{3}{2}}}\right|^2 
\end{equation}     By Lemma \ref{largedeviation} we have that, for $\omega$ outside a set of measure $e^{-\frac{1}{\delta^r}}$,
\begin{eqnarray*}\label{rrrrr1}
\eqref{r01}&\lesssim&\delta^{-2\mu r} \sum_{m\in \Z, n \in \Z^3} \,  \sum_{  \substack{ n= n_1-n_2 + n_3-n_4+n_5,\\ n_1,n_3,n_5\ne n_2,n_4, \, \, n_i\in C
\\ m = |n_1|^2 -|n_2|^2+ |n_3|^2-|n_4|^2+ |n_5|^2}}  \frac{1}{|n_1|^3} \frac{1}{|n_2|^3}\frac{1}{|n_3|^3}\frac{1}{|n_4|^3}\frac{1}{|n_5|^3}\\\notag
&\lesssim&
\delta^{-2\mu r}\, |S|\, \prod_{k=1}^5N^{-3}_k\lesssim N_i^{-1}N_j
\end{eqnarray*}  
where 
$$S=\left\{(m, n, n_1,\dots,n_5) \, : \,  \begin{aligned}&n= n_1-n_2 + n_3-n_4+n_5, \, \,  n_i\in C\\ 
&m = |n_1|^2 -|n_2|^2+ |n_3|^2-|n_4|^2+|n_5|^2\end{aligned}\right\}
$$
and $|S|\lesssim N_i^2N_j^4\prod_{k\ne i,j,0}N^3_k$.  Taking square root and normalizing we obtain the bound
$$N_i^{s+\alpha-\frac{3}{2}}N_j^{-\frac{1}{2}+\alpha}\prod_{k\ne i,j,0}N^{-1+\alpha}_k$$
which suffices provided $s+\alpha<\frac{3}{2}$ and $\alpha<\frac{1}{2}$.

\medskip
\noindent {\bf Case b):}  This is like {\bf Case a)}, but now we do not need to cut the support of the window $N_0$ by $N_j$.

\medskip

\subsubsection{The $U^4_\Delta L^2$ Estimates}

\smallskip
Assume that $N_i$ are dyadic numbers and without loss of generality that $N_1\geq N_2\geq \dots \geq N_5$. We start by rewriting 
\begin{eqnarray*}&& \int_0^{2\pi} \int_{\T^3}  D^s \bigl( \mathcal N(P_{N_i}(w+v_0^\omega))   \bigr) \overline{ P_{N_0}h }  \, dx \, dt = \int_0^{2\pi} \int_{\T^3}  D^s \bigl( \mathcal N(P_{N_i}w)   \bigr) \overline{ P_{N_0}h }  \, dx \, dt\\
&+& \int_0^{2\pi} \int_{\T^3}  D^s \bigl( \mathcal N(P_{N_i}v_0^\omega)   \bigr) \overline{ P_{N_0}h }  \, dx \, dt\\
&+& \int_0^{2\pi} \int_{\T^3}  D^s \bigl( \mathcal N(P_{N_i}w,P_{N_i}v_0^\omega)   \bigr) \overline{ P_{N_0}h }  \, dx \, dt,\\
&=&\TT_1+\TT_2+\TT_3,
\end{eqnarray*}
where in the  term $\TT_3$ we include all the nonlinear expressions with both random and deterministic terms. 
Our goal is to obtain an  estimate for the first and last term with the  $U^4_\Delta L^2$ norms of $w$ in the right hand side possibly paying the prize of $N^\gamma_2$, with $\gamma>0$. Then using  the interpolation Proposition \ref{U-interpolation}, we combine this estimate with the ones involving the norms $U^2_\Delta L^2$ in previous sections and the embeddings \eqref{embed3} and \eqref{embed4} to finally conclude the proof of Proposition \ref{main-lemma}.

Clearly we do not need to estimate $\TT_2$ that involves  purely  random terms. For the other two we have 
$$ \TT_1+\TT_3
\lesssim [\|\mathcal N(P_{N_i}w)\|_{L^4_tL^2_x}+\|\mathcal N(P_{N_i}w,P_{N_i}v_0^\omega)\|_{L^4_tL^2_x}]\|P_{N_0}h\|_{L^\infty_tL^2_x},
$$
and from \eqref{rewrite}, we a certain abuse of notation,  
\begin{eqnarray*}&&\|\mathcal N(P_{N_i}w)\|_{L^4_tL^2_x}+\|\mathcal N(P_{N_i}w,P_{N_i}v_0^\omega)\|_{L^4_tL^2_x}\\&\lesssim&  \sum_{i=1}^9\|\FF^{-1}J_i(w)\|_{L^4_tL^2_x}+\sum_{j=1}^9\|\FF^{-1}J_i(P_{N_i}w,P_{N_i}v_0^\omega)\|_{L^4_tL^2_x} \\
&=&\sum_{i=1}^9(\FS_1^i+\FS^i_2)
\end{eqnarray*}
where $J_i(w, v_0^\omega)$ indicates that the functions involved could be both $w$ and $v_0^\omega$. To estimate 
 $\FS_1^i$ and $\FS_2^i$ we use the transfer principle in Proposition \ref{transfer} and we assume that $\hat w(t,n)=e^{it|n|^2}b_n(t)$. Below we write $a^j_{n_j}$ to indicate $b_n$ or the  Fourier coefficients of $v_0^\omega$ or their conjugates.    Now define
$$\Phi_i(n,t):= \left| \sum_{ \substack{ n=\sum_{i=1}^5\pm n_i, \, \, n_i\sim N_i,\\
W_i(n_1, n_2,n_3,n_4,n_5)} }  a^1_{n_1}e^{\pm it|n_1|^2}a^2_{n_2}e^{\pm it|n_2|^2}
a^3_{n_3}e^{\pm it|n_3|^2}
a^4_{n_4}e^{\pm it|n_4|^2}
a^5_{n_5}e^{\pm it|n_5|^2}
\right|^2, $$ 
where $W_i(n_1, n_2, n_3,n_4,n_5)$ indicates the  constraints among the five frequencies in $J_i$.
 Then for $i=1,\dots,9$ and $k=1,2$
\begin{eqnarray*}
(\FS_k^i)^2&\lesssim& \sup_{t\in [0,2\pi]}  \sum_{n}\Phi_i(n,t) 
\lesssim\sup_{t\in [0,2\pi]} \sum_{n}\left| \sum_{ S_{(n)}} |a^1_{n_1}|
|a^2_{n_2}||a^3_{n_3}||a^4_{n_4}||a^5_{n_5}|\right|^2, 
\end{eqnarray*}
 where
$$S_{(n)}=\{ (n_1, n_2, n_3, n_4, n_5) \, \, : \, \, \sum_{j=1}^5\pm n_j, \, \, n_j\sim N_j \}.$$
Assume now that $N_1$, the   highest frequency, is such that  $ N_1\sim N_0$, which is in fact the least favorable situation. Then $|S_{(n)}|\lesssim N_2^3N_3^3N_4^3N_5^3$ and by Cauchy-Schwarz 
\begin{equation}\label{lastest}
\left| \sum_{ S_{(n)}} |a^1_{n_1}|
|a^2_{n_2}||a^3_{n_3}||a^4_{n_4}||a^5_{n_5}|\right|^2\lesssim N_2^3N_3^3N_4^3N_5^3\|a^1_{n_1}\|_{\ell^2}^2\prod_{j=2}^5\|a^j_{n_j}\|_{\ell^2}^2.\end{equation}
We then have 
\begin{equation}\label{u4l2part1}\FS_1^i\lesssim N_2^{6-4s}\|P_{N_1}w\|_{U^4_\Delta H^s}\prod_{j=2}^5\|P_{N_j}w\|_{U^4_\Delta H^s}.\end{equation}
We observe that a similar estimate holds for $\FS_2^i$ when the function associated to frequency $N_1$ is also deterministic. In fact in this case we have 
\begin{equation}\label{u4l2part2}\FS_2^i\lesssim N_2^{2+4\alpha}\|P_{N_1}w\|_{U^4_\Delta H^s}\prod_{j\notin J, \, j\ne 1}\|P_{N_j}w\|_{U^4_\Delta H^s}.\end{equation}
Finally if the function associated to frequency $N_1$ is random, then we have 
\begin{equation}\label{u4l2part3}\FS_2^i\lesssim N_1^{s-1+\alpha}N_2^{2+4\alpha}\prod_{j\notin J}\|P_{N_j}w\|_{U^4_\Delta H^s}.\end{equation}

\medskip
We conclude by using the interpolation Proposition  \ref{U-interpolation}. Note here that in both \eqref{u4l2part1} and \eqref{u4l2part2} the interpolation at most  introduces a factor of $N_2^\varepsilon$  which can be easily absorbed by 
the negative power of $N_2$ in the estimates involving norms $U^2_\Delta L^2$, see previous subsections.  On the other hand \eqref{u4l2part3}
and  interpolation   introduce a factor $N_1^\varepsilon$. But this too can be absorbed thanks to the presence of a negative power of $N_1$ in the estimates involving norms $U^2_\Delta L^2$ in those cases in which the highest frequency is associated to  a random function.

This  concludes the proof of Proposition \ref{main-lemma}.

\section{ Proof of Proposition \ref{main-estimate}}\label{proof-main-estimate} 

\begin{proof} [Proof of Proposition \ref{main-estimate}] We first show an improved version of Proposition \ref{main-lemma}, that is we show that if $r>0$ is small enough then there exists $\theta>0$ such that for $\omega\in \Omega_\delta$ we have: 
if  $N_1\gg N_0$  or $P_{N_1}w=P_{N_1}v_0^\omega$
\begin{eqnarray}\label{outcome1delta}
 &&\left| \int_0^{2\pi} \int_{\T^3}  D^s \bigl(\psi_{\delta}(t) \mathcal N(P_{N_i}(w+v_0^\omega))   \bigr) \overline{ P_{N_0}h }  \, dx \, dt \right| \\\notag
&& \qquad \qquad \lesssim \delta^{\theta} N_1^{-\varepsilon} \| P_{N_0}h\|_{Y^{-s}}\left( 1+ \prod_{i\notin J}\|\psi_{\delta}P_{N_i}w\|_{X^s}\right),
 \end{eqnarray}
 and if $N_1\sim N_0$ and $P_{N_1}w\ne P_{N_1}v_0^\omega$
\begin{eqnarray}\label{outcome2delta} &&\left| \int_0^{2\pi} \int_{\T^3}  D^s \bigl(\psi_{\delta}(t) \mathcal N(P_{N_i}(w+v_0^\omega))   \bigr) \overline{ P_{N_0}h }  \, dx \, dt \right|\\\notag  
& \lesssim&
 \delta^{\theta} N_2^{-\varepsilon} \| P_{N_0}h\|_{Y^{-s}}\|\psi_{\delta}P_{N_1}w\|_{X^s}
\left( 1+ \prod_{i\notin J, i\ne 1}\|\psi_{\delta}P_{N_i}w\|_{X^s}\right),
 \end{eqnarray}
 for some small $\varepsilon>0$.
 
To prove \eqref{outcome1delta} and \eqref{outcome2delta} we first observe that   in the proof of Proposition \ref{main-lemma}, in particular the estimates involving the terms $J_2,\dots, J_7$ in \eqref{rewritedifference}, we always have the factor $\|P_{N_0}h\|_{L^2_tL^2_x}$ in the right hand side. We can then replace this by 
\begin{equation}\label{deltah}\|\psi_{\delta}(t)P_{N_0}h\|_{L^2_tL^2_x}\lesssim \delta^\frac{1}{2}\|\psi_{\delta}(t)P_{N_0}h\|_{L^\infty_tL^2_x}\lesssim \delta^\frac{1}{2}\|\psi_{\delta}(t)P_{N_0}h\|_{Y^0},\end{equation}
where we used \eqref{embed4} and obtain the proof of 
Proposition \ref{main-estimate} for the nonlinear terms involving $J_2,\dots, J_7$.

 To estimate the term involving $J_1$ we go back to Subsections \ref{ddddd}-\ref{rrrrr}.  We recall that except when the top frequencies, say $N_1$ and $N_2$, are associated to two deterministic functions,
also in this case we have $\|P_{N_0}h\|_{L^2_tL^2_x}$ in the right hand side and \eqref{deltah} can be used again. 
 
 We are then reduced to estimating  the term involving $J_1$ where the top frequencies $N_1$ and $N_2$ are associated to two deterministic functions. So we consider 
\begin{equation}\label{L1time}
\left|\int_0^{2\pi}\int_{\T^3}\FF_{-1}J_1(\psi_{\delta}(t)P_{N_i}(u_i) \psi_{\delta}(t)\overline{ P_{N_0}h } \, dx\,dt\right|
\end{equation}
where without loss of generality $N_1\geq N_2\geq N_3\geq N_4\geq N_5$ and $u_1$ and $u_2$ are deterministic functions while $u_{N_i}, \, i=3,4,5$ represents either $w$ or $v_0^\omega$.  We consider two cases:
\begin{itemize}
\item {\bf Case 1:} $\delta^{-\sigma} >N_2$
\item {\bf Case 2:} $\delta^{-\sigma} \leq N_2$
\end{itemize}
where $\sigma>0$ will be determined later.

\medskip

{\bf Case 1:} We observe that   the estimate of  \eqref{L1time} can be reduced to analyzing an expression such as 
\begin{equation}\label{fullproduct}\left|\int_0^\delta\int_{\T^3}\tilde u_{N_1}\tilde u_{N_2}\tilde u_{N_3}\tilde u_{N_4}\tilde u_{N_5}\tilde h_{N_0}\, dx\,dt\right|\end{equation}
where $u_{N_i}$ are as above. In fact to obtain the full product as in \eqref{fullproduct} one needs to put back some frequencies, and hence some terms, see for example \eqref{addons} in Subsection \ref{ddddd}. But these terms are similar  to those  involved in $J_2, \dots, J_7$ and again the gain on $\delta$ is garanteed  by   \eqref{deltah}.

 We then go back to \eqref{fullproduct} and  we further assume that $N_1\sim N_0$, which is the least favorable situation. We cut the frequency window $N_0$, and hence $N_1$, with respect to cubes $C$ of sidelength $N_2$ and we obtain the bound
\begin{eqnarray*}&&\left|\int_0^\delta\int_{\T^3}\tilde u_{N_1}\tilde u_{N_2}\tilde u_{N_3}\tilde u_{N_4}\tilde u_{N_5}\tilde h_{N_0}\, dx\,dt\right|^2
\lesssim \delta \sum_C\|P_{C}\tilde u_{N_1}\|_{L^{12}_tL^{12}_x}^2\|P_{C}\tilde h_{N_0}\|_{L^{12}_tL^{12}_x}^2\\
&\times&\|\tilde u_{N_2}\|_{L^{12}_tL^{12}_x}^2\|\tilde u_{N_3}\|_{L^{12}_tL^{12}_x}^2\|\tilde u_{N_4}\|_{L^{12}_tL^{12}_x}^2\|\tilde u_{N_5}\|_{L^{12}_tL^{12}_x}^2
\end{eqnarray*}
and from \eqref{Strichartz-1}, \eqref{Strichartz-2} and \eqref{Strichartz-3}  we can continue with 
\begin{equation}\label{5bound12}\lesssim \delta N_2^{m(\alpha,s)}\sum_C\|P_{C}u_{N_1}\|_{U^{12}_\Delta L^2}^2
\|P_{C} h_{N_0}\|_{U^{12}_\Delta L^2}^2\prod_{i\notin J, i\ne 1}\| u_{N_i}\|_{U^{12}_\Delta L^2}^2,\end{equation}
where $J\subset \{2,3,4,5\}$ is the set of indices corresponding to random linear solutions.
%
Then normalizing, interpolating through  Proposition \ref{U-interpolation},   and using the embedding \eqref{embed4} combined with  \eqref{sum-cubes}, we have 
\begin{eqnarray*}&&\left|\int_0^\delta\int_{\T^3}\tilde u_{N_1}\tilde u_{N_2}\tilde u_{N_3}\tilde u_{N_4}\tilde u_{N_5}\tilde h_{N_0}\, dx\,dt\right|\\
&\lesssim& \delta^\frac{1}{2} N_2^{m(\alpha,s)}\| P_{N_0}h\|_{Y^{-s}}\|\psi_{\delta}P_{N_1}w\|_{X^s}
\left( 1+ \prod_{i\notin J, i\ne 1}\|\psi_{\delta}P_{N_i}w\|_{X^s}\right)\\&\lesssim& \delta^\frac{1}{3}  \| P_{N_0}h\|_{Y^{-s}}\|\psi_{\delta}P_{N_1}w\|_{X^s}N_2^{-\varepsilon}
\left( 1+ \prod_{i\notin J, i\ne 1}\|\psi_{\delta}P_{N_i}w\|_{X^s}\right),
\end{eqnarray*}
if we take  $\sigma<\frac{1}{100m(\alpha,s)}.$

\medskip
{\bf Case 2:} Here we go back to   \eqref{outcome1} and  \eqref{outcome2}. We recall that $P_{N_1}u_1$ is deterministic and again we assume that $N_1\sim N_0$, the other cases can be treated similarly. Then we use \eqref{outcome1} and we have 
\begin{eqnarray*}
&&\left|\int_0^{2\pi}\int_{\T^3}\FF_{-1}J_1(\psi_{\delta}(t)P_{N_i}(u_i) \psi_{\delta}(t)\overline{ P_{N_0}h } \, dx\,dt\right|\\
&\lesssim&\delta^\gamma \delta ^{-\gamma-\mu r}N_2^{-\rho(\alpha,s)} \| P_{N_0}h\|_{Y^{-s}}\|\psi_{\delta}P_{N_1}w\|_{X^s}
 \prod_{i\notin J, i\ne 1}\|\psi_{\delta}P_{N_i}w\|_{X^s}\\
 &\lesssim& \delta ^\gamma N_2^{-\varepsilon} \| P_{N_0}h\|_{Y^{-s}}\|\psi_{\delta}P_{N_1}w\|_{X^s}
 \left( 1+ \prod_{i\notin J, i\ne 1}\|\psi_{\delta}P_{N_i}w\|_{X^s}\right),
 \end{eqnarray*}
provided $\sigma\geq \frac{\gamma+\mu r}{\rho(\alpha,s)}$ which is satisfied for $\gamma, r$ small enough.

\smallskip

To finish the proof we now need to sum the dyadic blocks just as in \cite{HTT}. In \eqref{outcome1delta} we have enough decay in the highest frequency $N_1$ that we can use Cauchy-Schwarz in all the smaller frequencies  terms and just pay with a $N_1^{-\frac{\varepsilon}{2}}$. In \eqref{outcome2delta}  instead we use Cauchy-Schwarz for the lower frequencies $N_5, \dots N_3$ and  pay with a $N_2^{-\frac{\varepsilon}{2}}$ that can be absorbed and use Cauchy-Schwarz on $N_0\sim N_1$.

\end{proof}


 \end{document}